\documentclass[11pt]{article}

\usepackage{amsmath,amssymb,amsfonts,enumitem,amsthm,accents,pstricks,pst-node,pstricks-add,mathrsfs,xy}
\xyoption{all}

\numberwithin{equation}{section}

\newtheorem*{prn}{Principle}
\newtheorem*{nex}{Nexus}

\newtheorem{thm}{Theorem}[section]
\newtheorem{lmm}[thm]{Lemma}
\newtheorem{prp}[thm]{Proposition}
\newtheorem{crl}[thm]{Corollary}
\theoremstyle{definition}
\newtheorem{dfn}[thm]{Definition}
\newtheorem{eg}[thm]{Example}
\newtheorem{rmk}[thm]{Remark}

\def\BE#1{\begin{equation}\label{#1}}
\def\EE{\end{equation}}
\def\eref#1{(\ref{#1})}

\def\lra{\longrightarrow}

\def\lhra{\ensuremath{\lhook\joinrel\relbar\joinrel\rightarrow}}

\def\ov#1{\overline{#1}}

\def\wt#1{\widetilde{#1}}
\def\sf#1{\textsf{#1}}

\def\sm#1{\begin{small}{#1}\end{small}}

\def\wh#1{\widehat{#1}}

\def\al{\alpha}

\def\de{\delta}
\def\ep{\epsilon}

\def\io{\iota}

\def\na{\nabla}
\def\om{\omega}

\def\ve{\varepsilon}
\def\vr{\varrho}
\def\vph{\varphi}
\def\vp{\varpi}

\def\ze{\zeta}

\def\La{\Lambda}
\def\Om{\Omega}
\def\Th{\Theta}

\def\bI{\mathbb I}

\def\C{\mathbb C}

\def\fD{\mathfrak D}

\def\bI{\mathbb I}
\def\cK{\mathcal K}

\def\cN{\mathcal N}

\def\cP{\mathcal P}
\def\fR{\mathfrak R}
\def\R{\mathbb R}
\def\cR{\mathcal R}

\def\cW{\mathcal W}
\def\X{\mathbf X}
\def\cX{\mathcal X}

\def\Z{\mathbb Z}

\def\fs{\mathfrak s}

\def\ne{\textnormal{e}}
\def\fI{\mathfrak i}

\def\AK{\textnormal{AK}}
\def\Aux{\textnormal{Aux}}

\def\codim{\textnormal{codim}}

\def\Dom{\textnormal{Dom}}

\def\id{\textnormal{id}}
\def\Im{\textnormal{Im}}

\def\nd{\textnormal{d}}
\def\hor{\textnormal{hor}}
\def\pt{\textnormal{pt}}

\def\supp{\textnormal{supp}}
\def\Symp{\textnormal{Symp}}

\def\std{\textnormal{std}}

\def\ver{\textnormal{ver}}

\def\i{\infty}
\def\w{\wedge}

\def\eset{\emptyset}
\def\prt{\partial}
\def\1{\mathbf 1}

\def\bu{\bullet}

\begin{document}

\title{Normal Crossings Singularities for Symplectic Topology}
\author{Mohammad F.~Tehrani, Mark McLean, and 
Aleksey Zinger\thanks{Partially supported by NSF grants 0846978 and 1500875}}

\date{\today}

\maketitle

\begin{abstract}
\noindent
We introduce topological notions of  normal crossings symplectic divisor
and variety and establish that they are equivalent, in a suitable sense,
to the desired geometric notions.
Our proposed concept of equivalence of associated topological and geometric notions 
fits ideally with important constructions in symplectic topology.
This partially answers Gromov's question on the feasibility of defining 
singular symplectic (sub)varieties
and lays foundation for rich developments in the future.
In subsequent papers, 
we establish a smoothability criterion for symplectic normal crossings varieties,
in the process providing the multifold symplectic sum  envisioned 
by Gromov, and introduce symplectic analogues of logarithmic structures 
in the context of normal crossings symplectic divisors.
\end{abstract}

\tableofcontents

\section{Introduction}
\label{intro_sec}

\noindent
Algebraic and complex analytic (sub)varieties are the central objects of study
in the fields of algebraic geometry and of complex geometry, respectively.
Curves and divisors, i.e.~subvarieties of dimension and codimension~1 over the ground field,
have long been of particular importance in these fields;
they can be viewed as dual to each other.
Gromov's introduction~\cite{Gr} of pseudoholomorphic curve techniques into symplectic topology 
has led to numerous connections with algebraic geometry and 
to the appearance of symplectic divisors in different contexts,
including relations with complex line bundles~\cite{Donaldson},
symplectic sum constructions \cite{Gf,MW}, 
degeneration and decomposition formulas for Gromov-Witten invariants
\cite{Tian,CH,LR,Jun2,Brett}, affine symplectic geometry \cite{MAff,McLean},
and homological mirror symmetry~\cite{Sheridan}.
While most applications of divisors in symplectic topology have so far concerned
{\it smooth} divisors in ({\it smooth}) symplectic manifolds, 
recent developments in symplectic topology and algebraic geometry suggest the need for notions 
of {\it singular} symplectic varieties and subvarieties 
(at least with certain kinds of singularities).
Gromov \cite[p343]{GrBook} in fact asked about the feasibility of 
introducing such notions by the mid~1980s.
They should involve only some soft intrinsic symplectic data, but at the same time
be compatible with rigid auxiliary almost K\"ahler data needed for 
making such notions useful.\\

\noindent
In this paper, we propose a new perspective on the fundamental quandary conceived 
in~\cite{GrBook} and demonstrate that it is suitable for introducing \sf{normal crossings} 
singularities into symplectic topology.
In symplectic topology, it is common to study a ``topological'' object 
(such as a symplectic manifold) by adding some auxiliary ``geometric'' 
data (such as a compatible complex structure) and then constructing 
an invariant which is independent of this auxiliary data. 
One then shows that such an invariant is also an invariant of 
the deformation equivalence class of the corresponding topological object.  
This approach works well when studying symplectic manifolds, but 
is much more difficult to carry out once singularities are introduced
since there is no Darboux theorem in this case; 
see in particular Remark~\ref{SympNeigh_rmk}.
We propose an alternative philosophy involving 
the entire deformation equivalence class of the topological object, 
as opposed to the topological object itself, and 
looking at the subspace of ``nice" objects in this deformation equivalence class.
These objects should have particularly well-behaved auxiliary geometric data 
which can be used to construct invariants in usual ways.
The subspace of ``nice" objects should topologically reflect the space
of all objects.
Our perspective is summarized by the principle below
and a more detailed nexus on page~\pageref{equiv_nex}.

\begin{prn}\label{equiv_prn}
A symplectic variety/subvariety should be viewed as a deformation equivalence class
of objects with the same topology, not as a single object.
\end{prn}

\noindent
The viewpoint on the compatibility between the topological and geometric sides 
we propose in the nexus is symmetric in taking deformation equivalence classes on both sides,
in contrast to the presently standard viewpoint of taking individual objects
on the topological side and deformation equivalence classes on the geometric side.
Our focus on the deformation equivalence classes
to begin with fits ideally with the concern of symplectic topology 
with  the deformation equivalence classes of symplectic manifolds,
instead of individual symplectic manifolds, as illustrated below.

\subsection{Topological vs.~geometric symplectic data}
\label{TopGeom_subs}

\noindent
Every symplectic manifold $(X,\om)$ admits a tame (and compatible) 
almost complex structure~$J$.
Furthermore, the fibers of the~projection
\BE{AK2Symp_e0} \AK(X)\lra\Symp(X), \qquad (\om,J)\lra \om,\EE
from  the space of pairs $(\om,J)$ consisting of a symplectic form~$\om$
on~$X$ and an $\om$-tame almost complex structure~$J$ to 
the space of symplectic forms on~$X$ are contractible.
This fibration is thus a weak homotopy equivalence, i.e.~it induces isomorphisms
\BE{AK2Symp_e0b} \pi_k\big(\AK(X)\big)\lra\pi_k\big(\Symp(X)\big) \EE
between the homotopy groups $\pi_k$ with $k\!\in\!\Z^+$ and 
the sets~$\pi_0$ of deformation equivalence classes;
see \cite[Theorem~6.3]{Dold}.
The bijectivity of~\eref{AK2Symp_e0b} for $k\!=\!0$ is key to the program initiated by Gromov
in the 1980s  to bring algebro-geometric techniques into symplectic topology.\\

\noindent
For a symplectic submanifold~$V$ in a symplectic manifold $(X,\om)$,
the normal bundle 
\BE{cNXVsymp_e}\cN_XV\equiv \frac{TX|_V}{TV}\approx TV^{\om}
\equiv \big\{v\!\in\!T_xX\!:\,x\!\in\!V,\,\om(v,w)\!=\!0~\forall\,w\!\in\!T_xV\big\}\EE
of~$V$ in~$X$ inherits a fiberwise symplectic form $\Om\!\equiv\!\om|_{\cN_XV}$ from~$\om$.
A \sf{smooth divisor} in a symplectic manifold $(X,\om)$ is 
a closed symplectic submanifold~$V$ of real codimension~2.
For every such triple $(X,V,\om)$, there is an $\om$-tame almost complex structure~$J$
such that $J(TV)\!=\!TV$.
It can be chosen to be very regular near~$V$ in the following sense.
An  $\Om$-compatible (fiberwise) complex structure~$\fI$ on~$\cN_XV$ 
and a compatible connection~$\na$ on~$\cN_XV$
determine a closed 2-form~$\wh\om$ on the total space of~$\cN_XV$, 
which is symplectic on a neighborhood of~$V$ in~$\cN_XV$.
By the Symplectic Neighborhood Theorem \cite[Theorem~3.30]{MS1},
there is an identification~$\Psi$ of small neighborhoods of~$V$ in $(\cN_XV,\wh\om)$
and in~$(X,\om)$.
The tuple $(\fI,\na,\Psi)$ is equivalent to an $\om$-regularization $(\rho,\na,\Psi)$ for~$V$
in~$X$ in the terminology of Definition~\ref{SCDregul_dfn}\ref{sympregul_it};
we view it as an \sf{auxiliary} structure for~$(X,V,\om)$.
We call an $\om$-tame almost complex structure~$J$ \sf{compatible} with $(\rho,\na,\Psi)$
if $\Psi^*J$ agrees with the almost complex structure~$\wh{J}$ 
determined by $J|_V$, $\fI$, and~$\na$;
such a~$J$ is integrable in the normal direction to~$V$
(i.e.~the image of its Nijenhuis tensor on~$TX|_V$ is contained in~$TV$). 
The fibers of the~projection
\BE{AK2Aux_e1} \AK(X,V)\lra\Aux(X,V), \qquad (\om,\cR,J)\lra (\om,\cR),\EE
from  the space of triples $(\om,\cR,J)$ consisting of a symplectic form~$\om$ on~$(X,V)$,  
an $\om$-regularization~$\cR$ for~$V$ in~$X$, and an $\cR$-compatible almost complex structure~$J$
to the space of pairs   consisting of a symplectic form~$\om$ on~$(X,V)$ and  
an $\om$-regularization for~$V$ in~$X$ are contractible.
This fibration thus  induces isomorphisms
$$ \pi_k\big(\AK(X,V)\big)\lra\pi_k\big(\Aux(X,V)\big) $$
between the homotopy groups $\pi_k$ with $k\!\in\!\Z^+$ and 
the sets~$\pi_0$ of deformation equivalence classes.\\

\noindent
For a closed codimension~2 submanifold $V$ of~$X$, the fibers of the~projection
\BE{Aux2Symp_e1} \Aux(X,V)\lra\Symp(X,V), \qquad (\om,\cR)\lra\om,\EE
to the space of symplectic forms on~$X$ restricting to symplectic forms on~$V$
are also contractible.
The composition of~\eref{AK2Aux_e1} and~\eref{Aux2Symp_e1},
\BE{AK2Symp_e1} \AK(X,V)\lra\Symp(X,V), \qquad (\om,\cR,J)\lra \om,\EE
thus has contractible fibers and induces isomorphisms
\BE{AK2Symp_e1b} \pi_k\big(\AK(X,V)\big)\lra\pi_k\big(\Symp(X,V)\big) \EE
between the homotopy groups $\pi_k$ with $k\!\in\!\Z^+$ and 
the sets~$\pi_0$ of deformation equivalence classes.
The former in particular implies that the map~\eref{AK2Symp_e1} is surjective
and ensures that paths in the base can be lifted to paths with specified endpoints.
While these two properties of~\eref{AK2Symp_e1} feature prominently 
in the standard perspective on applications of symplectic divisors,
only the bijectivity of the map~\eref{AK2Symp_e1b} with $k\!=\!0$ is material
for applications in symplectic topology.
This is consistent with symplectic topology  being fundamentally about
deformation equivalence classes of symplectic manifolds, 
not about individual manifolds, as illustrated by the well-known applications
recalled in the next two paragraphs.\\

\noindent
The approach to relative Gromov-Witten invariants for~$(X,V,\om)$ in~\cite{LR}
involves choosing an $\om$-regularization $\cR$ for~$V$ in~$X$
and an $\cR$-compatible almost complex structure~$J$. 
The resulting numbers do not change along a path $(\om_t,\cR_t,J_t)$
in $\AK(X,V)$.
Since a path~$\om_t$ in $\Symp(X,V)$ can be lifted to a path $(\om_t,\cR_t,J_t)$ 
with specified endpoints $(\om_0,\cR_0,J_0)$ and $(\om_1,\cR_1,J_1)$,
the relative Gromov-Witten invariants of~$(X,V,\om)$
depend only on the path-component of~$\Symp(X,V)$ containing~$\om$.
It would thus have been sufficient to show~that
\begin{enumerate}[label=$\bu$,leftmargin=*]

\item $(X,V)$ admits $\om$-regularizations $\cR$  for
{\it at least  some} symplectic forms~$\om$ on~$(X,V)$,

\item every path~$\om_t$ in the subspace of such ``good" symplectic forms
lifts to a path~$\cR_t$ of $\om_t$-regularizations 
for~$V$ in~$X$ with given endpoints,

\item the inclusion of the subspace of  ``good" symplectic forms into
the space of all symplectic forms on~$(X,V)$ induces a bijection between
the corresponding sets of path components.  

\end{enumerate}
This change in perspective turns out to be useful when dealing with
NC symplectic divisors.\\

\noindent
The symplectic sum construction of~\cite{Gf} smooths the union of two symplectic manifolds 
$(X,\om_X)$ and $(Y,\om_Y)$ glued along a common smooth symplectic divisor~$V$
such~that
\BE{Gfcond_e}  c_1\big(\cN_XV,\om_X\big)+c_1\big(\cN_YV,\om_Y\big)=0\in H^2(V;\Z)\EE
into a new symplectic manifold $(X_{\#},\om_{\#})$.
In the terminology of Definition~\ref{SCC_dfn}, the~tuples
\BE{Gfconf_e}\big((X_1\!\equiv\!X,X_2\!\equiv\!Y,X_{12}\!\equiv\!V),
(\om_1\!\equiv\!\om_X,\om_2\!=\!\om_Y)\big)
\quad\hbox{and}\quad
\big(X\!\cup_V\!Y,(\om_X,\om_Y)\big)\EE
are a \sf{2-fold simple crossings symplectic configuration} and 
the \sf{associated simple crossings symplectic variety}.
The topological type of~$X_{\#}$ depends only on the choice of the homotopy class 
of isomorphisms
$$\big(\cN_XV,\om_X\big)\otimes_{\C}\big(\cN_YV,\om_Y\big)\approx V\!\times\!\C$$
as complex line bundles.
With such a choice fixed, the construction of~\cite{Gf} involves choosing 
an $\om_X$-regularization for~$V$ in~$X$, an $\om_Y$-regularization for~$V$ in~$Y$,
and a representative for the above homotopy class.
Because of these choices, the resulting symplectic manifold $(X_{\#},\om_{\#})$
is determined by $(X,\om_X)$, $(Y,\om_Y)$, and the choice of the homotopy class
only up to symplectic deformation equivalence.
The symplectic sum construction of~\cite{Gf} can thus be viewed as a~map
\begin{equation*}\begin{split}
\Big\{&\!\big([\om_X],[\om_Y]\big)\!\in\!\pi_0\big(\Symp(X,V)\big)
\!\times\!\pi_0\big(\Symp(Y,V)\big)\!:\\
&\big[\om_X|_V\big]\!=\!\big[\om_Y|_V\big]\!\in\!\pi_0\big(\Symp(V)\big),
c_1\big(\cN_XV,\om_X\big)\!+\!c_1\big(\cN_YV,\om_Y\big)\!=\!0\Big\}
\lra\bigsqcup_{X_{\#}}\pi_0\big(\Symp(X_{\#})\big).
\end{split}\end{equation*}
It would have been sufficient to carry it out only on a path-connected set 
of representatives for 
each deformation equivalence class of the tuples~\eref{Gfconf_e}.
This change in perspective turns out to be useful for smoothing out
more elaborate simple and normal crossings symplectic varieties.\\

\noindent
The above observations concerning~\eref{AK2Symp_e1} and~\eref{AK2Symp_e1b} motivate
the principle introduced in the present paper for adapting algebro-geometric notions 
of singularities  to symplectic topology. 
It can be summarized as follows.

\begin{nex}\label{equiv_nex}
\begin{enumerate}[label=(\arabic*),leftmargin=*]

\item A \sf{symplectic variety} should be a stratified space~$X$ with some additional 
smooth-type structure and symplectic-type structure~$\om$
so that the restriction of~$\om$ to each smooth stratum~$X_i$ of~$X$ is a symplectic form in 
the usual sense.
The set $\Symp(X)$ of the ``symplectic structures" on~$X$ compatible with 
a given ``smooth structure" should have a  natural topology.

\item There should be a notion of a
\sf{regularization}~$\fR$ for a symplectic structure~$\om$ on~$X$
which models neighborhoods of the strata~$X_i$ on subspaces of complex vector bundles~$\cN_i$
over~$X_i$ consisting of fiberwise subvarieties in a compatible fashion.
The set $\Aux(X)$ of such pairs $(\om,\fR)$ should have a natural topology so that
the projection 
\BE{Aux2Symp_e} \Aux(X)\lra\Symp(X), \qquad (\om,\fR)\lra\om,\EE
induces a bijection between the connected components of the two spaces
(or better yet, is a weak homotopy equivalence).
However, this projection need not be surjective.

\item There should be a notion of an \sf{almost complex structure}~$J$ on~$X$ 
compatible with an $\om$-regularization $\fR$ which restricts 
on each~$X_i$  to an almost complex structure in the usual sense and 
is integrable in the normal directions to~$X_i$.
The set $\AK(X)$ of such triples $(\om,\fR,J)$ should have a natural topology so that
the fibers of the projection 
\BE{AK2Aux_e} \AK(X)\lra\Aux(X), \qquad (\om,\fR,J)\lra(\om,\fR),\EE
are contractible.

\end{enumerate}
\noindent
A \sf{symplectic subvariety} $V$ in a symplectic variety~$X$ should be a topological subspace
of~$X$ with associated spaces $\Symp(X,V)$, $\Aux(X,V)$, and~$\AK(X,V)$
which are related as~above.
\end{nex}

\subsection{Normal crossings singularities}
\label{NCsing_subs}

\noindent
A \sf{normal crossings} (or \sf{NC}) \sf{complex analytic variety}~$X$ of (complex) dimension~$n$
is a Hausdorff topological space covered by~charts
$$\vph_x\!:U_x\lra \C^{n-k}\!\times\!\big\{(z_1,\ldots,z_{k+1})\!\in\!\C^{k+1}\!:
\,z_1\!\ldots\!z_{k+1}\!=\!0\big\}
\quad\hbox{with}\quad k\!=\!k(x)\in\!\{0,1,\ldots,n\},\,x\!\in\!X,$$
that overlap analytically.
An \sf{NC divisor}~$V$ in a K\"ahler manifold~$X$ of complex dimension~$n$
 is a subspace of~$X$ locally of the form
$$\C^{n-k}\!\times\!\big\{(z_1,\ldots,z_k)\!\in\!\C^k\!:
\,z_1\!\ldots\!z_k\!=\!0\big\}
\quad\hbox{with}\quad k\!=\!k(x)\in\!\{0,1,\ldots,n\},\,x\!\in\!X,$$
in holomorphic coordinates on~$X$.
Such a divisor is the image of a generically injective proper K\"ahler 
immersion $\io\!:\wt{V}\!\lra\!X$ from a K\"ahler manifold~$\wt{V}$ 
of complex dimension $n\!-\!1$ such that all self-intersections of~$\io$ 
are transverse.
A basic example, which we call a \sf{simple crossings} (or \sf{SC}) \sf{divisor}, 
is provided by the union of transversely intersecting closed K\"ahler hypersurfaces~$V_i$ in~$X$.
NC~singularities are the simplest, non-trivial singularities and 
are also of the most direct relevance to symplectic topology.\\

\noindent
It has long been a mystery what an NC or even SC divisor in the symplectic category should~be.
In~this paper, we introduce soft topological notions of an \sf{SC symplectic divisor}
in a symplectic manifold and of an \sf{SC symplectic variety} and show that they are compatible, 
in a suitable sense, with associated rigid geometric notions.
As all of our arguments are essentially local, they readily apply
to the arbitrary NC case as well.
However, the latter involves a more elaborate setup, with 
the normal bundle of an immersion replacing the normal bundle of a submanifold.
For this reason, we defer the arbitrary NC case to~\cite{SympNC}
in order to highlight the ideas involved.\\

\noindent
For a subspace~$V$ of a symplectic manifold~$(X,\om)$
to be an SC symplectic divisor, 
it should at least be the union of transversely intersecting closed symplectic 
submanifolds $\{V_i\}_{i\in S}$ of $(X,\om)$ of real codimension~2.
However, as \cite[Example~1.9]{Inc} illustrates, 
the intersection number of a pair of symplectic submanifolds~$V_1$ and~$V_2$ 
in a compact symplectic 4-manifold~$X$ can be negative;
in such a case, there is no $\om$-tame almost complex structure~$J$ on~$X$ which restricts
to almost complex structures on~$V_1$ and~$V_2$.
If $J$ is an $\om$-tame almost complex structure on~$X$ which restricts  
to an almost complex structure on each~$V_i$, then the intersection~$V_I$ of 
the smooth divisors in any subcollection of~$\{V_i\}_{i\in S}$ 
is a symplectic submanifold of~$(X,\om)$ and 
the $\om$-orientation of each~$V_I$ agrees with its intersection orientation 
induced by the orientations of~$X$ and $\{V_i\}_{i\in I}$; 
see Section~\ref{SCdfn_subs}.
These two properties, appearing in Definition~\ref{SCD_dfn}, are thus necessary
for the existence of an $\om$-tame $J$ which restricts to an almost complex structure
on each~$V_i$.
It turns out that these two, essentially topological, properties
suffice for a kind of virtual existence of such a~$J$ as well as
of compatible collections of $\om$-regularizations $(\rho_i,\na^{(i)},\Psi_i)$
for $V_i$ in~$X$; see Definition~\ref{SCDregul_dfn}\ref{sympregul_it}
and Theorem~\ref{SCD_thm}.\\
 
\noindent
The compatibility-of-orientations condition of Definition~\ref{SCD_dfn},
which is equivalent to the \sf{positively intersecting} notion 
of \cite[Definition~5.1]{MAff}, is 
preserved under deformations of~$\om$ that keep every intersection~$V_I$ symplectic.
Thus, it is a necessary condition for the existence of 
an $\om'$-tame almost complex structure~$J$ that restricts to an almost complex structure 
on each~$V_I$ for some deformation~$\om'$ of~$\om$ through symplectic structures~$\om_t$ 
on $\{V_I\}_{I\subset S}$
(i.e.~symplectic forms~$\om_t$ on~$X$ such that $\om_t|_{V_I}$ is symplectic for all 
$I\!\subset\!S$). 
By Theorem~\ref{SCD_thm} with $B$ being the point, this condition suffices 
not only for the existence of such an $\om'$-tame~$J$, but also 
for the existence of compatible collections of $\om$-regularizations 
$(\rho_i,\na^{(i)},\Psi_i)$ for $V_i$ in~$X$.
By Theorem~\ref{SCD_thm} with $B\!=\![0,1]$, for every path
$\om_t$ of symplectic structures on~$\{V_I\}_{I\subset S}$
and all $\om_0$- and $\om_1$-regularizations $\cR_0$ and~$\cR_1$ for $V_i$ in~$X$,
there exists a path~$\om_t'$ homotopic to the path~$\om_t$ 
through paths of symplectic structures on~$\{V_I\}_{I\subset S}$
and a path~$\cR_t$ of compatible $\om_t'$-regularizations for~$V_i$
in~$X$.
By the general case of Theorem~\ref{SCD_thm}, the projection
\BE{AuxSympDiv_e}
\Aux\big(X,(V_i)_{i\in S}\big)\lra\Symp^+\big(X,(V_i)_{i\in S}\big), \qquad 
(\om,\cR)\lra\om,\EE
from the space of symplectic forms~$\om$ on~$X$
with compatible  regularizations~$\cR$ for $(V_i)_{i\in S}$ in~$X$ 
to the space  of symplectic structures~$\om$
on $\{V_I\}_{I\subset S}$ such that the $\om$-orientation of each~$V_I$   
agrees with its intersection orientation
is a weak homotopy equivalence.
Theorem~\ref{SCC_thm} is the analogue of Theorem~\ref{SCD_thm} for SC~symplectic varieties 
as in Definition~\ref{SCC_dfn}.
These are collections of symplectic manifolds identified along SC~symplectic divisors.
Some applications of these four theorems are described in the next two paragraphs.\\

\noindent
Two versions of an NC~divisor~$V$ in an almost K\"ahler manifold~$X$
are described in \cite[Definition~1.3]{Inc} and \cite[Section~2]{Brett07};
see also \cite[Definition~14.6]{Brett2}.
The main objective of~\cite{Inc} and one of the two main objectives of \cite{Brett1,Brett}
are to define Gromov-Witten type invariants of~$X$ relative to such~$V$.
The constructions in~\cite{Brett,Inc} automatically imply that the resulting invariants 
do not change under deformations of the {\it almost K\"ahler} data compatible with~$(X,V)$.
In \cite[Section~3]{Brett07}, it is shown that the relevant almost K\"ahler data exists 
for a certain, fairly rigid, class of symplectic forms on~$X$ (for which the branches of~$V$ 
are symplectic and meet orthogonally) and that deformations of the symplectic form within
this class extend to deformations of compatible almost K\"ahler data.
However, it would be desirable to know that the resulting invariants depend only on 
some topological deformation equivalence class of {\it symplectic} structures on~$(X,V)$ and
apply to all classes that satisfy a specific simple condition.
An $\om$-regularization for an NC~symplectic divisor~$V$ in~$(X,\om)$ 
can be used to construct an almost K\"ahler structure on~$X$ so that $V$ 
becomes an~NC almost K\"ahler divisor in the sense of 
 \cite[Definition~1.3]{Inc} and \cite[Definition~14.6]{Brett2}.
By Theorem~\ref{SCD_thm}, every deformation equivalence class
of SC symplectic divisors in the sense of 
Definition~\ref{SCD_dfn} contains a representative~$\om$
admitting a regularization and any two such representatives with compatible regularizations
can be joined by a path. 
Thus,  Theorem~\ref{SCD_thm}  implies that any  invariants arising 
from~\cite{Brett,Inc} depend only on the deformation 
equivalence class of symplectic structures on~$(X,V)$ and specifies
to which classes the constructions of~\cite{Brett,Inc}  can be~applied.\\  

\noindent
Theorem~\ref{SCC_thm} is used in~\cite{SympSumMulti}
to show an NC symplectic variety $(X_{\eset},(\om_i)_{i=1,\ldots,N})$
is the central fiber of a one-parameter family of degenerations with a smooth total space  
if and only if it satisfies a simple topological condition on the Chern class
of a complex line bundle over the \sf{singular locus}~$X_{\prt}$ of~$X_{\eset}$.
In the $N\!=\!2$ case, this condition reduces to~\eref{Gfcond_e} and 
every non-central fiber of the resulting family is 
a representative of the deformation equivalence class of the associated symplectic~sum 
$(X_{\#},\om_{\#})$  of~\cite{Gf}.
In general, a non-central fiber of such a family is 
a representative of the deformation equivalence class of 
the multifold symplectic construction on  $(X_{\eset},(\om_i)_{i=1,\ldots,N})$ envisioned
in~\cite[p343]{GrBook}.
It yields a multitude of new symplectic manifolds;
some of them contain closed non-orientable hypersurfaces. 
Going in the opposite direction, the symplectic cut/degeneration construction 
of~\cite{SympCutMulti} produces one-parameter families as above out
of symplectic manifolds with compatible Hamiltonian torus actions 
on open subsets.
The second main objective of \cite{Brett1,Brett} is to obtain decomposition formulas
for Gromov-Witten invariants under certain almost K\"ahler splittings.
An important consequence of Theorem~\ref{SCC_thm}  is that
the decomposition formulas arising from~\cite{Brett} include 
splitting formulas for the Gromov-Witten invariants of 
the $N$-fold symplectic sums constructed in~\cite{SympSumMulti}.\\

\noindent
While the present paper connects directly with deep questions raised by Gromov~\cite{GrBook}
over 3 decades ago, 
the immediate inspiration for our overall project comes from the Gross-Siebert 
program~\cite{GS0} for a direct proof of mirror symmetry and from
distinct recent developments in symplectic topology.
Theorems~\ref{SCD_thm} and~\ref{SCC_thm}, along with the deformation principle behind them,
lay the foundation for symplectic topology versions of logarithmic structures
of algebraic geometry and of stable logarithmic maps of \cite{GS,QChen,AC}
that are central to the Gross-Siebert program.
The almost K\"ahler and exploded manifold versions of these objects proposed in
\cite{Inc,Brett1} are more rigid than desirable and have so far proved too unwieldy 
for practical applications.
We expect Theorem~\ref{SCD_thm} to be also useful for studying smooth affine varieties
and isolated singularities from a symplectic perspective.
For instance, an affine variety can be embedded into a smooth projective variety 
so that its complement is an  NC~divisor; see \cite[Section~2.1]{MAff}.
Theorem~\ref{SCD_thm} describes what a neighborhood of this divisor looks like and 
hence what the affine variety looks like at infinity;
this is useful for analyzing the symplectic cohomology of such varieties.
In contrast to \cite[Theorem 5.14]{MAff},  Theorem~\ref{SCD_thm} describes
such neighborhoods for families of affine varieties.
Links of isolated singularities or
families of isolated singularities (viewed as contact manifolds)
can also be described explicitly by looking at neighborhoods
of the exceptional curves of some resolution, 
using Theorem~\ref{SCD_thm} to put such neighborhoods in a standard form,
and then applying techniques from~\cite{McLean}.

\subsection{Outline and acknowledgments}
\label{outline_subs}

\noindent
We formally define SC symplectic divisors and varieties 
in Section~\ref{SCdfn_subs},
regularizations for the former in Section~\ref{SCDregul_subs},
and regularizations for the latter in Section~\ref{SCCregul_subs}.
While the precise definitions of regularizations are a bit technical,
their substance is that a neighborhood of each point in the divisor or variety looks
as expected.
In particular, the branches of the divisor symplectically correspond to hyperplane subbundles 
of a split complex vector bundle; this implies that they are symplectically orthogonal.
Sections~\ref{SCDregul_subs} and~\ref{SCCregul_subs} conclude with theorems stating
that the spaces of symplectic forms with regularizations are weakly homotopy equivalent
to the spaces of all admissible symplectic forms.
The necessary deformation arguments on split vector bundles are carried out in
Section~\ref{LocalDeform_sec}, especially in the proof of
Proposition~\ref{SympDef_prp}.
Section~\ref{TubulNeigh_sec} contains stratified versions of 
the usual smooth Tubular Neighborhood Theorem.
We prove Theorems~\ref{SCD_thm} and~\ref{SCC_thm} in Section~\ref{SCCpf_sec}
by applying Theorem~\ref{SympDefVB_thm} via Proposition~\ref{TubulNeigh_prp};
the crucial compatibility-of-orientations condition in Definition~\ref{SCD_dfn}
allows us to apply Proposition~\ref{SympDef_prp}.\\

\noindent
We would like to thank E.~Ionel and B.~Parker for enlightening discussions
related to normal crossings divisors in the symplectic category
and E.~Lerman for pointing out related literature.

\section{Simple crossings divisors and varieties}
\label{SC_sec}

\noindent
We begin by introducing the most commonly used notation.
If $N\!\in\!\Z^{\ge0}$ and $I\!\subset\!\{1,\ldots,N\}$, let 
$$[N]=\{1,\ldots,N\}, \qquad
\C_I^N=\big\{(z_1,\ldots,z_N)\!\in\!\C^N\!:\,z_i\!=\!0~\forall\,i\!\in\!I\big\}.$$
Denote by $\cP(N)$ the collection of subsets of~$[N]$ and
by $\cP^*(N)\!\subset\!\cP(N)$ the collection of nonempty subsets.
If in addition $i\!\in\![N]$, let 
$$\cP_i(N)=\big\{I\!\in\!\cP(N)\!:\,i\!\in\!I\big\}.$$
If $\cN\!\lra\!V$ is a vector bundle, $\cN'\!\subset\!\cN$, and $V'\!\subset\!V$, we define
\BE{cNrestrdfn_e} \cN'|_{V'}=\cN|_{V'}\cap\cN'\,.\EE
Let $\bI\!=\![0,1]$.

\subsection{Definitions and examples}
\label{SCdfn_subs}

\noindent
Let $X$ be a (smooth) manifold. 
For any submanifold $V\!\subset\!X$, let
$$\cN_XV\equiv \frac{TX|_V}{TV}\lra V$$
denote the normal bundle of~$V$ in~$X$.
For a collection $\{V_i\}_{i\in S}$ of submanifolds of~$X$ and $I\!\subset\!S$, let
$$V_I\equiv \bigcap_{i\in I}\!V_i\subset X\,.$$
Such a collection 
is called \sf{transverse} if any subcollection $\{V_i\}_{i\in I}$ of these submanifolds
intersects transversely, i.e.~the homomorphism
\BE{TransVerHom_e}
T_xX\oplus\bigoplus_{i\in I}T_xV_i\lra \bigoplus_{i\in I}T_xX, \qquad
\big(v,(v_i)_{i\in I}\big)\lra (v\!+\!v_i)_{i\in I}\,,\EE
is surjective for all $x\!\in\!V_I$. 
By the Inverse Function Theorem \cite[Theorem~1.30]{Warner},
each subspace $V_I\!\subset\!X$ is then a submanifold of~$X$ 
of codimension
$$\codim_XV_I=\sum_{i\in I}\codim_XV_i$$
and the homomorphisms 
\BE{cNorient_e2}\begin{split}
\cN_XV_I\lra \bigoplus_{i\in I}\cN_XV_i\big|_{V_I}\quad&\forall~I\!\subset\!S,\qquad
\cN_{V_{I-i}}V_I\lra \cN_XV_i\big|_{V_I} \quad\forall~i\!\in\!I\!\subset\!S,\\
&\bigoplus_{i\in I-I'}\!\!\cN_{V_{I-i}}V_I\lra \cN_{V_{I'}}V_I \quad\forall~I'\!\subset\!I\!\subset\!S
\end{split}\EE
induced by inclusions of the tangent bundles are isomorphisms.\\

\noindent
Let $X$ be an oriented manifold.
If $V\!\subset\!X$ is an oriented submanifold of even codimension,
the short exact sequence of vector bundles
\BE{cNorient_e1} 0\lra TV\lra TX|_V\lra \cN_XV\lra 0\EE
over $V$ induces an orientation on~$\cN_XV$
(if the codimension and dimension of~$V$ are odd, 
the induced orientation on~$\cN_XV$ depends also on a sign convention). 
If $\{V_i\}_{i\in S}$ is a transverse collection of oriented submanifolds of~$X$
of even codimensions,
the orientations on~$\cN_XV_i$ induced by the orientations of~$X$ and~$V_i$ induce 
an orientation on~$\cN_XV_I$ via the first isomorphism in~\eref{cNorient_e2}.
The orientations of~$X$ and~$\cN_XV_I$ then induce an orientation on~$V_I$
via the short exact sequence~\eref{cNorient_e1}.
Thus, a transverse collection $\{V_i\}_{i\in S}$ of oriented submanifolds of~$X$
of even codimensions induces an orientation on each submanifold $V_I\!\subset\!X$
with $|I|\!\ge\!2$, which we  call \sf{the intersection orientation of~$V_I$}.
If $V_I$ is zero-dimensional, it is a discrete collection of points in~$X$
and the homomorphism~\eref{TransVerHom_e} is an isomorphism at each point $x\!\in\!V_I$;
the intersection orientation of~$V_I$ at $x\!\in\!V_I$
then corresponds to a plus or minus sign, depending on whether this isomorphism
is orientation-preserving or orientation-reversing.
For convenience, we  call the original orientations of 
$X\!=\!V_{\eset}$ and $V_i\!=\!V_{\{i\}}$ \sf{the intersection orientations}
of these submanifolds~$V_I$ of~$X$ with $|I|\!<\!2$.\\

\noindent
Suppose $(X,\om)$ is a symplectic manifold and $\{V_i\}_{i\in S}$ is a transverse collection 
of submanifolds of~$X$ such that each $V_I$ is a symplectic submanifold of~$(X,\om)$.
Each $V_I$ then carries an orientation induced by $\om|_{V_{I}}$,
which we call the \sf{$\om$-orientation}.
If $V_I$ is zero-dimensional, it is automatically a symplectic submanifold of~$(X,\om)$;
the $\om$-orientation of~$V_I$ at each point $x\!\in\!V_I$ corresponds to the plus sign 
by definition.
By the previous paragraph, the $\om$-orientations of~$X$ and~$V_i$ with $i\!\in\!I$
also induce intersection orientations on all~$V_I$.

\begin{dfn}\label{SCD_dfn}
Let $(X,\om)$ be a symplectic manifold.
A \sf{simple crossings} (or \sf{SC}) \sf{symplectic divisor} 
in~$(X,\om)$ is a finite transverse collection 
$\{V_i\}_{i\in S}$ of closed submanifolds of~$X$ of codimension~2 such that 
$V_I$ is a symplectic submanifold of~$(X,\om)$ for every $I\!\subset\!S$
and the intersection and $\om$-orientations of~$V_I$ agree.
\end{dfn}

\noindent
The intersection and symplectic orientations of~$V_I$ agree if $|I|\!<\!2$.
Thus, an SC symplectic divisor $\{V_i\}_{i\in S}$ with $|S|\!=\!1$ is 
a smooth symplectic divisor in the usual sense. 
If $(X,\om)$ is a 4-dimensional symplectic manifold, 
a finite transverse collection  $\{V_i\}_{i\in S}$ of closed symplectic submanifolds of~$X$ 
of codimension~2 is an SC symplectic divisor 
if and only if all points of the pairwise intersections
$V_{i_1}\!\cap\!V_{i_2}$ with $i_1\!\neq\!i_2$ are positive. 
By \cite[Example~1.9]{Inc}, the latter need not be the case in general.
By Example~\ref{orient_eg2} below, in higher dimensions 
it is not sufficient to consider either the pairwise intersections or 
the deepest (non-empty) intersections .\\

\noindent
As with symplectic manifolds and smooth symplectic divisors,
it is natural to consider the space of all structures compatible 
with an SC symplectic divisor.

\begin{dfn}\label{SCdivstr_dfn}
Let $X$ be a manifold and $\{V_i\}_{i\in S}$ be a finite transverse collection of 
closed submanifolds of~$X$ of codimension~2.
A \sf{symplectic structure on $\{V_i\}_{i\in S}$ in~$X$} is a symplectic form~$\om$ 
on~$X$ such that $V_I$ is a symplectic submanifold of $(X,\om)$ for all $I\!\subset\!S$.
\end{dfn}

\noindent
For $X$ and $\{V_i\}_{i\in S}$ as in Definition~\ref{SCdivstr_dfn}, 
we denote by $\Symp(X,\{V_i\}_{i\in S})$ the space of all symplectic structures 
on $\{V_i\}_{i\in S}$ in~$X$ and by 
$$\Symp^+\big(X,\{V_i\}_{i\in S}\big)\subset \Symp\big(X,\{V_i\}_{i\in S}\big)$$
the subspace of the symplectic forms~$\om$ such that $\{V_i\}_{i\in S}$
is an SC symplectic divisor in~$(X,\om)$.
The latter is a union of topological components of the former.\\

\noindent
We next introduce analogous notions for SC varieties.
A 3-fold SC configuration and the associated SC variety are shown
in Figure~\ref{3conf_fig}.

\begin{dfn}\label{TransConf_dfn1}
Let $N\!\in\!\Z^+$.
An \sf{$N$-fold transverse configuration} is a tuple $\{X_I\}_{I\in\cP^*(N)}$
of manifolds such that $\{X_{ij}\}_{j\in[N]-i}$ is a transverse collection 
of submanifolds of~$X_i$ for each $i\!\in\![N]$ and
$$X_{\{ij_1,\ldots,ij_k\}}\equiv \bigcap_{m=1}^k\!\!X_{ij_m}
=X_{ij_1\ldots j_k}\qquad\forall~j_1,\ldots,j_k\in[N]\!-\!i.$$
\end{dfn}

\begin{dfn}\label{TransConf_dfn2}
Let $N\!\in\!\Z^+$ and $\X\!\equiv\!\{X_I\}_{I\in\cP^*(N)}$ be an $N$-fold transverse configuration
such that $X_{ij}$ is a closed submanifold of~$X_i$ of codimension~2
for all $i,j\!\in\![N]$ distinct.
A \sf{symplectic structure on~$\X$} is a~tuple 
$$(\om_i)_{i\in[N]}\in 
\prod_{i=1}^N\Symp\big(X_i,\{X_{ij}\}_{j\in[N]-i}\big)$$
such that $\om_{i_1}|_{X_{i_1i_2}}\!=\!\om_{i_2}|_{X_{i_1i_2}}$ for all $i_1,i_2\!\in\![N]$.
\end{dfn}

\begin{figure}
\begin{pspicture}(-3,-2)(11,2)
\psset{unit=.3cm}
\psline[linewidth=.1](15,-2)(21,-2)\psline[linewidth=.1](15,-2)(15,4)
\psline[linewidth=.1](15,-2)(11.34,-5.66)\pscircle*(15,-2){.3}
\rput(19.5,2.5){\sm{$X_i$}}\rput(11.5,-1){\sm{$X_j$}}\rput(17,-5){\sm{$X_k$}}
\rput(22.2,-2.1){\sm{$X_{ik}$}}\rput(15.1,4.8){\sm{$X_{ij}$}}
\rput(10.8,-6.1){\sm{$X_{jk}$}}\rput(16.3,-1.2){\sm{$X_{ijk}$}}
\end{pspicture}
\caption{A 3-fold simple crossings configuration and variety.}
\label{3conf_fig}
\end{figure}
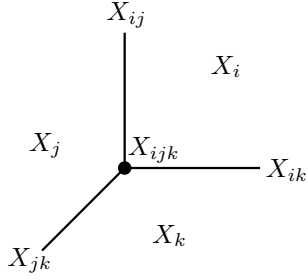

\noindent
For an $N$-fold transverse configuration as in Definition~\ref{TransConf_dfn1}, let
\begin{gather}\label{Xesetdfn_e}
X_{\eset}=\bigg(\bigsqcup_{i=1}^NX_i\bigg)\bigg/\!\!\sim, \quad
X_i\ni x\sim x\in X_j~~\forall~x\in X_{ij}\subset X_i,X_j,~i\neq j\,,\\
\label{Xprtdfn_e}
X_{\prt}\equiv\bigcup_{I\in\cP(N),|I|=2}\hspace{-.25in}\!X_I\subset X_{\eset}\,.
\end{gather}
For $k\!\in\!\Z^{\ge0}$, we call a tuple $(\om_i)_{i\in[N]}$ a \sf{$k$-form on~$X_{\eset}$}
if $\om_i$ is a $k$-form on~$X_i$ for each $i\!\in\![N]$ and 
$$\om_i\big|_{X_{ij}}=\om_j\big|_{X_{ij}} \qquad\forall~i,j\!\in\![N].$$
For $\X$ as in Definition~\ref{TransConf_dfn2}, 
let $\Symp(\X)$ denote the space of all symplectic structures 
on $\X$ and
\BE{Sympdfn_e2}
\Symp^+\big(\X\big)= \Symp\big(\X\big)
\cap \prod_{i=1}^N\Symp^+\big(X_i,\{X_{ij}\}_{j\in[N]-i}\big)\,.\EE
Thus, if $(\om_i)_{i\in[N]}$ is an element of $\Symp^+(\X)$,
then $\{X_{ij}\}_{j\in[N]-i}$ is an SC symplectic divisor in $(X_i,\om_i)$
for each $i\!\in\![N]$.

\begin{dfn}\label{SCC_dfn}
Let $N\!\in\!\Z^+$.
An \sf{$N$-fold simple crossings} (or \sf{SC}) \sf{symplectic configuration} 
is a~tuple 
\BE{SCCdfn_e}\X=\big((X_I)_{I\in\cP^*(N)},(\om_i)_{i\in[N]}\big)\EE
such that $\{X_I\}_{I\in\cP^*(N)}$ is an $N$-fold transverse configuration,
$X_{ij}$ is a closed submanifold of~$X_i$ of codimension~2
for all $i,j\!\in\![N]$ distinct, and
$(\om_i)_{i\in[N]}\in\Symp^+(\X)$.
The \sf{SC symplectic variety associated~to} such a tuple~$\X$ 
is the pair~$(X_{\eset},(\om_i)_{i\in[N]})$.
\end{dfn}

\begin{eg}\label{SCdivvsconf_eg}
An SC symplectic divisor $\{V_i\}_{i\in S}$ in~$(X,\om)$ 
gives rise to an $N$-fold SC symplectic configuration with $N\!=\!|S|\!+\!1$.
For each $i\!\in\![N]$, let
$$\pi_1,\pi_2\!:V_i\!\times\!\C\lra V_i,\C$$
be the component projection maps.
We identify $S$ with $[N\!-\!1]$ and denote by~$\om_{\C}$ the standard symplectic form
on~$\C$.
For $I\!\in\cP^*(N)$ and $i\!\in\![N]$, we define
$$X_I=\begin{cases} V_I\!\times\!\C, &\hbox{if}~N\!\not\in\!I;\\
V_{I-\{N\}},&\hbox{if}~N\!\in\!I;\end{cases} 
\qquad
\om_i=\begin{cases}
\pi_1^*(\om|_{V_i})\!+\!\pi_2^*\om_{\C},&\hbox{if}~i\!\neq\!N;\\
\om,&\hbox{if}~i\!=\!N.
\end{cases} $$
The resulting tuple~$\X$ as in~\eref{SCCdfn_e} is then an $N$-fold 
SC symplectic configuration.
\end{eg}

\noindent
Suppose $\om$ is a symplectic structure on $\{V_i\}_{i\in S}$ in~$X$
in the sense of Definition~\ref{SCdivstr_dfn}.
The symplectic part of the requirements on an SC almost K\"ahler divisor 
$$V\!\equiv\!\bigcup_{i\in S}\!V_i\subset X$$ 
in \cite[Definition~1.3]{Inc}
is equivalent to the existence for each $p\!\in\!V$ of 
an oriented chart~$\psi$ on~$X$ which restricts to oriented charts on
the smooth divisors~$V_i$ after projecting to some coordinate hyperplanes.
The existence of an $\om$-tame almost complex structure~$J$ on~$X$ 
which restricts to an almost complex structure on each~$V_i$
implies the existence of such charts.
However,  the symplectic part of the requirements on an SC almost K\"ahler divisor 
in \cite[Definition~1.3]{Inc} sees the orientations 
only of~$X$, each~$V_i$, and their zero-dimensional intersections,
but not of the intermediate-dimensional intersections of the divisors~$V_i$.
By the $a\!>\!1$ case in the next example, this part does not by itself ensure the existence
of a~$J$ compatible with every~$V_i$.
By the $-1\!<\!a\!<\!-\frac12$ case in this example, the consideration of
the orientations of the pairwise intersections only does not suffice either.

\begin{eg}\label{orient_eg2}
Let $X\!=\!\C^3$ and
\begin{equation*}\begin{split}
\om=&\nd x_1\!\w\!\nd y_1+\nd x_2\!\w\!\nd y_2+\nd x_3\!\w\!\nd y_3\\
&+a\big(\nd x_1\!\w\!\nd y_2\!-\nd y_1\!\w\!\nd x_2\big)
+a\big(\nd x_1\!\w\!\nd y_3\!-\!\nd y_1\!\w\!\nd x_3\big)
+a\big(\nd x_2\!\w\!\nd y_3\!-\!\nd y_2\!\w\!\nd x_3\big)
\end{split}\end{equation*}
for some $a\!\in\!\R$.
We note~that
\begin{gather*}
\om^3=6(1\!-\!a)^2(1\!+\!2a)\,\nd x_1\!\w\!\nd y_1\!\w\!
\nd x_2\!\w\!\nd y_2\!\w\!\nd x_3\!\w\!\nd y_3\,,\\
\om^2|_{\C^3_i}=2\big(1\!-\!a^2\big)\nd x_j\!\w\!\nd y_j\!\w\!\nd x_k\!\w\!\nd y_k
\quad\hbox{if}~\{i,j,k\}\!=\!\{1,2,3\}.
\end{gather*}
Thus, $\om$ is a symplectic structure on $\{\C^3_i\}_{i\in[3]}$
in~$\C^3$ if $a\!\neq\!\pm1,-\frac12$.
The $\om$-orientations on the coordinate lines $\C^3_{ij}$ with $i\!\neq\!j$ and
on the point $\C^3_{123}\!=\!\{0\}$ are the canonical complex orientations.
If $a\!>\!1$,
\begin{enumerate}[label=$\bu$,leftmargin=*]

\item the $\om$-orientation on~$\C^3$ is the canonical complex orientation,

\item the $\om$-orientations on the hyperplanes $\C_i^3$ are 
the opposite of the canonical complex orientations,

\item the intersection and $\om$-orientations on 
 the coordinate lines $\C^3_{ij}$ with $i\!\neq\!j$ are the same,

\item the intersection and $\om$-orientations on the point $\C^3_{123}$ are opposite.

\end{enumerate}
If $-1\!<\!a\!<\!-\frac12$,
\begin{enumerate}[label=$\bu$,leftmargin=*]

\item the $\om$-orientation on~$\C^3$ is the opposite of the canonical complex orientation,

\item the $\om$-orientations on the hyperplanes $\C_i^3$ are 
the canonical complex orientations,

\item the intersection and $\om$-orientations on 
 the coordinate lines $\C^3_{ij}$ with $i\!\neq\!j$ are opposite,

\item the intersection and $\om$-orientations on the point $\C^3_{123}$ are the same.

\end{enumerate}
\end{eg}

\subsection{Regularizations for SC symplectic divisors}
\label{SCDregul_subs}

\noindent
In this section, we formally define the notions of regularizations
for a submanifold $V\!\subset\!X$, 
for a symplectic submanifold with a split normal bundle,
and for a transverse collection $\{V_i\}_{i\in S}$ of submanifolds 
with a symplectic structure~$\om$;
see Definitions~\ref{smreg_dfn}, \ref{sympreg1_dfn}, and~\ref{SCDregul_dfn}\ref{sympregul_it}, 
respectively.
A regularization in the sense of Definition~\ref{SCDregul_dfn}\ref{sympregul_it} 
symplectically models a neighborhood of $x\!\in\!V_I$ in~$X$ on a neighborhood
of the zero section~$V_I$ in the normal bundle~$\cN_XV_I$ split as in~\eref{cNorient_e2}
with a standardized symplectic form.
The existence of such a regularization requires the smooth symplectic divisors~$V_i$
to meet $\om$-orthogonally at~$V_I$, which is rarely the case.
However, by Theorem~\ref{SCD_thm} at the end of this section, 
a virtual kind of existence, 
which suffices for many important applications in symplectic topology, 
is always the case if $\{V_i\}_{i\in S}$ is an SC symplectic divisor in
the sense of Definition~\ref{SCD_dfn}.
This implies that our notion of an SC~symplectic divisor
is natural from the point of view of symplectic topology and its connections with
algebraic geometry simultaneously.\\

\noindent
If $B$ is a manifold, possibly with boundary, and $k\!\in\!\Z^{\ge0}$,
we call a family $(\om_t)_{t\in B}$ of $k$-forms on~$X$ \sf{smooth} 
if the $k$-form~$\wt\om$ on~$B\!\times\!X$ given~by
$$\wt\om_{(t,x)}(v_1,\ldots,v_k)=
\begin{cases}
\om_t|_x(v_1,\ldots,v_k),&\hbox{if}~v_1,\ldots\!v_k\!\in\!T_xX;\\
0,&\hbox{if}~v_1\!\in\!T_tB;
\end{cases}$$
is smooth. 
Smoothness for families of other objects is defined similarly.\\

\noindent
For a vector bundle $\pi\!:\cN\!\lra\!V$, we denote by $\ze_{\cN}$ 
\sf{the radial vector field} on the total space of~$\cN$; it is given~by
$$\ze_{\cN}(v)=(v,v)\in\pi^*\cN =T\cN^{\ver} \lhra T\cN\,.$$
Let $\Om$ be a fiberwise 2-form on~$\cN\!\lra\!V$.
A connection~$\na$ on~$\cN$ induces a projection $T\cN\!\lra\!\pi^*\cN$ and thus
determines an extension~$\Om_{\na}$ of~$\Om$ to a 2-form on 
the total space of~$\cN$.
If $\om$ is a closed 2-form on~$V$, the 2-form
\BE{ombund_e2}
\wh\om \equiv \pi^*\om+\frac12\nd\io_{\ze_{\cN}}\Om_{\na}
\equiv \pi^*\om+\frac12\nd\big(\Om_{\na}(\ze_{\cN},\cdot)\big)\EE
on the total space of $\cN$ is then closed and
restricts to~$\Om$ on $\pi^*\cN\!=\!T\cN^{\ver}$.
If $\om$ is a symplectic form on~$V$ and $\Om$ is a fiberwise symplectic form on~$\cN$,
then~$\wh\om$ is a symplectic form on a neighborhood of~$V$ in~$\cN$.\\

\noindent
We call $\pi\!:(L,\rho,\na)\!\lra\!V$ a \sf{Hermitian line bundle} if
$V$ is a manifold, $L\!\lra\!V$ is a smooth complex line bundle,
$\rho$ is a Hermitian metric on~$L$, 
and $\na$ is a $\rho$-compatible connection on~$L$.
We use the same notation~$\rho$ to denote the square of the norm function on~$L$
and the Hermitian form on~$L$ which is $\C$-antilinear in the second input.
Thus,
$$\rho(v)\equiv\rho(v,v), \quad 
\rho(\fI v,w)=\fI\rho(v,w)=-\rho(v,\fI w) 
\qquad\forall~(v,w)\!\in\!L\!\times_V\!L.$$
Let $\rho^{\R}$ denote the real part of~$\rho$.
A smooth map $h\!:V'\!\lra\!V$ pulls back a Hermitian line bundle $(L,\rho,\na)$
over~$V$ to a Hermitian line bundle
$$h^*(L,\rho,\na)\equiv (h^*L,h^*\rho,h^*\na)\lra V'.$$

\vspace{.1in}

\noindent
A Riemannian metric on an oriented  real vector bundle \hbox{$L\!\lra\!V$} of rank~2
determines a complex structure on the fibers of~$L$.
A \sf{Hermitian structure} on an oriented  real vector bundle \hbox{$L\!\lra\!V$} of rank~2
is a pair $(\rho,\na)$ such that $(L,\rho,\na)$ is a Hermitian line bundle
with the complex structure~$\fI_{\rho}$ determined by the Riemannian metric~$\rho^{\R}$.
If $\Om$ is a fiberwise symplectic form on an oriented vector bundle \hbox{$L\!\lra\!V$} of rank~2,
an \sf{$\Om$-compatible Hermitian structure} on~$L$ is a Hermitian structure $(\rho,\na)$ on~$L$ 
such that $\Om(\cdot,\fI_{\rho}\cdot)=\rho^{\R}(\cdot,\cdot)$.\\

\noindent
Let $(L_i,\rho_i,\na^{(i)})_{i\in I}$ be a finite collection of Hermitian line
bundles over~$V$.
If each $(\rho_i,\na^{(i)})$ is compatible with a fiberwise symplectic form~$\Om_i$ on~$L_i$
and
$$(\cN,\Om,\na)\equiv\bigoplus_{i\in I}\big(L_i,\Om_i,\na^{(i)}\big),$$
then the 2-form~\eref{ombund_e2} is given~by 
\BE{ombund_e}
\wh\om=\wh\om_{(\rho_i,\na^{(i)})_{i\in I}}^{\bu}
\equiv  \pi^*\om+\frac12
\bigoplus_{i\in I} \pi_{I;i}^*\nd\big((\Om_i)_{\na^{(i)}}(\ze_{L_i},\cdot)\big),\EE
where $\pi_{I;i}\!:\cN\!\lra\!L_i$ is the component projection map.\\

\noindent
If in addition $\Psi\!:V'\!\lra\!V$ is an embedding, $I'\!\subset\!I$, and 
$(L_i',\rho_i',\na'^{(i)})_{i\in I'}$ is a finite collection of Hermitian line
bundles over~$V'$, a vector bundle homomorphism 
$$\wt\Psi\!: \bigoplus_{i\in I'}L'_i\lra \bigoplus_{i\in I}L_i$$
covering~$\Psi$ is a \sf{product Hermitian inclusion} if 
$$\wt\Psi\!: (L_i',\rho_i',\na'^{(i)}) \lra 
\Psi^*(L_i,\rho_i,\na^{(i)})$$
is an isomorphism of Hermitian line bundles over~$V'$ for every $i\!\in\!I'$.
We call such a morphism a \sf{product Hermitian isomorphism covering~$\Psi$}
if $|I'|\!=\!|I|$.

\begin{dfn}\label{smreg_dfn}
Let $X$ be a manifold and $V\!\subset\!X$ be a  submanifold
with normal bundle $\cN_XV\!\lra\!V$. 
A \sf{regularization for~$V$ in~$X$} is a diffeomorphism $\Psi\!:\cN'\!\lra\!X$
from a neighborhood of~$V$ in~$\cN_XV$ onto a neighborhood of~$V$ in~$X$ such
that $\Psi(x)\!=\!x$ and the isomorphism
$$ \cN_XV|_x=T_x^{\ver}\cN_XV \lhra T_x\cN_XV
\stackrel{\nd_x\Psi}{\lra} T_xX\lra \frac{T_xX}{T_xV}\equiv\cN_XV|_x$$
is the identity for every $x\!\in\!V$.
\end{dfn}

\noindent
By this definition, a regularization for $V\!=\!X$ in $X$ is the identity map
on $X\!=\!\cN_XX$.\\

\noindent
If $(X,\om)$ is a symplectic manifold and $V$ is a  symplectic submanifold in~$(X,\om)$,
then $\om$ induces a fiberwise symplectic form~$\om|_{\cN_XV}$ on
the normal bundle~$\cN_XV$ of~$V$ in~$X$ via the isomorphism~\eref{cNXVsymp_e}.
We denote the restriction of~$\om|_{\cN_XV}$ to a subbundle $L\!\subset\!\cN_XV$
by~$\om|_L$.

\begin{dfn}\label{sympreg1_dfn}
Let $X$ be a  manifold, $V\!\subset\!X$ be a  submanifold, and
$$\cN_XV=\bigoplus_{i\in I}L_i$$
be a fixed splitting into oriented rank~2 subbundles. 
\begin{enumerate}[label=(\arabic*),leftmargin=*]

\item\label{sympreg1_it} 
If $\om$ is a symplectic form on~$X$ such that $V$ is a symplectic submanifold
and $\om|_{L_i}$ is nondegenerate for every $i\!\in\!I$, then
an \sf{$\om$-regularization for~$V$ in~$X$} is a tuple $((\rho_i,\na^{(i)})_{i\in I},\Psi)$, 
where $(\rho_i,\na^{(i)})$ is an $\om|_{L_i}$-compatible Hermitian structure on~$L_i$
for each $i\!\in\!I$ and $\Psi$ is a regularization for~$V$ in~$X$, such that 
\BE{Psiomcond_e}\Psi^*\om=\wh\om_{(\rho_i,\na^{(i)})_{i\in I}}^{\bu}\big|_{\Dom(\Psi)}.\EE

\item\label{sympreg2_it} If $B$ is a  manifold, possibly with boundary, and
$(\om_t)_{t\in B}$ is a smooth family of symplectic forms on~$X$ 
which restrict to symplectic forms on~$V$, then
an \sf{$(\om_t)_{t\in B}$-family of regularizations} for~$V$ in~$X$ 
is a smooth family of tuples 
\BE{sympreg1dfn_e}(\cR_t)_{t\in B} \equiv 
\big((\rho_{t;i},\na^{(t;i)})_{i\in I},\Psi_t\big)_{t\in B}\EE
such that $\cR_t$ is an $\om_t$-regularization for~$V$ in~$X$ for each $t\!\in\!B$
and 
$$\big\{(t,v)\!\in\!B\!\times\!\cN_XV\!:\,v\!\in\!\Dom(\Psi_t)\big\}\lra X,\qquad
(t,v)\lra \Psi_t(v),$$
is a smooth map from a neighborhood of $B\!\times\!V$ in $B\!\times\!\cN_XV$.
\end{enumerate}
\end{dfn}

\noindent
We next extend these definitions to SC~divisors.
Suppose $\{V_i\}_{i\in S}$ is a transverse collection of codimension~2 submanifolds of~$X$.
For each $I\!\subset\!S$, the last isomorphism in~\eref{cNorient_e2} with $I'\!=\!\eset$
provides 
a natural decomposition 
$$\pi_I\!:\cN_XV_I\!=\!\bigoplus_{i\in I}\cN_{V_{I-i}}V_I  \lra V_I$$
of the normal bundle of~$V_I$ in~$X$ into oriented rank~2 subbundles. 
We take this decomposition as given for the purposes of applying Definition~\ref{sympreg1_dfn}.
If in addition $I'\!\subset\!I$, let
$$\pi_{I;I'}\!:\cN_{I;I'}\equiv 
\bigoplus_{i\in I-I'}\!\!\!\cN_{V_{I-i}}V_I=\cN_{V_{I'}}V_I\lra V_I\,.$$
There are canonical identifications
\BE{cNtot_e}\cN_{I;I-I'}=\cN_XV_{I'}|_{V_I}, \quad
\cN_XV_I=\pi_{I;I'}^*\cN_{I;I-I'}=\pi_{I;I'}^*\cN_XV_{I'}
\qquad\forall~I'\!\subset\!I\!\subset\![N].\EE
The first equality in the second statement above
is used in particular in~\eref{overlap_e}. 

\begin{dfn}\label{TransCollReg_dfn}
Let $X$ be a manifold and $\{V_i\}_{i\in S}$ be a transverse collection 
of submanifolds of~$X$.
A \sf{system of regularizations for}  $\{V_i\}_{i\in S}$ in~$X$ is a~tuple 
$(\Psi_I)_{I\subset S}$, where $\Psi_I$ is a regularization for~$V_I$ in~$X$
in the sense of Definition~\ref{smreg_dfn}, such~that
\BE{Psikk_e}
\Psi_I\big(\cN_{I;I'}\!\cap\!\Dom(\Psi_I)\big)=V_{I'}\!\cap\!\Im(\Psi_I)\EE
for all $I'\!\subset\!I\!\subset\!S$.
\end{dfn}

\noindent
Given a system of regularizations as in Definition~\ref{TransCollReg_dfn}
and $I'\!\subset\!I\!\subset\!S$, let
$$\cN_{I;I'}' = \cN_{I;I'}\!\cap\!\Dom(\Psi_I), \qquad
\Psi_{I;I'}\equiv \Psi_I\big|_{\cN_{I;I'}'}\!: \cN_{I;I'}'\lra V_{I'}\,.$$
The map $\Psi_{I;I'}$ is a regularization for $V_I$ in~$V_{I'}$.
Let 
$$\io\!:\pi_{I;I'}^*\cN_{I;I-I'}\lhra 
\pi_{I;I'}^*\cN_XV_I = \pi_I^*\cN_XV_I\big|_{\cN_{I;I'}} \lhra T\cN_XV_I\big|_{\cN_{I;I'}}$$
denote the canonical inclusion as a subspace of the vertical tangent bundle.
By \eref{Psikk_e},  
$$\nd\Psi_I\!:T\cN_XV_I|_{\cN_{I;I'}'}\lra TX|_{V_{I'}\cap\Im(\Psi_I)}
\quad\hbox{and}\quad
\nd\Psi_I\!: T\cN_{I;I'}|_{\cN_{I;I'}'}\lra TV_{I'}\big|_{V_{I'}\cap\Im(\Psi_I)}$$
are isomorphisms of vector bundles for all $I'\!\subset\!I\!\subset\!S$.
This implies that the composition
\BE{wtPsiIIdfn_e}\begin{split}
\fD\Psi_{I;I'}\!:   \pi_{I;I'}^*\cN_{I;I-I'}|_{\cN_{I;I'}'} \stackrel{\io}{\lhra}  
T\cN_XV_I|_{\cN_{I;I'}'}
&\stackrel{\nd\Psi_I}{\lra} 
TX|_{V_{I'}\cap\Im(\Psi_I)}\\
&\lra \frac{TX|_{V_{I'}}}{TV_{I'}}
\Big|_{V_{I'}\cap\Im(\Psi_I)}=\cN_XV_{I'}|_{V_{I'}\cap\Im(\Psi_I)}
\end{split}\EE
is an isomorphism respecting the natural decompositions of 
$\cN_{I;I-I'}\!=\!\cN_XV_{I'}|_{V_I}$ and $\cN_XV_{I'}$.
For example, 
$$\fD\Psi_{I;\eset}=\Psi_I, \qquad  \fD\Psi_{I;I}=\id_{\cN_XV_I}\,.$$
By the last assumption in Definition~\ref{smreg_dfn}, 
\BE{wtPsiIIprop_e}
\fD\Psi_{I;I'}\big|_{\pi_{I;I'}^*\cN_{I;I-I'}|_{V_I}}\!=\!\id\!:\,
\cN_{I;I-I'}\lra \cN_XV_{I'}|_{V_I}\EE
under the canonical identification of $\cN_{I;I-I'}$ with $\cN_XV_{I'}|_{V_I}$.

\begin{dfn}\label{TransCollregul_dfn}
Let $X$ be a manifold and  $\{V_i\}_{i\in S}$ be a transverse 
collection of submanifolds of~$X$. 
A \sf{regularization for $\{V_i\}_{i\in S}$ in~$X$} 
is a system of regularizations $(\Psi_I)_{I\subset S}$ 
for $\{V_i\}_{i\in S}$ in~$X$ such~that
\BE{overlap_e}
\Dom(\Psi_I)=\fD\Psi_{I;I'}^{\,-1}\big(\Dom(\Psi_{I'})\big), 
\quad
\Psi_I=\Psi_{I'}\circ\fD\Psi_{I;I'}|_{\Dom(\Psi_I)}\EE
for all $I'\!\subset\!I\!\subset\!S$.
\end{dfn}

\begin{dfn}\label{SCDregul_dfn}
Let $X$ be a manifold and  $\{V_i\}_{i\in S}$ be a finite transverse collection 
of closed submanifolds of~$X$ of codimension~2. 

\begin{enumerate}[label=(\arabic*),leftmargin=*]

\item\label{sympregul_it} 
If $\om\in\Symp(X,\{V_i\}_{i\in S})$, then 
an \sf{$\om$-regularization for $\{V_i\}_{i\in S}$ in~$X$} 
is a~tuple
\BE{regdfn_e0}
(\cR_I)_{I\subset S} \equiv  
\big((\rho_{I;i},\na^{(I;i)})_{i\in I},\Psi_I\big)_{I\subset S}\EE
such that $\cR_I$ is an $\om$-regularization for~$V_I$ in~$X$ for each $I\!\subset\!S$,
$(\Psi_I)_{I\subset S}$ is a regularization for $\{V_i\}_{i\in S}$ in~$X$,
and the induced vector bundle isomorphisms
$$\fD\Psi_{I;I'}\!:  \pi_{I;I'}^*\cN_{I;I-I'}\big|_{\cN_{I;I'}'}
\lra \cN_XV_{I'}\big|_{V_{I'}\cap\Im(\Psi_I)}$$
in~\eref{wtPsiIIdfn_e} are product Hermitian isomorphisms
for all $I'\!\subset\!I\!\subset\!S$.

\item\label{sympregul_it2}  If $B$ is a manifold, possibly with boundary, and
$(\om_t)_{t\in B}$ is a smooth family of symplectic forms in $\Symp(X,\{V_i\}_{i\in S})$,
then an \sf{$(\om_t)_{t\in B}$-family of regularizations for $\{V_i\}_{i\in S}$ in~$X$}
is a smooth family of~tuples 
\BE{regdfn_e3}(\cR_{t;I})_{t\in B,I\subset S}
\equiv 
\big((\rho_{t;I;i},\na^{(t;I;i)})_{i\in I},\Psi_{t;I}\big)_{t\in B,I\subset S}\EE
such that $(\cR_{t;I})_{I\subset S}$ is an $\om_t$-regularization for $\{V_i\}_{i\in S}$ in~$X$ 
for each $t\!\in\!B$
and $(\cR_{t;I})_{t\in B}$ is
an $(\om_t)_{t\in B}$-family of regularizations for $V_I$ in~$X$ 
for each $I\!\subset\!S$.
\end{enumerate}
\end{dfn}

\noindent
Let $X$, $\{V_i\}_{i\in S}$, and  $(\om_t)_{t\in B}$ be as in Definition~\ref{SCDregul_dfn}
and
\begin{equation*}\begin{split}
\big(\cR_{t;I}^{(1)}\big)_{t\in B,I\subset S}
&\equiv\big((\rho_{t;I;i},\na^{(t;I;i)})_{i\in I},\Psi_{t;I}^{(1)}\big)_{t\in B,I\subset S},\\
\big(\cR_{t;I}^{(2)}\big)_{t\in B,I\subset S}
&\equiv\big((\rho_{t;I;i},\na^{(t;I;i)})_{i\in I},\Psi_{t;I}^{(2)}\big)_{t\in B,I\subset S}
\end{split}\end{equation*}
be two $(\om_t)_{t\in B}$-families of regularizations for  $(V_i)_{i\in S}$ in~$X$.
We define
\BE{cRequidfn_e}\big(\cR_{t;I}^{(1)}\big)_{t\in B,I\subset S} \cong 
\big(\cR_{t;I}^{(2)}\big)_{t\in B,I\subset S}\EE
if the two families of regularizations agree on the level of germs, i.e.~there exists 
an  $(\om_t)_{t\in B}$-family 
of regularizations as in~\eref{regdfn_e3} such that 
$$\Dom(\Psi_{t;I})\subset\Dom(\Psi_{t;I}^{(1)}),\Dom(\Psi_{t;I}^{(2)})
\qquad\hbox{and}\qquad
\Psi_{t;I}=\Psi_{t;I}^{(1)}\big|_{\Dom(\Psi_{t;I})},\Psi_{t;I}^{(2)}\big|_{\Dom(\Psi_{t;I})}$$
for all $t\!\in\!B$ and $I\!\subset\!S$.\\

\noindent
Definition~\ref{SCDregul_dfn}\ref{sympregul_it2} topologizes the set
$\Aux(X,\{V_i\}_{i\in S})$ of pairs $(\om,(\cR_I)_{I\subset S})$ 
consisting of a symplectic structure~$\om$ on $\{V_i\}_{i\in S}$ in~$X$
and an $\om$-regularization $(\cR_I)_{I\subset S}$ for $\{V_i\}_{i\in S}$ in~$X$.
Families of regularizations satisfying~\eref{cRequidfn_e} are homotopic.\\

\noindent
By Theorem~\ref{SCD_thm} below, a family $(\om_t)_{t\in B}$ of symplectic forms
on~$X$ so that $\{V_i\}_{i\in S}$ is an SC symplectic divisor in $(X,\om_t)$
can be deformed through such symplectic forms to a family $(\om_{t,1})_{t\in B}$
which admits a family $(\wt\cR_{t;I})_{t\in B,I\subset S}$ 
of regularizations for $(V_i)_{i\in S}$ in~$X$.
If $\prt B\!\neq\!\eset$ and the family $(\om_t)_{t\in\prt B}$
admits a family $(\cR_{t;I})_{t\in\prt B,I\subset S}$ 
of regularizations for $(V_i)_{i\in S}$ in~$X$, then
$(\om_t)_{t\in B}$ can be deformed keeping it fixed for $t\!\in\!\prt B$
and  $(\wt\cR_{t;I})_{t\in B,I\subset S}$ can be chosen to extend
$(\cR_{t;I})_{t\in\prt B,I\subset S}$.
This implies that the projection~\eref{AuxSympDiv_e}
is a weak homotopy equivalence.
Furthermore, the family $(\om_t)_{t\in B}$ can be deformed without changing
the cohomology class of each~$\om_t$ or 
the restriction of~$\om_t$ to the complement~$X^*$ of an arbitrarily small neighborhood
of the singular locus of the divisor $\{V_i\}_{i\in S}$.
Since this locus is empty if $|S|\!=\!1$,
the case $|S|\!=\!1$ and $X^*\!=\!X$
of Theorem~\ref{SCD_thm} is a parametrized version of 
the standard Symplectic Neighborhood Theorem \cite[Theorem~3.30]{MS1}.

\begin{thm}\label{SCD_thm}
Let $X$ be a manifold, $\{V_i\}_{i\in S}$ be a finite transverse collection  
of closed submanifolds of~$X$ of codimension~2,
and $X^*\!\subset\!X$ be an open subset, possibly empty, such that 
$\ov{X^*}\!\cap\!V_I\!=\!\eset$ for all $I\!\subset\!S$ with $|I|\!=\!2$.
Suppose
\begin{enumerate}[label=$\bullet$,leftmargin=*]
\item $B$ is a compact manifold, possibly with boundary, 

\item $N(\prt B),N'(\prt B)$ are neighborhoods of $\prt B$ in~$B$
such that $\ov{N'(\prt B)}\!\subset\!N(\prt B)$,

\item $(\om_t)_{t\in B}$ is a smooth family of symplectic forms 
in $\Symp^+(X,\{V_i\}_{i\in S})$, and

\item  $(\cR_{t;I})_{t\in N(\prt B),I\subset S}$ is 
an $(\om_t)_{t\in N(\prt B)}$-family of regularizations for 
$(V_i)_{i\in S}$ in~$X$.
\end{enumerate}
Then there exist a smooth family $(\mu_{t,\tau})_{t\in B,\tau\in\bI}$ of
1-forms on~$X$ such~that 
\begin{gather*}
\om_{t,\tau}\!\equiv\!\om_t\!+\!\nd\mu_{t,\tau}
\in \Symp^+\big(X,\{V_i\}_{i\in S}\big) ~~\forall\,(t,\tau)\!\in\!B\!\times\!\bI,\\
\mu_{t,0}=0~~\forall\,t\!\in\!B, \quad 
\supp\big(\mu_{\cdot,\tau}\big)\subset 
\big(B\!-\!N'(\prt B)\big)\!\times\!(X\!-\!X^*)~~\forall\,\tau\!\in\!\bI,
\end{gather*}
and an $(\om_{t,1})_{t\in B}$-family $(\wt\cR_{t;I})_{t\in B,I\subset S}$ 
of regularizations for $(V_i)_{i\in S}$ in~$X$  such~that
$$\big(\wt\cR_{t;I}\big)_{t\in N'(\prt B),I\subset S}\cong 
\big(\cR_{t;I}\big)_{t\in N'(\prt B),I\subset S}.$$
\end{thm}

\vspace{.1in}

\noindent
This theorem is an immediate consequence of Theorem~\ref{SCC_thm}
applied~to 
\begin{enumerate}[label=$\bu$,leftmargin=*]

\item the $N$-fold transverse configuration $\{X_I\}_{I\in\cP^*(N)}$ and 
the family $(\om_{t;i})_{t\in B,i\in[N]}$ of elements of $\Symp^+(\{X_I\}_{I\in\cP^*(N)})$ 
induced by $(X,\{V_i\}_{i\in S})$ and  $(\om_t)_{t\in B}$
as in Example~\ref{SCdivvsconf_eg},

\item the family $(\fR_t)_{t\in N(\prt B)}$ of regularizations for $\{X_I\}_{I\in\cP^*(N)}$
induced by $(\cR_{t;I})_{t\in N(\prt B),I\subset S}$ as
in Example~\ref{SCdivvsconf_eg2}.

\end{enumerate}
The family of tuples $(\wt\cR_{t;I})_{t\in B}$ with $I\!\in\!\cP_N(N)$
provided by Theorem~\ref{SCC_thm} then satisfies the requirements of
Theorem~\ref{SCD_thm}.\\

\noindent
Theorem~\ref{SCD_thm} can also be obtained without going through Theorem~\ref{SCC_thm}.
The argument would be fundamentally the same, but
Corollary~\ref{posinter_crl0} would no longer be needed
and Lemma~\ref{SympNeigh_lmm} would suffice in place of Proposition~\ref{TubulNeigh_prp}.

\subsection{Regularizations for SC symplectic varieties}
\label{SCCregul_subs}

\noindent
In this section, we define a \sf{regularization} for a transverse configuration~$\X$
of manifolds with a symplectic structure $(\om_i)_{i\in[N]}$
as a tuple of $\om_i$-regularizations for $\{X_{ij}\}_{j\in[N]-i}$ in~$X_i$
that agree on the overlaps; see \eref{SCCregCond_e0} and 
Definition~\ref{SCCregul_dfn}\ref{SCCreg_it}.
We conclude with Theorem~\ref{SCC_thm}: 
the space of SC symplectic varieties in the sense of Definition~\ref{SCC_dfn} 
is weakly homotopy equivalent to the space of those with regularizations.\\

\noindent
Suppose $\{X_I\}_{I\in\cP^*(N)}$ is a transverse configuration in the sense 
of Definition~\ref{TransConf_dfn1}.
For each $I\!\in\!\cP^*(N)$ with $|I|\!\ge\!2$, let
$$\pi_I\!:\cN X_I\equiv \bigoplus_{i\in I}\cN_{X_{I-i}}X_I\lra X_I\,.$$
If in addition $I'\!\subset\!I$, let
$$\pi_{I;I'}\!:\cN_{I;I'}\equiv 
  \bigoplus_{i\in I-I'}\!\!\!\cN_{X_{I-i}}X_I\lra X_I\,.$$
By the last isomorphism in~\eref{cNorient_e2} with $X\!=\!X_i$ for any $i\!\in\!I'$ and 
$\{V_j\}_{j\in S}\!=\!\{X_{ij}\}_{j\in[N]-i}$, 
$$	\cN_{I;I'}=\cN_{X_{I'}}X_I \qquad\forall~I'\!\subset\!I\!\subset\![N],~I'\!\neq\!\eset.$$
Similarly to~\eref{cNtot_e}, there are canonical identifications
\BE{cNtot_e4}\cN_{I;I-I'}=\cN X_{I'}|_{X_I}, \quad
\cN X_I=\pi_{I;I'}^*\cN_{I;I-I'}=\pi_{I;I'}^*\cN X_{I'}
\qquad\forall~I'\!\subset\!I\!\subset\![N];\EE
the first and last identities above hold if $|I'|\!\ge\!2$.

\begin{dfn}\label{TransConfregul_dfn}
Let $N\!\in\!\Z^+$ and $\X\!=\!\{X_I\}_{I\in\cP^*(N)}$ be a transverse configuration.
A \sf{regularization for $\X$} is a tuple 
$(\Psi_{I;i})_{i\in I\subset[N]}$, 
where for each $i\!\in\!I$ the tuple $(\Psi_{I;i})_{I\in\cP_i(N)}$
is a regularization for $\{X_{ij}\}_{j\in[N]-i}$ in~$X_i$ in the sense of
Definition~\ref{TransCollregul_dfn}, such that
\BE{SCCregCond_e0}
\Psi_{I;i_1}\big|_{\cN_{I;i_1i_2}\cap\Dom(\Psi_{I;i_1})}
=\Psi_{I;i_2}\big|_{\cN_{I;i_1i_2}\cap\Dom(\Psi_{I;i_2})}\EE
for all $i_1,i_2\!\in\!I\!\subset\![N]$.
\end{dfn}

\noindent
Given a regularization as in Definition~\ref{TransConfregul_dfn}
and $I'\!\subset\!I\!\subset\![N]$ with $|I|\!\ge\!2$ and $I'\!\neq\!\eset$, let
\BE{PsiIIconf_e}
\cN_{I;I'}'\!=\!\cN_{I;I'}\!\cap\!\Dom(\Psi_{I;i}),~~
\Psi_{I;I'}\!=\!\Psi_{I;i}|_{\cN_{I;I'}'}\!\!:\cN_{I;I'}'\lra X_{I'}
\qquad\hbox{if}~i\!\in\!I';\EE
by~\eref{SCCregCond_e0}, $\Psi_{I;I'}(v)$ does not depend on the choice of $i\!\in\!I'$.
Let
\BE{fDPsiIIconf_e0}\fD\Psi_{I;i;I'}\!: 
\pi_{I;I'}^*\cN_{I;i\cup(I-I')}\big|_{\cN_{I;I'}'}\lra \cN_{I';i}\big|_{\Im(\Psi_{I;I'})}\EE
be the associated vector bundle isomorphism as in~\eref{wtPsiIIdfn_e}.
If $|I'|\!\ge\!2$, we define an isomorphism of split vector bundles
\BE{fDPsiIIconf_e}\begin{split}
\fD\Psi_{I;I'}\!:\pi_{I;I'}^*\cN_{I;I-I'}\big|_{\cN_{I;I'}'}
&\lra \cN X_{I'}\big|_{\Im(\Psi_{I;I'})}\,, \\
\fD\Psi_{I;I'}\big|_{\pi_{I;I'}^*\cN_{I;i\cup(I-I')}\big|_{\cN_{I;I'}'}}
\!&=\fD\Psi_{I;i;I'} \quad\forall~i\!\in\!I';
\end{split}\EE
by~\eref{SCCregCond_e0}, the last maps agree on the overlaps.

\begin{dfn}\label{SCCregul_dfn}
Let $N\!\in\!\Z^+$ and $\X\!\equiv\!\{X_I\}_{I\in\cP^*(N)}$ be a transverse configuration.

\begin{enumerate}[label=(\arabic*),leftmargin=*]

\item\label{SCCreg_it} 
If $(\om_i)_{i\in[N]}$ is a symplectic structure on $\X$ in
the sense of Definition~\ref{TransConf_dfn2}, 
an \sf{$(\om_i)_{i\in[N]}$-regularization for~$\X$} is a~tuple
\BE{SCCregdfn_e0}
\fR\equiv (\cR_I)_{I\in\cP^*(N)} \equiv
\big(\rho_{I;i},\na^{(I;i)},\Psi_{I;i}\big)_{i\in I\subset[N]}\EE
such that $(\Psi_{I;i})_{i\in I\subset[N]}$ is a  regularization 
for $\X$ in the sense of Definition~\ref{TransConfregul_dfn} and
for each $i\!\in\![N]$ the tuple
$$\big((\rho_{I;j},\na^{(I;j)})_{j\in I-i},\Psi_{I;i}\big)_{I\in\cP_i(N)}$$
is an $\om_i$-regularization for $\{X_{ij}\}_{j\in[N]-i}$ in $X_i$
in the sense of Definition~\ref{SCDregul_dfn}\ref{sympregul_it}.

\item\label{SCCreg2_it}  If $B$ is a smooth manifold, possibly with boundary, and 
$(\om_{t;i})_{t\in B,i\in[N]}$ is a smooth family of symplectic structures
on $\X$, then
an \sf{$(\om_{t;i})_{t\in B,i\in[N]}$-family of regularizations for~$\X$} is a family of~tuples
\BE{SCCregul_e2}
(\fR_t)_{t\in B} \equiv
(\cR_{t;I})_{t\in B,I\in\cP^*(N)} \equiv
\big(\rho_{t;I;i},\na^{(t;I;i)},\Psi_{t;I;i}\big)_{t\in B,i\in I\subset [N]}\EE
such that $(\cR_{t;I})_{I\in\cP^*(N)}$ is an $(\om_{t;i})_{i\in[N]}$-regularization for
$\X$ for each $t\!\in\!B$ and
for each $i\!\in\![N]$  the tuple
$$\big((\rho_{t;I;j},\na^{(t;I;j)})_{j\in I-i},\Psi_{t;I;i}\big)_{t\in B,I\in\cP_i(N)}$$
is an $(\om_{t;i})_{t\in B}$-family of regularizations for 
$\{X_{ij}\}_{j\in[N]-i}$ in $X_i$
in the sense of Definition~\ref{SCDregul_dfn}\ref{sympregul_it2}.
\end{enumerate}
\end{dfn}

\noindent
The assumptions in Definition~\ref{SCCregul_dfn}\ref{SCCreg_it} imply
that the corresponding isomorphisms~\eref{fDPsiIIconf_e} are 
product Hermitian isomorphisms covering the maps~\eref{PsiIIconf_e}.

\begin{eg}\label{SCdivvsconf_eg2}
Suppose $X$ is a manifold,
$\{V_i\}_{i\in S}$ is a transverse collection of closed submanifolds of~$X$ of codimension~2,
$(\om_t)_{t\in B}$ is a smooth family of symplectic structures on $\{V_i\}_{i\in S}$ in~$X$,
and $(\cR_{t;I})_{t\in B,I\subset S}$ is an $(\om_t)_{t\in B}$-family of
regularizations for $\{V_i\}_{i\in S}$ in~$X$ as in~\eref{regdfn_e3}.
Let $\X$ and $(\om_{t;i})_{t\in B,i\in[N]}$
be the associated transverse configuration and family of symplectic structures on~it
constructed as in Example~\ref{SCdivvsconf_eg}.
Denote by $(\rho_{\C},\na^{(\C)})$ the standard Hermitian structure on~$\C$.
With notation as in Example~\ref{SCdivvsconf_eg}, 
for $i\!\in\!I\!\subset\![N]$ define
$$\wt\Psi_{t;I;i}=\begin{cases}
(\Psi_{t;I;i},\id_{\C})&\hbox{if}~i\!\neq\!N;\\
\Psi_{t;I},&\hbox{if}~i\!=\!N;
\end{cases} \qquad
(\wt\rho_{t;I;i},\wt\na^{(t;I;i)})=\begin{cases}
\pi_1^*(\rho_{t;I;i},\na^{(t;I;i)})&\hbox{if}~i\!\neq\!N;\\
\pi_2^*(\rho_{\C},\na^{(\C)}),&\hbox{if}~i\!=\!N.
\end{cases}$$
The tuple
$$(\fR_t)_{t\in B}  \equiv
(\wt\cR_{t;I})_{t\in B,I\in\cP^*(N)} \equiv
\big(\wt\rho_{t;I;i},\wt\na^{(t;I;i)},\wt\Psi_{t;I;i}\big)_{t\in B,i\in I\subset[N]}$$
is then an $(\om_{t;i})_{t\in B,i\in[N]}$-family of regularizations for~$\X$.
\end{eg}

\noindent
Let $\X\!\equiv\!\{X_I\}_{I\in\cP^*(N)}$ be an $N$-fold transverse configuration
such that $X_{ij}$ is a closed submanifold of~$X_i$ of codimension~2
for all $i,j\!\in\![N]$ distinct and
$(\om_{t;i})_{t\in B,i\in[N]}$ be a family of symplectic structures on~$\X$.
Suppose the tuples  
\begin{equation*}\begin{split}
\big(\fR_t^{(1)}\big)_{t\in B}& \equiv
\big(\rho_{t;I;i},\na^{(t;I;i)},\Psi_{t;I;i}^{(1)}\big)_{t\in B,i\in I\subset[N]},\\
\big(\fR_t^{(2)}\big)_{t\in B}& \equiv
\big(\rho_{t;I;i},\na^{(t;I;i)},\Psi_{t;I;i}^{(2)}\big)_{t\in B,i\in I\subset[N]}
\end{split}\end{equation*}
are $(\om_{t;i})_{t\in B,i\in[N]}$-families of regularizations for~$\X$.
We define
\BE{fRequidfn_e}\big(\fR_t^{(1)}\big)_{t\in B} \cong \big(\fR_t^{(2)}\big)_{t\in B}\EE
if the two families of regularizations agree on the level of germs, i.e.~there exists 
an $(\om_{t;i})_{t\in B,i\in[N]}$-family of regularizations
as in~\eref{SCCregul_e2} such~that 
\begin{gather*}
\Dom(\Psi_{t;I;i})\subset\Dom(\Psi_{t;I;i}^{(1)}),\Dom(\Psi_{t;I;i}^{(2)}),
\quad
\Psi_{t;I;i}=\Psi_{t;I;i}^{(1)}\big|_{\Dom(\Psi_{t;I;i})},
\Psi_{t;I;i}^{(2)}\big|_{\Dom(\Psi_{t;I;i})}
\end{gather*}
for all $t\!\in\!B$ and $i\!\in\!I\!\subset\![N]$.\\

\noindent
Definition~\ref{SCCregul_dfn}\ref{SCCreg2_it} topologizes the set
$\Aux(\X)$ of pairs $((\om_i)_{i\in[N]},\fR)$ 
consisting of a symplectic structure $(\om_i)_{i\in[N]}$ on~$\X$
and an $(\om_i)_{i\in[N]}$-regularization~$\fR$ for~$\X$.
Families of regularizations satisfying~\eref{fRequidfn_e} are homotopic.
By the following theorem, the projection map
$$\Aux(\X) \lra \Symp^+(\X), \qquad 
\big((\om_i)_{i\in[N]},\fR\big)\lra (\om_i)_{i\in[N]},$$
is a weak homotopy equivalence.

\begin{thm}\label{SCC_thm}
Let $N\!\in\!\Z^+$, $\X\!\equiv\!\{X_I\}_{I\in\cP^*(N)}$ be a transverse configuration
such that $X_{ij}$ is a closed submanifold of~$X_i$ of codimension~2
for all $i,j\!\in\![N]$ distinct,
and $X_i^*\!\subset\!X_i$ for each $i\!\in\![N]$ be an open subset, possibly empty,
such that $\ov{X_i^*}\!\cap\!X_I\!=\!\eset$ for all $i\!\in\!I\!\!\subset\![N]$ with $|I|\!=\!3$.
Suppose
\begin{enumerate}[label=$\bullet$,leftmargin=*]

\item $B$, $N(\prt B)$, and $N'(\prt B)$ are as in Theorem~\ref{SCD_thm},

\item $(\om_{t;i})_{t\in B,i\in[N]}$ is a smooth family of 
elements of $\Symp^+(\X)$, and

\item $(\fR_t)_{t\in N(\prt B)}$ is an 
$(\om_{t;i})_{t\in N(\prt B),i\in[N]}$-family of regularizations for~$\X$.

\end{enumerate}
Then there exist a smooth family $(\mu_{t,\tau;i})_{t\in B,\tau\in\bI,i\in[N]}$ of
1-forms on~$X_{\eset}$ such~that 
\BE{SCCom_e}\begin{split}
&\hspace{1in}\big(\om_{t,\tau;i}\!\equiv\!\om_{t;i}\!+\!\nd\mu_{t,\tau;i}\big)_{i\in[N]}
\in \Symp^+(\X) ~~\forall\,(t,\tau)\!\in\!B\!\times\!\bI,\\
&\mu_{t,0;i}=0~~\forall\,t\!\in\!B,\,i\!\in\![N], \quad 
\supp\big(\mu_{\cdot,\tau;i}\big)\subset 
\big(B\!-\!N'(\prt B)\big)\!\times\!(X_i\!-\!X_i^*)
~~\forall\,\tau\!\in\!\bI,\,i\!\in\![N],
\end{split}\EE
and an $(\om_{t,1;i})_{t\in B,i\in[N]}$-family $(\wt\fR_t)_{t\in B}$
of regularizations for~$\X$ such~that 
\BE{SCCom_e2}\big(\wt\fR_t\big)_{t\in N'(\prt B)} \cong \big(\fR_t\big)_{t\in N'(\prt B)}.\EE
\end{thm}

\vspace{.1in}

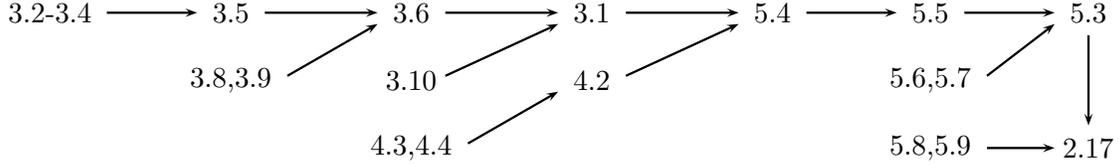
\begin{figure}
\begin{pspicture}(-5.5,-1)(11,1.5)
\psset{unit=.3cm}
\rput(-14,3){\ref{posinter_lmm0}-\ref{posinter_lmm}}\rput(-6,3){\ref{cNbndl_crl}}
\rput(2,3){\ref{SympDef_prp}}\rput(10,3){\ref{SympDefVB_thm}}
\rput(25,3){\ref{extendreg_crl}}\rput(32,3){\ref{SCCweak_prp}}
\rput(18,3){\ref{extendreg0_lmm}}
\rput(-6,0){\ref{cNprtform_lmm},\ref{cutoff_lmm}}\rput(2,0){\ref{SympNeigh_lmm2}}
\rput(10,0){\ref{TubulNeigh_prp}}
\rput(25,0){\ref{regulcomb_lmm},\ref{regulcomb_crl}}
\rput(2,-3){\ref{stratexp_lmm},\ref{SympNeigh_lmm}}
\rput(25,-3){\ref{weakregtoreg_lmm},\ref{weakregtoreg_crl}}\rput(32,-3){\ref{SCC_thm}}
\psline[linewidth=.1]{->}(-11.5,3)(-7.5,3)\psline[linewidth=.1]{->}(-4.5,3)(.5,3) 
\psline[linewidth=.1]{->}(3.5,3)(8.5,3)\psline[linewidth=.1]{->}(11.5,3)(16.5,3)
\psline[linewidth=.1]{->}(19.5,3)(23.5,3)\psline[linewidth=.1]{->}(26.5,3)(30.5,3)
\psline[linewidth=.1]{->}(-3.5,.2)(.5,2.5)\psline[linewidth=.1]{->}(3.5,.2)(8.5,2.5)
\psline[linewidth=.1]{->}(11.5,.2)(16.5,2.5)\psline[linewidth=.1]{->}(27.5,.2)(30.5,2.5)
\psline[linewidth=.1]{->}(4.5,-2.8)(8.5,-.5)\psline[linewidth=.1]{->}(27.5,-3)(30.5,-3)
\psline[linewidth=.1]{->}(32,2)(32,-2)
\end{pspicture}
\caption{The statements used in the proof of Theorem~\ref{SCC_thm}.}
\label{SCCthm_fig}
\end{figure}

\noindent
This theorem is proved in Section~\ref{SCCpf_sec} by induction on the strata of~$\X$
 using the essentially local notion of a weak regularization of 
Definition~\ref{LocalRegul_dfn}.
By Proposition~\ref{SCCweak_prp} and Corollary~\ref{extendreg_crl}, 
a family of elements of $\Symp^+(\X)$
with compatible {\it weak} regularizations over an open subset~$W$ of~$X$ 
which contains all $X_I$ with $I\!\supsetneq\!I^*$ extends over 
a neighborhood of all of~$X_{I^*}$.
Lemma~\ref{extendreg0_lmm} implements the deformations for symplectic forms on split 
vector bundles obtained in Theorem~\ref{SympDefVB_thm} via Proposition~\ref{TubulNeigh_prp};
the latter is a stratified version of the Tubular Neighborhood Theorem that respects
symplectic forms along the base. 
Proposition~\ref{SympDef_prp}, the main step in the proof of Theorem~\ref{SympDefVB_thm},
 makes use of the compatibility-of-orientations
assumptions in Definitions~\ref{SCD_dfn} and~\ref{SCC_dfn}.
By Lemma~\ref{weakregtoreg_lmm} and Corollary~\ref{weakregtoreg_crl},
weak regularizations and equivalences between them 
can be cut down to regularizations and equivalences between regularizations.
The connections between the different parts of the proof of Theorem~\ref{SCC_thm}
are indicated in Figure~\ref{SCCthm_fig}.

\section{Deformations of structures on vector bundles}
\label{LocalDeform_sec}

\noindent 
Let $V$ be an oriented manifold,
$I$ be a finite set, $L_i\!\lra\!V$ be an oriented rank~2 real vector bundle
for each $i\!\in\!I$, and 
\BE{cNsplit_e}\pi\!:\cN\equiv\bigoplus_{i\in I}L_i\lra V\,.\EE
We show that a symplectic structure~$\wt\om$ on a neighborhood~$\cN'$ of~$V$ in~$\cN$
can be deformed, keeping it fixed outside of a smaller neighborhood~$\cN''$
and keeping all natural submanifolds~$\cN_{I'}$ symplectic,
to a very standard symplectic structure~$\wh\om^{\bu}$ near~$V$ as long as~$\wt\om$
 satisfies a simple topological condition.
By Proposition~\ref{SympDef_prp}, this can be done for a symplectic structure~$\wh\om$
on~$\cN'$ induced in a standard way from
a symplectic form~$\om$ on~$V$ and a fiberwise symplectic structure~$\Om$ on~$\cN$.
By Lemma~\ref{SympNeigh_lmm2}, any symplectic structure~$\wt\om$ on a neighborhood~$\cN'$ 
of~$V$ in~$\cN$ can be deformed,  keeping it fixed outside of 
a smaller neighborhood~$\cN''$ and keeping the submanifolds~$\cN_{I'}$ symplectic,
so that it restricts  
to a standard symplectic structure~$\wh\om$ on a smaller neighborhood~$\wt\cN$.
The main statement of this section is Theorem~\ref{SympDefVB_thm};
the remaining statements are used in its proof, but not in the remainder of the paper.\\

\noindent
By Theorem~\ref{SympDefVB_thm},
a finite collection $\{V_i\}_{i\in I}$ of smooth symplectic divisors in~$(X,\om)$
intersecting positively at~$V_I$ can be deformed inside an arbitrarily small neighborhood~$W_I$
of~$V_I$ so that the pairwise intersections $V_i\!\cap\!V_j$ are 
symplectically orthogonal inside~$W_I$.
For $|I|\!=\!2$ and $V_I$ compact, this is \cite[Lemma~2.3]{Gf}.
The compactness assumption is technical and is not fundamental to the three-page proof in~\cite{Gf},
but the condition $|I|\!=\!2$ is.
The latter is clearly illustrated by the one-page proof of \cite[Lemma~3.2.3]{SymingtonThesis}
treating the $V_I\!=\!\{\pt\}$ case of \cite[Lemma~2.3]{Gf} (and thus $\dim_{\R}\!X\!=\!4$).
Our proof of Proposition~\ref{SympDef_prp}, the main ingredient in 
the proof Theorem~\ref{SympDefVB_thm}, follows a completely different approach.
It starts with the linear algebra observation of Lemma~\ref{posinter_lmm0}
and deforms the symplectic forms in three stages as described below~\eref{etaIdfn_e}
and indicated in Figure~\ref{SympDef_fig}.

\subsection{Notation and key statement}
\label{SympDefVB_subs}

\noindent
For a finite set $I$, denote by $\cP^*(I)$  the collection of non-empty subsets of~$I$.
With $\cN$ as in~\eref{cNsplit_e}, let 
\BE{cNsubdfn_e} 
\cN_{I'}=\bigoplus_{i\in I-I'}\!\!\!L_i \quad\forall\,I'\!\subset\!I, 
\qquad \cN_{\prt}=\bigcup_{i\in I}\cN_i\,.\EE
For any $\cN'\!\subset\!\cN$, we define
$$\cN_{I'}'=\cN_{I'}\cap\cN'  \quad\forall\,I'\!\subset\!I, 
\qquad \cN_{\prt}'=\cN_{\prt}\cap\cN'\,.$$
For any neighborhood~$\cN'$ of $V$ in~$\cN$,
$\{\cN_{I'}'\}_{I'\in\cP^*(I)}$ is a transverse configuration
in the sense of Definition~\ref{TransConf_dfn1} such that 
$\cN_{ij}'$ is a closed submanifold of $\cN_i'$ of codimension~2
for all $i,j\!\in\!I$ distinct.\\

\noindent
For $k\!\in\!\Z^{\ge0}$, denote by 
$$\pi\!: \La^k_{\C}\cN^*\lra V \qquad\hbox{and}\qquad
\pi\!:\La^k_{\C}\cN_i^*\lra V,~~i\!\in\!I,$$
the bundles of alternating $k$-tensors on $\cN$ and $\cN_i$, respectively.
For a tensor~$\al$ on~$\cN$ and $j\!\in\!I$, we view 
$\al|_{L_j}$ as a tensor on~$\cN$ via the projection to~$L_j$.
For a tensor~$\al_i$ on~$\cN_i$ and $j\!\in\!I\!-\!\{i\}$, we view 
$\al_i|_{L_j}$ as a tensor on~$\cN_i$ via the projection to~$L_j$.
For such~$\al$ and~$\al_i$, let
\BE{diagpart_e}\al^{\bu}=\sum_{j\in I}\al|_{L_j}\in\La^k_{\C}\cN^*
\qquad\hbox{and}\qquad
\al_i^{\bu}=\sum_{j\in I-\{i\}}\!\!\!\!\!\al_i|_{L_j}\in\La^k_{\C}\cN_i^*\EE
be the \sf{diagonal parts} of~$\al$ and~$\al_i$.
Define
$$\La^k_{\C}\cN_{\prt}^*\equiv
\Big\{(\al_i)_{i\in I}\!\in\!\bigoplus_{i\in I}\!\La^k_{\C}\cN_i^*\!:
\al_{i_1}\big|_{\cN_{i_1i_2}|_{\pi(\al_{i_1})}}
\!=\!\al_{i_2}\big|_{\cN_{i_1i_2}|_{\pi(\al_{i_2})}}~
\forall\,i_1,i_2\!\in\!I\Big\}\lra V\,.$$
This subspace is preserved under taking the diagonal part.
Denote by
$$r_{\cN;\prt}\!: \La^k_{\C}\cN^*\lra \La^k_{\C}\cN_{\prt}^*$$
the natural restriction homomorphism.
It commutes with taking the diagonal part.\\

\noindent
We call a section $(\Om_i)_{i\in I}$ of $\La^k_{\C}\cN_{\prt}^*$ 
a \sf{fiberwise $k$-form on~$\cN_{\prt}$}.
Each $\Om_i$ is then a fiberwise linear $k$-form on~$\cN_i$ and
$$\Om_{i_1}\big|_{\cN_{i_1i_2}}=\Om_{i_2}\big|_{\cN_{i_1i_2}}
\qquad\forall\,i_1,i_2\!\in\!I\,.$$
By Lemma~\ref{cNprtform_lmm}, any such form is the restriction of a fiberwise $k$-form on~$\cN$.
We call a fiberwise $2$-form $(\Om_i)_{i\in I}$ 
on~$\cN_{\prt}$ a \sf{fiberwise symplectic form} 
if each $\Om_i$ is a symplectic form on each fiber of~$\cN_i$.
Let
$$\Symp_V^+\big(\cN_{\prt}\big)\equiv\Symp_V^+\big(\{\cN_{I'}\}_{I'\in\cP^*(I)}\big)$$
be the subspace of fiberwise symplectic forms~$(\Om_i)_{i\in I}$ on $\cN_{\prt}$ such that
for all $i\!\in\!I'\!\subset\!I$ 
the fiberwise 2-form $\Om_i|_{\cN_{I'}}$ is symplectic and
the $\Om_i$-orientation of each fiber of~$\cN_{I'}$ agrees with its canonical
orientation, i.e.~the one induced by the orientations of~$L_i$.\\

\noindent
Let $\cN'$ be a neighborhood of $V$ in~$\cN$.
We call a tuple $(\wt\om_i)_{i\in I}$ a (\sf{closed}) \sf{$k$-form on $\cN_{\prt}'$}
if each $\wt\om_i$ is a (closed) $k$-form on $\cN_i'$ and
$$\wt\om_{i_1}\big|_{T\cN_{i_1i_2}'}
=\wt\om_{i_2}\big|_{T\cN_{i_1i_2}'}  \qquad\forall\,i_1,i_2\in I\,.$$
By a \sf{symplectic structure on $\cN_{\prt}'$}, we  mean an element
$(\wt\om_i)_{i\in I}$ of $\Symp(\{\cN_{I'}'\}_{I'\in\cP^*(I)})$,
i.e.~a closed 2-form on  $\cN_{\prt}'$ which restricts to 
a symplectic form on $\cN_{I'}'$ for each $I'\!\in\!\cP^*(I)$.
Let
$$\Symp^+\big(\cN_{\prt}'\big)\equiv\Symp^+\big(\{\cN_{I'}'\}_{I'\in\cP^*(I)}\big)$$
be the subspace of symplectic structures~$(\wt\om_i)_{i\in I}$ on~$\cN_{\prt}'$ such that 
for all $i\!\in\!I'\!\subset\!I$ the $\wt\om_i$-orientation of~$\cN_{I'}$ 
agrees with its canonical orientation, i.e.~the one induced by the orientations of~$V$
and~$L_i$.\\

\noindent
A symplectic structure $(\wt\om_i)_{i\in I}$ on~$\cN_{\prt}'$
restricts to a symplectic form~$\om$ on~$V$ and determines 
fiberwise symplectic structures $(\Om_i)_{i\in I}$ and $(\Om_i^{\bu})_{i\in I}$
on~$\cN_{\prt}$ via \eref{cNXVsymp_e} and~\eref{diagpart_e}.
If $(\wt\om_i)_{i\in I}$ lies in $\Symp^+(\cN_{\prt}')$, then
$$(\Om_i)_{i\in I},(\Om_i^{\bu})_{i\in I}\in  \Symp_V^+\big(\cN_{\prt}\big)\,.$$
We call $(\Om_i^{\bu})_{i\in I}$ the \sf{diagonalized fiberwise 2-form on~$\cN_{\prt}$ 
determined by $(\wt\om_i)_{i\in I}$}.
A tuple $(\na^{(i)})_{i\in I}$ of connections on~$L_i$ determines
a connection~$\na$ on each~$\cN_i$.
We call the 2-form $(\wh\om_i^{\bu})_{i\in I}$ on~$\cN_{\prt}$ determined by 
$\om$, $(\Om_i^{\bu})_{i\in I}$, and these connections~$\na$ as in~\eref{ombund_e2}
the \sf{diagonalized 2-form on~$\cN_{\prt}$ determined by
$(\wt\om_i)_{i\in I}$ and $(\na^{(i)})_{i\in I}$}.

\begin{thm}\label{SympDefVB_thm}
Let $V$ be a manifold, $I$ be a finite set with $|I|\!\ge\!2$, 
$L_i\!\lra\!V$ be an oriented rank~2 real vector bundle for each $i\!\in\!I$,
and $U\!\subset\!V$  be an open subset, possibly empty. 
Suppose 
\begin{enumerate}[label=$\bullet$,leftmargin=*]

\item $B$ is a compact manifold, possibly with boundary, 
$N(\prt B)$ is a neighborhood of $\prt B$ in~$B$, 
and $\cN'$ is a neighborhood of $V$ in~$\cN$,

\item $(\na^{(t;i)})_{t\in B}$ is a smooth family of connections
on~$L_i$ for each $i\!\in\!I$,

\item $(\wt\om_{t;i})_{t\in B,i\in I}$ is a smooth family in $\Symp^+(\cN_{\prt}')$
such~that 
\BE{SympDefVB_e0}
\big(\wt\om_{t;i}\big)_{i\in I}=\big(\wh\om_{t;i}^{\bu}|_{\cN_i'}\big)_{i\in I}
~~\forall\,t\!\in\!N(\prt B), \quad
\big(\wt\om_{t;i}|_{\cN_i'|_U}\big)_{i\in I}
=\big(\wh\om_{t;i}^{\bu}|_{\cN_i'|_U}\big)_{i\in I}~~\forall\,t\!\in\!B,\EE
where $(\wh\om_{t;i}^{\bu})_{t\in B,i\in I}$ is the smooth family of 
diagonalized 2-forms on~$\cN_{\prt}$ 
determined by $(\wt\om_{t;i})_{t\in B,i\in I}$ and $(\na^{(t;i)})_{t\in B,i\in I}$.
\end{enumerate}
Then there exist neighborhoods $\wh\cN\!\subset\!\cN''\!$ of~$V$ in~$\cN'$ 
such that $\ov{\cN''}\!\subset\!\cN'$
and a smooth family  $(\mu_{t,\tau;i})_{t\in B,\tau\in\bI,i\in I}$ of
1-forms on~$\cN_{\prt}'$ such~that 
\BE{SympDefThm_e0a}
\big(\wt\om_{t,\tau;i}\big)_{i\in I}\equiv
\big(\wt\om_{t;i}\!+\!\nd\mu_{t,\tau;i}\big)_{i\in I}\EE
is a symplectic structure on~$\cN_{\prt}'$ for all $(t,\tau)\!\in\!B\!\times\!\bI$ and
\BE{SympDefThm_e0}\begin{split}
&\mu_{t,0;i}=0, \quad  \wt\om_{t,\tau;i}|_V=\wt\om_{t;i}|_V,\quad
\wt\om_{t,1;i}|_{\wh\cN_i}=\wh\om_{t;i}^{\bu}|_{\wh\cN_i}, \quad
\supp\big(\mu_{\cdot,\tau;i}\big)\subset 
\big(B\!-\!N(\prt B)\big)\!\times\!\cN''|_{V-U}
\end{split}\EE
for all $t\!\in\!B$, $\tau\!\in\!\bI$, and $i\!\in\!I$.
\end{thm}

\subsection{Some linear algebra}
\label{LinAlg_subs}

\noindent
This section collects some basic, but crucial, observations.
Lemmas~\ref{posinter_lmm0} and~\ref{posinter_lmm} 
can be seen as versions of \cite[Lemmas~5.5,5.8]{MAff}.
According to these lemmas and Corollary~\ref{posinter_crl0}, 
every linear 2-form form~$\Om$ on~$\C^n$ such~that
$$\big(\Om|_{\C^n_i}\big)_{i\in[n]}\in 
\Symp_{\{0\}}^+\!\big(\C^n_{\prt}\big)\subset \Symp^+\big(\C^n_{\prt}\big), $$ 
i.e.~$\Om|_{\C^n_I}$ is symplectic and induces the complex orientation of~$\C^n_I$
for every $I\!\in\!\cP^*(n)$, 
can be homotoped in a canonical way to the standard symplectic~form
$$\Om_{\std}\equiv \nd x_1\w\nd y_1+\ldots+ \nd x_n\w\nd y_n$$
while keeping each coordinate subspace~$\C^n_I$ symplectic.
For a 2-form form~$\Om$ on~$\C^n$ and $s\!\in\!\R$, let
$$\Om_{i;s}=\Om+s\,\nd x_i\!\w\!\nd y_i~~\forall\,i\!\in\![n], \qquad
\Om_s=\Om+s\,\Om_{\std}.$$

\begin{lmm}\label{posinter_lmm0}
Let $\Om$ be a linear symplectic form on~$\C^n$ such that $\Om|_{\C^n_I}$ 
is symplectic for every $I\!\in\!\cP(n)$.
If the $\Om$-orientation of~$\C^n_I$ agrees with its complex orientation for every $I\!\in\!\cP(n)$,
then $\Om_{i;s}|_{\C^n_I}$ is symplectic for all $I\!\in\!\cP(n)$,
$s\!\in\!\R^{\ge0}$, and $i\!\in\![n]$.
\end{lmm} 

\begin{proof}
If $i\!\in\!I$, then $\Om_{i;s}|_{\C^n_I}\!=\!\Om|_{\C^n_I}$ and there is nothing to prove.
Suppose $i\!\not\in\!I$, as in Figure~\ref{PosInter_fig}.
Let $\C^{\Om}_{Ii;i}\!\subset\!\C_I^n$ be the $\Om$-orthogonal complement of~$\C^n_{Ii}$.
Since 
the $\Om$-orientations of~$\C^n_{Ii}$ and $\C^n_{Ii}\!\oplus\!\C^{\Om}_{Ii;i}$ agree with 
the complex orientations of $\C^n_{Ii}$ and~$\C^n_I$, respectively,
the $\Om$-orientation of~$\C^{\Om}_{Ii;i}$ is the same as the orientation induced
by the restriction of $\nd x_i\!\w\!\nd y_i$.
It follows that the restrictions of $\Om_{i;s}$ to $\C^{\Om}_{Ii;i}$
and $\C^n_I$ are symplectic.
\end{proof}

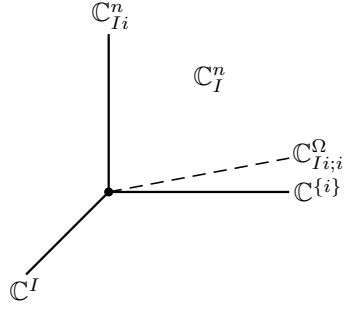
\begin{figure}
\begin{pspicture}(-3,-2)(11,2.2)
\psset{unit=.3cm}
\psline[linewidth=.1](15,-2)(23,-2)\psline[linewidth=.1](15,-2)(15,5)
\psline[linewidth=.1](15,-2)(11.34,-5.66)\pscircle*(15,-2){.2}
\psline[linewidth=.07,linestyle=dashed](15,-2)(23,-.5)
\rput(19.5,3){\sm{$\C^n_I$}}
\rput(24.3,-1.9){\sm{$\C^{\{i\}}$}}\rput(24.3,-.4){\sm{$\C^{\Om}_{Ii;i}$}}
\rput(15.1,5.8){\sm{$\C^n_{Ii}$}}\rput(11.3,-6.3){\sm{$\C^I$}}
\end{pspicture}
\caption{An illustration for the proof of Lemma~\ref{posinter_lmm0}.}
\label{PosInter_fig}
\end{figure}

\begin{crl}\label{posinter_crl0}
Let $\Om$ be a linear 2-form on~$\C^n$ such that $\Om|_{\C^n_I}$ 
is symplectic for every $I\!\in\!\cP^*(n)$.
If the $\Om$-orientation of~$\C^n_I$ agrees with its complex orientation for every $I\!\in\!\cP^*(n)$,
then $\Om_{i;s}|_{\C^n_I}$ is symplectic for all $I\!\in\!\cP^*(n)$,
$s\!\in\!\R^{\ge0}$, and $i\!\in\![n]$.
\end{crl} 

\begin{proof}
If $i\!\in\!I$, $\Om_{i;s}|_{\C^n_I}\!=\!\Om|_{\C^n_I}$ and there is nothing to prove.
If $j\!\in\!I\!-\!i$, the claim follows from Lemma~\ref{posinter_lmm0} 
with~$n$ replaced by $n\!-\!1$ (drop $j$ from~$I$ and~$[n]$). 
\end{proof}

\begin{lmm}\label{posinter_lmm}
If $\Om$ and $\Om^{\circ}$ are 2-forms on~$\C^n$,
then there exists $s_0\!\in\!\R^{\ge0}$ such that 
$(\Om_s\!+\!\tau\Om^{\circ})|_{\C^n_I}$ is symplectic 
for all $I\!\in\!\cP(n)$, $\tau\!\in\!\bI$, and $s\!\ge\!s_0$.
\end{lmm} 

\begin{proof} 
This statement is equivalent to the restriction of the 2-form
\hbox{$\Om_{\std}\!+\!\frac1s\Om\!+\!\frac\tau{s}\Om^{\circ}$}
to each $\C^n_I$ being symplectic for all $s$ sufficiently large.
This is clear, since being symplectic is an open condition.
\end{proof}

\begin{crl}\label{cNbndl_crl}
Let $V$, $I$, $\cN$, and $B$ be as in Theorem~\ref{SympDefVB_thm}
and $(\Om_t)_{t\in B,i\in I}$ be a smooth family of fiberwise 2-forms on~$\cN$
such that $(\Om_t|_{\cN_i})_{i\in I}\in\Symp_V^+(\cN_{\prt})$ for every $t\!\in\!B$.
\begin{enumerate}[label=(\arabic*),leftmargin=*]

\item\label{posinter0_it} For all $t\!\in\!B$ and $s\!\ge\!0$, 
$((\Om_t\!+\!s\Om_t^{\bu})|_{\cN_i})_{i\in I}\!\in\!\Symp_V^+(\cN_{\prt})$.

\item\label{posinter_it} For every smooth family $(\Om_t^{\circ})_{t\in B}$
of fiberwise 2-forms on~$\cN$ and every compact subset $K\!\subset\!V$,
there exist $s_0\!\in\!\R^{\ge0}$ such that 
$$\big((\Om_t\!+\!s\Om_t^{\bu}\!+\!\tau\Om_t^{\circ})|_{\cN_i}\big)_{i\in I}
\in\Symp_K^+\big(\cN_{\prt}|_K\big)
\qquad\forall\,t\!\in\!B,~\tau\!\in\!\bI,~s\!\in\!\R^{\ge0},\,s\!\ge\!s_0.$$
\end{enumerate}
\end{crl}

\begin{proof}
This follows immediately from Corollary~\ref{posinter_crl0}
and Lemma~\ref{posinter_lmm}.
\end{proof}

\subsection{Deformations of standard structures}
\label{LinDeform_subs}

\noindent
Let $\om$ be a symplectic form on a manifold~$V$ and $\cN_{I'}\!\subset\!\cN$ 
be as in~\eref{cNsubdfn_e}.
For a fiberwise 2-form $(\Om_i)_{i\in I}$ on~$\cN_{\prt}$,
\eref{ombund_e2} induces a closed 2-form $(\wh\om_i)_{i\in I}$ on~$\cN_{\prt}$.
If $(\Om_i)_{i\in I}$ is an element of $\Symp_V^+\big(\cN_{\prt})$, then 
\begin{enumerate}[label=$\bu$,leftmargin=*]

\item $(\Om_i)_{i\in I}$ induces a fiberwise symplectic form on the subbundle~$\cN_{I'}$
compatible with its canonical orientation for every $I'\!\in\!\cP^*(I)$,

\item $(\Om_i^{\bu})_{i\in I}$ is a fiberwise symplectic form on~$\cN_{\prt}$, and

\item $\wh\om_i|_{\cN_{I'}'}$ is a symplectic form for all $i\!\in\!I'\!\subset\!I$
and for some neighborhood~$\cN'$ of $V\!\subset\!\cN$. 

\end{enumerate}
By Proposition~\ref{SympDef_prp} below, the tuple $(\wh\om_i)_{i\in I}$ can then be deformed,
while keeping it fixed outside of some neighborhood $\cN''\!\subsetneq\!\cN'$
and keeping all submanifolds $\cN_{I'}'$ with $I'\!\in\!\cP^*(I)$ symplectic,
to a symplectic form $(\wh\om_{1;i})_{i\in I}$ on~$\cN_{\prt}'$ 
so that $(\wh\om_{1;i})_{i\in I}$ agrees with
the 2-form $(\wh\om_i^{\bu})_{i\in I}$ induced by $(\Om_i^{\bu})_{i\in I}$
on a smaller neighborhood $\wh\cN_{\prt}\!\subset\!\cN_{\prt}''$ of~$V$.

\begin{prp}\label{SympDef_prp}
Let $U\!\subset\!V$, $I$, $\cN$, $N(\prt B)\!\subset\!B$, and
$(\na^{(t;i)})_{t\in B,i\in I}$ be as in Theorem~\ref{SympDefVB_thm}.
Suppose 
\begin{enumerate}[label=$\bullet$,leftmargin=*]

\item $(\om_t)_{t\in B}$ is a smooth family of symplectic forms on~$V$, 

\item $(\Om_{t;i})_{t\in B,i\in I}$ and $(\Om_{t;i}')_{t\in B,i\in I}$
are smooth families in $\Symp_V^+(\cN_{\prt})$ such~that
\BE{Omcond_e}\begin{split}
&\hspace{1.5in}\big(\Om_{t;i}^{\bu}\big)_{i\in I}=\big(\Om_{t;i}'^{\,\bu}\big)_{i\in I}
~~\forall\,t\!\in\!B,\\
&\big(\Om_{t;i}\big)_{i\in I}=\big(\Om_{t;i}'\big)_{i\in I}
~~\forall\,t\!\in\!N(\prt B),\quad
\big(\Om_{t;i}|_U\big)_{i\in I}=\big(\Om_{t;i}'|_U\big)_{i\in I}~~
\forall\,t\!\in\!B,
\end{split}\EE

\item $(\wh\om_{t;i})_{t\in B,i\in I}$ (resp.~$(\wh\om_{t;i}')_{t\in B,i\in I}$)
is the family of closed 2-forms on~$\cN_{\prt}$ 
induced as in~\eref{ombund_e2} 
by the families $(\om_t)_{t\in B}$ of symplectic forms on~$V$, 
$(\Om_{t;i})_{t\in B,i\in I}$ (resp.~$(\Om_{t;i}')_{t\in B,i\in I}$)
of fiberwise symplectic forms on~$\cN_{\prt}$,
and $(\na^{(t;i)})_{t\in B,i\in I}$ of connections on~$L_i$.

\end{enumerate}
If $V\!-\!U$ is compact, then
there exist neighborhoods $\wh\cN\!\subset\!\cN''\!\subset\!\cN'$ of $V\!\subset\!\cN$
such that $\ov{\cN''}\!\subset\!\cN'$ and a smooth family 
$(\mu_{t,\tau;i})_{t\in B,\tau\in\bI,i\in I}$ of
1-forms on~$\cN_{\prt}$ such~that 
\BE{SympDef_e0a}\big(\wh\om_{t,\tau;i}\big)_{i\in I}\equiv
\big((\wh\om_{t;i}\!+\!\nd\mu_{t,\tau;i})|_{\cN_i'}\big)_{i\in I}\EE
is a symplectic structure on~$\cN_{\prt}'$ for all $(t,\tau)\!\in\!B\!\times\!\bI$ and
\BE{SympDef_e0}\begin{split}
&\mu_{t,0;i}=0, \quad \wh\om_{t,\tau;i}|_V=\om_t,\quad
\wh\om_{t,1;i}|_{\wh\cN_i}=\wh\om_{t;i}'|_{\wh\cN_i}, \quad
\supp\big(\mu_{\cdot,\tau;i}\big)\subset 
\big(B\!-\!N(\prt B)\big)\!\times\!\cN''|_{V-U}
\end{split}\EE
for all $t\!\in\!B$, $\tau\!\in\!\bI$, and $i\!\in\!I$.
\end{prp}

\begin{rmk}\label{SympDefPrp_rmk}
If $(\Om_i)_{i\in I}$ and $(\Om_i')_{i\in I}$ are diagonal elements of $\Symp_V^+(\cN_{\prt})$,
then the associated closed 2-forms $(\wh\om_i)_{i\in I}$ and $(\wh\om_i')_{i\in I}$ 
can be directly deformed into each other using the bump function of Lemma~\ref{cutoff_lmm}
so that the restrictions of these deformations stay in $\Symp^+(\cN_{\prt}')$
for some neighborhood~$\cN_{\prt}'$ of~$V$ in~$\cN_{\prt}$.
Thus, the first assumption in~\eref{Omcond_e} is needed for the very last conclusion~only.
\end{rmk}

\begin{lmm}\label{cNprtform_lmm}
Let $V$, $I$, and $\cN$ be as in Theorem~\ref{SympDefVB_thm} and $k\!\in\!\Z^{\ge0}$.
There exists a smooth bundle~map
$$\Phi_{\cN;\prt}\!: \La^k_{\C}\cN_{\prt}^*\lra \La^k_{\C}\cN^*$$
such that $r_{\cN;\prt}\!\circ\!\Phi_{\cN;\prt}\!=\!\id$.
\end{lmm}

\begin{proof}
Let $(\al_i)_{i\in I}$ be an element of $\La^k_{\C}\cN_{\prt}^*$ in the fiber
over a point $x\!\in\!V$.
Thus,
\BE{SympDefPrp_e1b}
\al_{i_1}|_{\cN_{i_1i_2}|_x}=\al_{i_2}|_{\cN_{i_1i_2}|_x}\quad\forall~i_1,i_2\!\in\!I.\EE
Assume that $I\!=\![\ell^*]$ for some $\ell^*\!\ge\!2$.
For $i,\ell\!\in\!I$, let 
$$\pi_i\!:\cN\lra\cN_i \qquad\hbox{and}\qquad  \pi_{i;\ell}\!:\cN_i\lra \cN_{i\ell}$$
denote the projection maps.
Define
$$\al'_1=\pi_1^*\al_1, \quad
\al'_{\ell+1}=\al'_{\ell}+ 
\pi_{\ell+1}^*\big(\al_{\ell+1}\!-\!\al'_{\ell}|_{\cN_{\ell+1}}\big)
~~\forall\,\ell\!\in\![\ell^*\!-\!1], \quad \al=\al'_{\ell^*}.$$
Since $\pi_i|_{\cN_i}\!=\!\id_{\cN_i}$ and $\pi_{\ell}|_{\cN_i}\!=\!\pi_{i;\ell}$,
it follows~that
$$\al'_i|_{\cN_i}=\al_i, \quad
\al'_{\ell}|_{\cN_i}=\al'_{\ell-1}|_{\cN_i}+ 
\pi_{i;\ell}^*\big(\al_{\ell}|_{\cN_{i\ell}}
\!-\!\al'_{\ell-1}|_{\cN_{i\ell}}\big) ~~\forall\,\ell\!\in\![\ell^*]\!-\![i].$$
By~\eref{SympDefPrp_e1b} and induction, these identities imply that 
$$\al'_{\ell}|_{\cN_i}=\al_i \quad 
\forall\, \ell\!\in\![\ell^*]\!-\![i\!-\!1],\,i\!\in\![\ell^*].$$
Thus, the constructed smooth bundle homomorphism 
$$\Phi_{\cN;\prt}\!:\La^k_{\C}\cN_{\prt}^*\lra \La^k_{\C}\cN^*,
\qquad \Phi_{\cN;\prt}\big((\al_i)_{i\in I}\big)=\al,$$
is a right inverse for $r_{\cN;\prt}$.
\end{proof}

\begin{lmm}\label{cutoff_lmm}
There exists a smooth function
$\chi\!:(0,1)\!\times\!(1,\i)\!\times\!\R\lra\R^{\ge0}$
such that 
\BE{cutoffLmm_e1}
\chi(\de,s,r)=\begin{cases}
s,&\hbox{if}~r\!\le\!\de;\\
0,&\hbox{if}~r\!\ge\!\de\ne^{4s/\de};
\end{cases}\quad
\chi(\de,s,r)\le s, \quad
\bigg|\frac{\prt}{\prt r}\chi(\de,s,r)\bigg|r\le\de.\EE
\end{lmm}

\begin{proof}
If $\de\!\in\!(0,1)$ and $s\!\in\!(1,\i)$, then $2\de\le \de\ne^{4s/\de}\!-\!1$.
Let $\eta\!:\R\!\lra\!\bI$ be a smooth function such~that
$$\eta(r)=\begin{cases}0,&\hbox{if}~r\le0;\\
1,&\hbox{if}~r\ge1;
\end{cases}\qquad \big|\eta'(r)\big|\le2.$$
The smooth function 
$$\chi(\de,s,r)=\eta\big(\de\ne^{4s/\de}\!-\!r\big)
\Big(s-\eta(r/\de\!-\!1)\frac{\de}{4}\ln\big(r/\de)\Big)$$
then satisfies~\eref{cutoffLmm_e1}.
\end{proof}

\begin{proof}[{\bf{\emph{Proof of Proposition \ref{SympDef_prp}}}}]
Let $\Phi_{\cN;\prt}$ be as in Lemma~\ref{cNprtform_lmm}. 
For each $t\!\in\!B$, define 
\begin{gather}
\Om_t=\Phi_{\cN;\prt}\big((\Om_{t;i})_{i\in I}\big), \quad 
\Om_t'=\Phi_{\cN;\prt}\big((\Om_{t;i}')_{i\in I}\big), \quad
\Om_t^{\circ}=\Om_t'\!-\!\Om_t\,,\\
\label{SympDef_e4}
\wh\om_t=\pi^*\om_t+\frac12\nd\io_{\ze_{\cN}}\{\Om_t\}_{\na^{(t)}}\,, \quad
\wh\om_t'=\pi^*\om_t+\frac12\nd\io_{\ze_{\cN}}\{\Om_t'\}_{\na^{(t)}}
=\wh\om_t+\frac12\nd\io_{\ze_{\cN}}\{\Om_t^{\circ}\}_{\na^{(t)}}\,.
\end{gather}
By~\eref{Omcond_e},
\BE{SympDef_e5a}
\Om_t^{\bu}=\wh\Om_t'^{\,\bu}~~\forall\,t\!\in\!B,\qquad
\supp\big(\Om_{\cdot}^{\circ}\big)\subset \big(B\!-\!N(\prt B)\big)
\!\times\!\big(V\!-\!U\big)\,.\EE
Since $r_{\cN;\prt}\!\circ\!\Phi_{\cN;\prt}\!=\!\id$,
\BE{SympDef_e5} 
\Om_t|_{\cN_i}=\Om_{t;i},~~ \Om_t'|_{\cN_i}=\Om_{t;i}',~~
\wh\om_t|_{\cN_i}=\wh\om_{t;i}, ~~ \wh\om_t'|_{\cN_i}=\wh\om_{t;i}'
\quad\forall\,t\!\in\!B,\,i\!\in\!I.\EE

\vspace{.2in}

\noindent
We construct the desired families of 1-forms by pasting together 
three families of 1-forms via smooth functions 
$\eta_{\bI;1},\eta_{\bI;2},\eta_{\bI;3}\!:\R\!\lra\!\bI$ such that 
\BE{etaIdfn_e}\eta_{\bI;1}(\tau)=\begin{cases}
0,&\hbox{if}~\tau\!\le\!0;\\
1,&\hbox{if}~\tau\!\ge\!\frac13;
\end{cases} \quad
\eta_{\bI;2}(\tau)=\begin{cases}
0,&\hbox{if}~\tau\!\le\!\frac13;\\
1,&\hbox{if}~\tau\!\ge\!\frac23;
\end{cases}\quad
\eta_{\bI;3}(\tau)=\begin{cases}
0,&\hbox{if}~\tau\!\le\!\frac23;\\
1,&\hbox{if}~\tau\!\ge\!1.
\end{cases}\EE
We first increase the diagonal part~$\Om_t^{\bu}$ of~$\Om_t$ and~$\Om_t'$ as 
in Corollary~\ref{cNbndl_crl}\ref{posinter0_it}.
We then add the difference~$\Om_t^{\circ}$ with~$\Om_t'$
as in Corollary~\ref{cNbndl_crl}\ref{posinter_it}.
Finally, we reduce the diagonal part back to where it started.
We cut off all three deformations by bump functions supported near~$V$ so that 
the forms do not change too far away from~$V$, i.e.~on $\cN\!-\!\cN''$.
This construction is illustrated in Figure~\ref{SympDef_fig}.\\ 

\noindent
Fix a metric on~$V$ and a norm $\rho(\cdot)\!=\!|\cdot|^2$ on~$\cN$. 
For any $\vr\!\in\!\R^+$, let
$$\cN(\vr) =\big\{v\!\in\!\cN\!:\,|v|\!<\!\vr\big\}.$$
Since $B$ is compact, we can choose the norm on~$\cN$ so that 
the 2-form~$\wh\om_t$ is nondegenerate on  $\cN'\!\equiv\!\cN(1)$ for every $t\!\in\!B$.
Since $B$ and $V\!-\!U$ are compact, 
for every smooth family $\Xi\!\equiv(\Xi_t)_{t\in B}$ of fiberwise \hbox{2-forms}
on~$\cN$ there exists $C_{\Xi}\!\in\!\R^+$ such~that 
\BE{XiBnd_e}\begin{split}
\Big|\io_{\ze_{\cN}}\{\Xi_t\}_{\na^{(t)}}\Big|_v,
\Big|\frac12\nd\,\io_{\ze_{\cN}}\{\Xi_t\}_{\na^{(t)}}-\{\Xi_t\}_{\na^{(t)}}\Big|_v
&\le C_{\Xi}|v|,\\
\Big|\frac{\nd\rho}{\rho}\w\io_{\ze_{\cN}}\{\Xi_t\}_{\na^{(t)}}\Big|_v
\le C_{\Xi}
\end{split}
\qquad \forall~v\!\in\!\cN'|_{V-U}\,.\EE

\vspace{.2in}

\noindent
For $\fs\!\in\!C^{\i}(B\!\times\!V;\R)$ and $\tau\!\in\!\R$, 
let $\Om_{t;\fs}$ and $\Om_{t;\fs,\tau}$ be the fiberwise 2-forms on~$\cN$
given~by
\BE{Omdef_e}\Om_{t;\fs}\big|_x=\Om_t\big|_x+\fs(t,x)\,\Om_t^{\bu}\big|_x, \quad
\Om_{t;\fs,\tau}\big|_x=\Om_{t;\fs}\big|_x +\tau\,\Om_t^{\circ}\big|_x
=\Om_{t;\fs}'\big|_x -(1\!-\!\tau)\Om_t^{\circ}\big|_x\EE
for all $x\!\in\!V$;
the second equality in the second statement above holds by the first property 
in~\eref{SympDef_e5a}.
By the first two equalities in~\eref{SympDef_e5} and 
Corollary~\ref{cNbndl_crl}\ref{posinter0_it}, the restrictions of 
\hbox{$\Om_{t;\fs,0}\!=\!\Om_{t;\fs}$} and of {$\Om_{t;\fs,1}\!=\!\Om_{t;\fs}'$}
to~$\cN_{I'}$ are nondegenerate 
for all $I'\!\in\!\cP^*(I)$, $t\!\in\!B$, and \hbox{$\fs\!\in\!C^{\i}(B\!\times\!V;\R^{\ge0})$}. 
By Corollary~\ref{cNbndl_crl}\ref{posinter_it}, there exists $s_0\!\in\!\R^+$
such that $\Om_{t;\fs,\tau}|_{\cN_{I'}}$ is nondegenerate over $x\!\in\!V\!-\!U$ whenever 
$$I'\in\cP^*(I),\quad t\!\in\!B,\quad \tau\!\in\!\bI,\quad
\fs\!\in\!C^{\i}(B\!\times\!V;\R^{\ge0}), \quad\hbox{and}\quad \fs(t,x)\ge s_0.$$
We assume that $s_0\!\ge\!2$.\\

\noindent
By the second property in~\eref{SympDef_e5a} and the first equality in~\eref{SympDef_e5},
\BE{Omdef_e2} 
\Om_{t;0,\tau}\big|_{\cN_i}=\Om_{t;i} 
~~\forall\,t\!\in\!N(\prt B),\,\tau\!\in\!\bI,\,i\!\in\!I,   \quad
\Om_{t;0,\tau}\big|_{\cN_i|_U}=\Om_{t;i}\big|_{\cN_i|_U} 
~~\forall\,t\!\in\!B,\,\tau\!\in\!\bI,\,i\!\in\!I.\EE
By the first equality in~\eref{Omdef_e2}, the openness of the nondegeneracy condition, and
the compactness of~$\bI$,
there exists a neighborhood $\cW$ of $\ov{N(\prt B)}\!\times\!V$ in $B\!\times\!V$
such that $\Om_{t;0,\tau}|_{\cN_{I'}}$  is nondegenerate over~$x$ for all 
$(t,x)\!\in\!\ov\cW$, $\tau\!\in\!\bI$,  and $I'\!\in\!\cP^*(I)$.
By the second equality in~\eref{Omdef_e2},  the openness of the nondegeneracy condition, and
the compactness of~$B\!\times\!\bI$, 
there exists a neighborhood $U'$ of $\ov{U}\!\subset\!V$ such that $\Om_{t;0,\tau}|_{\cN_{I'}}$ 
is nondegenerate over $x\!\in\!\ov{U'}$ for all
$t\!\in\!B$, $\tau\!\in\!\bI$, and $I'\!\in\!\cP^*(I)$.
By Corollary~\ref{cNbndl_crl}\ref{posinter0_it}, both nondegeneracy statements apply to
$\Om_{t;\fs,\tau}$ for all $\fs\!\in\!C^{\i}(B\!\times\!V;\R^{\ge0})$.\\

\noindent
By the choices made above, the restriction of the 2-tensor 
$\pi^*\om_t\!+\!\{\Om_{t;s,\tau}\}_{\na^{(t)}}$ to 
$T_v\cN_{I'}$, for any \hbox{$v\!\in\!\cN_{I'}|_{V-U}$} and $I'\!\in\!\cP^*(I)$, 
is nondegenerate~if
\begin{alignat*}{1}
s\!\in\!\R^{\ge0},\,\tau\!\in\!\{0,1\}, \quad&\hbox{or}\quad
s\!\ge\!s_0,\,\tau\!\in\!\bI, \quad\hbox{or}\\
s\!\in\!\R^{\ge0},\,\tau\!\in\!\bI,\, \big(t,\pi(v)\big)\!\in\!\ov\cW,\quad&\hbox{or}\quad
s\!\in\!\R^{\ge0},\,\tau\!\in\!\bI,\,\pi(v)\!\in\!\ov{U'}.
\end{alignat*}
By the openness of the nondegeneracy condition and the compactness of~$B$, $\bI$, $[0,s_0]$,
and $V\!-\!U$, 
there thus exists $\ep^*\!\in\!\R^+$ with the property~that $\wh\om_v|_{T_v\cN_{I'}}$ 
is nondegenerate whenever \hbox{$v\!\in\!\cN_{I'}|_{V-U}$}, $I'\!\in\!\cP^*(I)$,
and $\wh\om_v$ is a 2-tensor on $T_v\cN$ such~that 
\BE{nondegeb_e0}  
\big|\wh\om_v-\big(\pi^*\om_t\!+\!\{\Om_{t;s,\tau}\}_{\na^{(t)}}\big)_v\big|<\ep^*\EE
for some $t\!\in\!B$ and $s,\tau\!\in\!\R$ with
\begin{alignat}{1}\label{nondegeb_e1}
s\!\in\![0,s_0],\,\tau\!\in\!\{0,1\}, \quad&\hbox{or}\quad
s\!=\!s_0,\,\tau\!\in\!\bI, \quad\hbox{or}\\
\label{nondegeb_e2}
s\!\in\![0,s_0],\,\tau\!\in\!\bI,\,
\big(t,\pi(v)\big)\!\in\!\ov{\cW},\quad&\hbox{or}\quad
s\!\in\!\big[0,s_0],\,\tau\!\in\!\bI,\,\pi(v)\!\in\!\ov{U'}.
\end{alignat}
We assume that $\ep^*\!\le\!1$.\\

\noindent
Let $\eta_B\!:B\!\times\!V\!\lra\!\bI$ and $\eta_V\!:V\!\lra\!\bI$ be smooth functions 
such~that 
\BE{etaBetaV_e}\eta_B(t,x)=\begin{cases}
0,&\hbox{if}~t\!\in\!N(\prt B);\\
1,&\hbox{if}~(t,x)\!\not\in\!\cW;
\end{cases} \qquad
\eta_V(x)=\begin{cases}
0,&\hbox{if}~x\!\in\!U;\\
1,&\hbox{if}~x\!\not\in\!U'.
\end{cases}\EE
With the notation as in~\eref{XiBnd_e} and~\eref{nondegeb_e0}, define 
$$C^*=C_{\Om}+C_{\Om^{\circ}}+s_0C_{\eta_B\eta_V\Om^{\bu}}\,, ~~
\de=\ep^*/2C^*, ~~ \vr=\de\ne^{4s_0/\de}, ~~
\cN''=\cN(\de), ~~ \wh\cN=\cN\big(\de^4/4\vr^3).$$
We assume that $C^*\!\ge\!1$.
Let $\chi$ be as in Lemma~\ref{cutoff_lmm}.  
For any $\ep\!\in\!\R^+$, let
$$\chi_{\ep}\!:\R\lra\R^{\ge0}, \qquad
\chi_{\ep}(r)=\chi\big(\de,s_0,r/\ep\big)\,.$$
By~\eref{cutoffLmm_e1},
\BE{etadfn_e}
\chi_{\ep}(r)=\begin{cases}
s_0,&\hbox{if}~r\!\le\!\de\ep;\\
0,&\hbox{if}~r\!\ge\!\vr\ep;
\end{cases}\qquad
\begin{aligned}
&0\le \chi_{\ep}(r)\le s_0,\\
&\big|\chi_{\ep}'(r)\big|r\le\de.
\end{aligned}\EE

\vspace{.2in}

\noindent
(1) Let $\ep_1\!=\!\de/\vr$.
For $v\!\in\!\cN$, let 
\BE{newOm_e}\begin{split}
\big\{\Om_{t,\tau}^{(1)}\big\}_{\na^{(t)}}\big|_v
&=\{\Om_t\}_{\na^{(t)}}\big|_v
+\eta_{\bI;1}(\tau)\eta_B\!\big(t,\pi(v)\!\big)\eta_V\!\big(\pi(v)\!\big)
\chi_{\ep_1}\!\big(|v|\big)\big\{\Om_t^{\bu}\big\}_{\na^{(t)}}\big|_v,\\
\mu_{t,\tau}^{(1)}\big|_v
&= \frac12\eta_{\bI;1}(\tau)\eta_B\!\big(t,\pi(v)\!\big)
\eta_V\!\big(\pi(v)\!\big)\chi_{\ep_1}\!\big(|v|\big)
\io_{\ze_{\cN}(v)}\!\big\{\Om_t^{\bu}\big\}_{\!\na^{(t)}}\big|_v\,.
\end{split}\EE
Define a closed 2-form on the total space of~$\cN$~by
\BE{newOm_e1b}\wh\om_{t,\tau}^{(1)}\equiv \pi^*\om_t+
\frac12\nd\io_{\ze_{\cN}}\!\big\{\Om_{t,\tau}^{(1)}\big\}_{\!\na^{(t)}}
=\wh\om_t+\nd\mu_{t,\tau}^{(1)};\EE
the last equality holds by the first definition in~\eref{SympDef_e4}.
By~\eref{etaIdfn_e},  \eref{etaBetaV_e}, and~\eref{etadfn_e},
\begin{gather}\label{SympDef_e1617}
\mu_{t,0}^{(1)}=0~~\forall\,t\!\in\!B, \qquad
\supp\big(\mu_{\cdot,\tau}^{(1)}\big)\subset 
\big(B\!-\!N(\prt B)\big)\!\times\!\cN(\de)|_{V-U}
\quad\forall\,\tau\!\in\!\bI,\\
\label{SympDef_e15}
\mu_{t,\tau}^{(1)}\big|_{\cN(\de\ep_1)}=
\frac{s_0}{2}
\eta_B(t,\cdot)\eta_V\,\io_{\ze_{\cN}}\big\{\Om_t^{\bu}\big\}_{\na^{(t)}}\big|_{\cN(\de\ep_1)}
\qquad\forall~t\!\in\!B,~\tau\!\in\![\frac13,1].
\end{gather}
By~\eref{newOm_e}, \eref{XiBnd_e}, and~\eref{etadfn_e},   
\BE{SympDef_e15b}\begin{aligned}
\Big|\frac12\nd\io_{\ze_{\cN}}\{\Om_{t,\tau}^{(1)}\}_{\na^{(t)}}
-\{\Om_{t,\tau}^{(1)}\}_{\na^{(t)}}\Big|_v 
&\le C^*\big(|v|\!+\!\de\big) &\qquad &\forall\,v\!\in\!\cN'|_{V-U}\,,\\
&<\ep^* &\qquad &\forall\,v\!\in\!\cN(\de)|_{V-U}\,.
\end{aligned}\EE
By~\eref{newOm_e},
$$\Om_{t,\tau}^{(1)}=\Om_{t;s_{t,\tau}(v),0} \qquad\hbox{with}\quad
s_{t,\tau}(v)=\eta_{\bI;1}(\tau)\eta_B\!\big(t,\pi(v)\!\big)\eta_V\!\big(\pi(v)\!\big)
\chi_{\ep_1}\!\big(|v|\big)\in[0,s_0]\,.$$
Along with~\eref{newOm_e1b} and~\eref{SympDef_e15b}, this implies~that
$$\Big|\wh\om_{t,\tau}^{(1)}-\Big(\!\pi^*\om_t
+\{\Om_{t;s_{t,\tau}(v),0}\}_{\na^{(t)}}\!\!\Big)\Big|_v<\ep^*
\quad\forall\,(t,\tau)\!\in\!B\!\times\!\bI,\,v\!\in\!\cN(\de)|_{V-U}\,.$$
By the first case in~\eref{nondegeb_e1}, 
the restriction of~$\wh\om_{t,\tau}^{(1)}$ to $\cN(\de)_{I'}|_{V-U}$ is thus
nondegenerate for all \hbox{$(t,\tau)\!\in\!B\!\times\!\bI$} and $I'\!\in\!\cP^*(I)$.
By the last equality in~\eref{newOm_e1b} and the second statement in~\eref{SympDef_e1617}, 
this is also the case for the restriction of~$\wh\om_{t,\tau}^{(1)}$ to 
$(\cN'_{I'}\!-\!\cN(\de))|_{V-U}$.\\

\noindent
(2) Let $\ep_2\!=\!\de\ep_1/2\vr$; thus, $\ov{\cN(\vr\ep_2)}\!\subset\!\cN(\de\ep_1)$.
For $v\!\in\!\cN$, let 
\BE{newOm_e2}\begin{split}
\big\{\Om_{t,\tau}^{(2)}\big\}_{\na^{(t)}}\big|_v
&=\big\{\Om_{t;s_0\eta_B\eta_V}\big\}_{\na^{(t)}}\big|_v
+\frac{1}{s_0}\eta_{\bI;2}(\tau)
\chi_{\ep_2}\!\big(|v|\big)\big\{\Om_t^{\circ}\big\}_{\na^{(t)}}\big|_v,\\
\mu_{t,\tau}^{(2)}\big|_v&=\frac1{2s_0}  \eta_{\bI;2}(\tau)
\chi_{\ep_2}\!\big(|v|\big)
\io_{\ze_{\cN}(v)}\big\{\Om_t^{\circ}\big\}_{\!\na^{(t)}}\big|_v\,.
\end{split}\EE
Define a closed 2-form on the total space of~$\cN$ by
\BE{newOm_e2b}
\wh\om_{t,\tau}^{(2)} \equiv \pi^*\om_t+
\frac12\nd\io_{\ze_{\cN}}\big\{\Om_{t,\tau}^{(2)}\big\}_{\!\na^{(t)}}
=\wh\om_t+
\frac{s_0}{2}\nd\big(\eta_B\eta_V\,\io_{\ze_{\cN}}\big\{\Om_t^{\bu}\big\}_{\na^{(t)}}\big)
+\nd\mu_{t,\tau}^{(2)}\,;\EE
the last equality holds by~\eref{Omdef_e} and the first equations in~\eref{SympDef_e4}.
By~\eref{etaIdfn_e}, the second equation in~\eref{SympDef_e5a}, and~\eref{etadfn_e},
\begin{gather}\label{SympDef_e2627}
\mu_{t,\tau}^{(2)}=0~~\forall\,t\!\in\!B,\,\tau\!\in\![0,\frac13], \quad
\supp\big(\mu_{\cdot,\tau}^{(2)}\big)\subset 
\big(B\!-\!N(\prt B)\big)\!\times\!\cN(\vr\ep_2)|_{V-U}
~~\forall\,t\!\in\!B,\,\tau\!\in\!\bI,\\
\label{SympDef_e25}
\mu_{t,\tau}^{(2)}\big|_{\cN(\de\ep_2)}
=\frac12 \io_{\ze_{\cN}}\!\big\{\Om_t^{\circ}\big\}_{\!\na^{(t)}}\big|_{\cN(\de\ve_2)}
\quad\forall\,\tau\!\in\![\frac23,1].
\end{gather}
By~\eref{newOm_e2}, \eref{XiBnd_e}, and~\eref{etadfn_e},    
\BE{SympDef_e25b}
\Big|\frac12\nd\io_{\ze_{\cN}}\{\Om_{t,\tau}^{(2)}\}_{\na^{(t)}}
-\{\Om_{t,\tau}^{(2)}\}_{\na^{(t)}}\Big|_v
\le C^*\big(|v|\!+\!\de\big) \quad \forall\,v\!\in\!\cN'|_{V-U}\,.\EE
By~\eref{newOm_e2},
$$\Om_{t,\tau}^{(2)}=\Om_{t;s_t(v),\tau_{t,\tau}'(v)} \quad\hbox{with}~~
s_t(v)=s_0\eta_B\!\big(t,\pi(v)\!\big)\eta_V\!\big(\pi(v)\!\big),~
\tau_{t,\tau}'(v)=\eta_{\bI;2}(\tau)\frac{\chi_{\ep_2}(|v|)}{s_0}\in\bI.$$
Along with~\eref{newOm_e2b} and~\eref{SympDef_e25b}, this implies~that
$$\Big|\wh\om_{t,\tau}^{(2)}-\Big(\pi^*\om_t
+\{\Om_{t;s_t(v),\tau_{t,\tau}'(v)}\}_{\na^{(t)}}\Big)\Big|_v
<\ep^* \quad\forall\,(t,\tau)\!\in\!B\!\times\!\bI,\,v\!\in\!\cN(\de)|_{V-U}\,.$$
By~\eref{etaBetaV_e}, 
$$s_t(v)\in[0,s_0]~\forall\,v\!\in\!\cN, \qquad
s_t(v)=s_0~~\hbox{if}~\big(t,\pi(v)\big)\!\not\in\!\cW
~\hbox{and}~\pi(v)\!\not\in\!U'.$$
By the last two displayed statements, \eref{nondegeb_e2},  and 
the second case in~\eref{nondegeb_e1},
the restriction of~$\wh\om_{t,\tau}^{(2)}$ to $\cN(\de)_{I'}|_{V-U}$
is then 
nondegenerate for all \hbox{$(t,\tau)\!\in\!B\!\times\!\bI$} and $I'\!\in\!\cP^*(I)$.
By the last equality in~\eref{newOm_e2b}, the second statement in~\eref{SympDef_e2627}, 
and~\eref{SympDef_e15},
this is also the case for the restriction of~$\wh\om_{t,\tau}^{(2)}$ to 
$\cN(\de\ep_1)_{I'}|_U$.\\

\noindent
(3) Let $\ve_3\!=\!\de\ve_2/2\vr$; thus, $\ov{\cN(\vr\ep_3)}\!\subset\!\cN(\de\ep_2)$. 
We now reduce the diagonal part~of 
\BE{Omdiag_e}
\Om_{t;s_0\eta_B\eta_V,1}\equiv \Om_t\!+\!s_0\eta_B\eta_V\Om_t^{\bu}\!+\!\Om_t^{\circ}
=\Om_t'\!+\!s_0\eta_B\eta_V\Om_t^{\bu}\EE
back to~$\Om_t^{\bu}\!=\!\Om_t'^{\,\bu}$.
For $v\!\in\!\cN$, let 
\BE{newOm_e3}\begin{split}
\big\{\Om_{t,\tau}^{(3)}\big\}_{\na^{(t)}}\big|_v
&=\big\{\Om_{t;s_0\eta_B\eta_V,1}\big\}_{\na^{(t)}}\big|_v
-\eta_{\bI;3}(\tau)\eta_B\!\big(t,\pi(v)\!\big)\eta_V\!\big(\pi(v)\!\big)
\chi_{\ep_3}\!\big(|v|\big)\big\{\Om_t^{\bu}\big\}_{\!\na^{(t)}}\big|_v,\\
\mu_{t,\tau}^{(3)}\big|_v&=-\frac12
\eta_{\bI;3}(\tau)\eta_B\!\big(t,\pi(v)\!\big)\eta_V\!\big(\pi(v)\!\big)
\chi_{\ep_3}\!\big(|v|\big)
\,\io_{\ze_{\cN}(v)}\!\big\{\Om_t^{\bu}\big\}_{\!\na^{(t)}}\big|_v\,.
\end{split}\EE
Define a closed 2-form on the total space of~$\cN$ by
\BE{newOm_e3b}\begin{split}
\wh\om_{t,\tau}^{(3)} &\equiv \pi^*\om_t+
\frac12\nd\io_{\ze_{\cN}}\{\Om_{t,\tau}^{(3)}\}_{\na^{(t)}}\\
&=\wh\om_t+\frac12\nd
\Big(s_0\eta_B\eta_V\io_{\ze_{\cN}}\big\{\Om_t^{\bu}\big\}_{\na^{(t)}}
+\io_{\ze_{\cN}}\big\{\Om_t^{\circ}\big\}_{\na^{(t)}}\!\!\bigg)
+\nd\mu_{t,\tau}^{(3)}\,;
\end{split}\EE
the last equality holds by~\eref{Omdiag_e} and 
the first definition in~\eref{SympDef_e4}.
By~\eref{etaIdfn_e}, \eref{etaBetaV_e}, and~\eref{etadfn_e}, 
\begin{gather}\label{SympDef_e3637}
\mu_{t,\tau}^{(3)}=0~~\forall\,t\!\in\!B,\,\tau\!\in\![0,\frac23], \quad
\supp\big(\mu_{\cdot,\tau}^{(3)}\big)\subset 
\big(B\!-\!N(\prt B)\big)\!\times\!\cN(\vr\ep_3)|_{V-U}
~~\forall\,\tau\!\in\!\bI,\\
\label{SympDef_e35}
\mu_{t,1}^{(3)}\big|_{\cN(\de\ep_3)}
=-\frac{s_0}2\eta_B(t,\cdot)\eta_V
\io_{\ze_{\cN}}\big\{\Om_t^{\bu}\big\}_{\na^{(t)}}\big|_{\cN(\de\ep_3)}
~~\forall\,t\!\in\!B.
\end{gather}
By~\eref{newOm_e3}, \eref{Omdiag_e}, \eref{XiBnd_e}, and~\eref{etadfn_e},   
\BE{SympDef_e35b}
\Big|\frac12\nd\io_{\ze_{\cN}}\{\Om_{t,\tau}^{(3)}\}_{\na^{(t)}}
-\{\Om_{t,\tau}^{(3)}\}_{\na^{(t)}}\Big|_v 
\le C^*\big(|v|\!+\!\de\big)\quad\forall\,\,v\!\in\!\cN'|_{V-U}\,.\EE
By~\eref{newOm_e3} and~\eref{Omdiag_e},
$$\Om_{t,\tau}^{(3)}=\Om_{t;s_{t,\tau}'(v),1} \qquad\hbox{with}\quad
s_{t,\tau}'(v)=\eta_B\!\big(t,\pi(v)\!\big)\eta_V\!\big(\pi(v)\!\big)
\big(s_0\!-\!\eta_{\bI;3}(\tau)\chi_{\ep_3}\!\big(|v|\big)\big)\in[0,s_0]\,.$$
Along with~\eref{newOm_e3b} and~\eref{SympDef_e35b}, this implies~that
$$\Big|\wh\om_{t,\tau}^{(3)}-\Big(\!\pi^*\om_t
+\{\Om_{t;s_{t,\tau}'(v),1}\}_{\na^{(t)}}\!\!\Big)\Big|_v<\ep^*
\quad\forall\,(t,\tau)\!\in\!B\!\times\!\bI,\,v\!\in\!\cN(\de)|_{V-U}\,.$$
By the first case in~\eref{nondegeb_e1}, 
the restriction of~$\wh\om_{t,\tau}^{(3)}$ to $\cN(\de)_{I'}|_{V-U}$
is then
nondegenerate for all \hbox{$(t,\tau)\!\in\!B\!\times\!\bI$} and $I'\!\in\!\cP^*(I)$.
By the last equality in~\eref{newOm_e3b}, the second statement in~\eref{SympDef_e3637}, 
\eref{SympDef_e15}, and~\eref{SympDef_e25},
this is also the case for the restriction of~$\wh\om_{t,\tau}^{(3)}$ to 
$\cN(\de\ep_2)_{I'}|_U$.\\

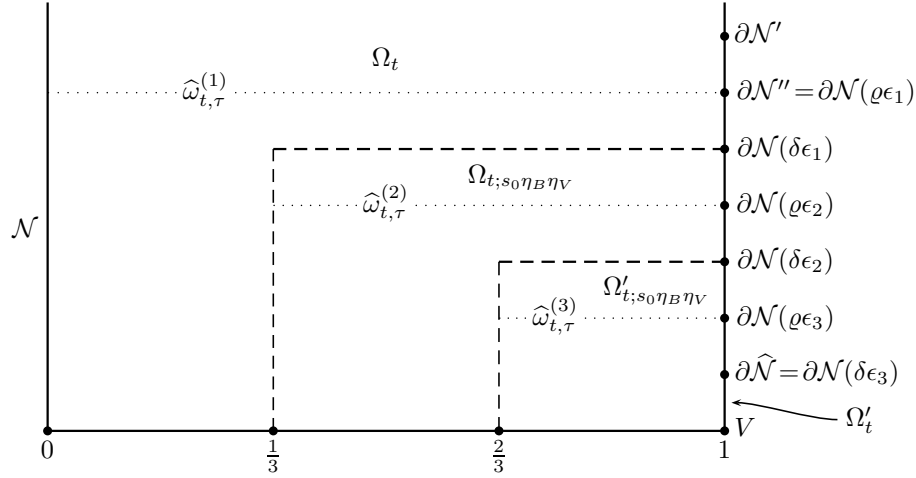
\begin{figure}
\begin{pspicture}(-4,-3)(11,3.2)
\psset{unit=.3cm}
\psline[linewidth=.1](27,10)(27,-9)\psline[linewidth=.1](-3,10)(-3,-9)
\psline[linewidth=.1](-3,-9)(27,-9)
\psline[linewidth=.07,linestyle=dotted](-3,6)(2.7,6)
\psline[linewidth=.07,linestyle=dotted](5.3,6)(27,6)
\rput(12,7.5){\sm{$\Om_t$}}\rput(4,6.1){\sm{$\wh\om_{t,\tau}^{(1)}$}}
\psline[linewidth=.1,linestyle=dashed](7,3.5)(27,3.5)
\psline[linewidth=.07,linestyle=dashed](7,-9)(7,3.5)
\psline[linewidth=.07,linestyle=dotted](7,1)(10.7,1)
\psline[linewidth=.07,linestyle=dotted](13.3,1)(27,1)
\rput(18,2.25){\sm{$\Om_{t;s_0\eta_B\eta_V}$}}\rput(12,1.1){\sm{$\wh\om_{t,\tau}^{(2)}$}}
\psline[linewidth=.1,linestyle=dashed](17,-1.5)(27,-1.5)
\psline[linewidth=.07,linestyle=dashed](17,-9)(17,-1.5)
\psline[linewidth=.07,linestyle=dotted](17,-4)(18.2,-4)
\psline[linewidth=.07,linestyle=dotted](20.8,-4)(27,-4)
\rput(19.5,-3.9){\sm{$\wh\om_{t,\tau}^{(3)}$}}
\rput(24,-2.75){\sm{$\Om_{t;s_0\eta_B\eta_V}'$}}
\pscircle*(7,-9){.2}\rput(7,-10.1){\sm{$\frac13$}}
\pscircle*(17,-9){.2}\rput(17,-10.1){\sm{$\frac23$}}
\pscircle*(-3,-9){.2}\rput(-3,-9.8){\sm{$0$}}\rput(27,-9.8){\sm{$1$}}
\pscircle*(27,6){.2}\pscircle*(27,8.5){.2}
\pscircle*(27,3.5){.2}\pscircle*(27,1){.2}\pscircle*(27,-1.5){.2}
\pscircle*(27,-4){.2}\pscircle*(27,-6.5){.2}\pscircle*(27,-9){.2}
\rput(28.5,8.7){\sm{$\prt\cN'$}}
\rput(31.5,6){\sm{$\prt\cN''\!=\!\prt\cN(\vr\ep_1)$}}\rput(29.7,3.5){\sm{$\prt\cN(\de\ep_1)$}}
\rput(29.7,1){\sm{$\prt\cN(\vr\ep_2)$}}\rput(29.7,-1.5){\sm{$\prt\cN(\de\ep_2)$}}
\rput(29.7,-4){\sm{$\prt\cN(\vr\ep_3)$}}
\rput(31.1,-6.3){\sm{$\prt\wh\cN\!=\!\prt\cN(\de\ep_3)$}}
\rput(27.9,-8.8){\sm{$V$}}
\rput(33,-8.5){\sm{$\Om_t'$}}\rput(-4,0){\sm{$\cN$}}
\pnode(27.3,-7.75){B}\pnode(32,-8.5){A}
\nccurve[linewidth=.07,angleA=180,angleB=0,ncurv=.5]{->}{A}{B}
\end{pspicture}
\caption{The patched families $(\wh\om_{t,\tau})_{t\in B,\tau\in\bI}$ of closed 2-forms
on~$\cN$ and $(\Om_{t,\tau})_{t\in B,\tau\in\bI}$ of fiberwise 2-forms on~$\cN$
so that $\wh\om_{t,\tau}\!=\!\pi^*\om_t\!+\!
\frac12\nd\io_{\ze_{\cN}}\{\Om_{t,\tau}\}_{\na^{(t)}}$
in the indicated regions.}
\label{SympDef_fig}
\end{figure}

\noindent
We define smooth families of 1-forms and 2-forms on the total spaces of~$\cN$ and~$\cN_{\prt}$ by
\begin{alignat*}{3}
\mu_{t,\tau}&=\mu_{t,\tau}^{(1)}\!+\!\mu_{t,\tau}^{(2)}
\!+\!\mu_{t,\tau}^{(3)}, &\quad
\big(\mu_{t,\tau;i}\big)_{i\in I} &=\big(\mu_{t,\tau}\big|_{\cN_i}\big)_{i\in I}
&\qquad&\forall\,t\!\in\!B,\,\tau\!\in\!\bI,\\
\wh\om_{t,\tau}&=\wh\om_t\!+\!\nd\mu_{t,\tau},&\quad
\big(\wh\om_{t,\tau;i}\big)_{i\in I} &=\big(\wh\om_{t,\tau}\big|_{\cN_i}\big)_{i\in I}
&\qquad&\forall\,t\!\in\!B,\,\tau\!\in\!\bI\,.
\end{alignat*}
By~\eref{SympDef_e1617}, \eref{SympDef_e2627}, and~\eref{SympDef_e3637}, 
$$\mu_{t,0}=0~~\forall\,t\!\in\!B, \qquad
\supp\big(\mu_{\cdot,\tau}\big)\subset 
\big(B\!-\!N(\prt B)\big)\!\times\!\big(\cN''|_{V-U}\big)
~~\forall\,\tau\!\in\!\bI.$$
By~\eref{SympDef_e2627}, \eref{SympDef_e3637}, \eref{SympDef_e15}, and~\eref{SympDef_e25},  
\BE{SympDef_e39}
\wh\om_{t,\tau}\big|_v=\begin{cases}
\wh\om_{t,\tau}^{(1)}|_v,&\hbox{if}~
(\tau,v)\!\in\!\bI\!\times\!\cN'\!-\![\frac13,1]\!\times\!\cN(\vr\ep_2);\\
\wh\om_{t,\tau}^{(2)}|_v,&\hbox{if}~
(\tau,v)\!\in\![\frac13,1]\!\times\!\cN(\de\ep_1)\!-\![\frac23,1]\!\times\!\cN(\vr\ep_3);\\
\wh\om_{t,\tau}^{(3)}|_v,&\hbox{if}~
(\tau,v)\!\in\![\frac23,1]\!\times\!\cN(\de\ep_2).
\end{cases}\EE
Along with the observations at the end of each step (1)-(3) of the construction,
this implies that $\wh\om_{t,\tau}|_{\cN_{I'}'}$ is nondegenerate for all 
$I'\!\in\!\cP^*(I)$ for all $(t,\tau)\!\in\!B\!\times\!\bI$.
By~\eref{SympDef_e39}, \eref{newOm_e1b}, \eref{newOm_e2b}, and~\eref{newOm_e3b}, 
$$\wh\om_{t,\tau}|_V=\om_t \qquad\forall~(t,\tau)\!\in\!B\!\times\!\bI\,.$$
By the last case in~\eref{SympDef_e39}, \eref{newOm_e3b}, \eref{SympDef_e35}, 
and the last equality in~\eref{SympDef_e4},
$\wh\om_{t,1}|_{\wh\cN}\!=\!\wh\om_t'|_{\wh\cN}$ for all $t\!\in\!B$.
Along with the last equation in~\eref{SympDef_e5},
the conclusions in the paragraph imply that the smooth family 
$(\mu_{t,\tau;i})_{t\in B,\tau\in\bI,i\in I}$ of
1-forms on~$\cN_{\prt}$ satisfies all requirements of the proposition. 
\end{proof}

\subsection{Deformations of arbitrary structures}
\label{NonLinDeform_subs}

\noindent
We continue with the notation introduced in~\eref{cNsplit_e} and~\eref{cNsubdfn_e}.
By Lemma~\ref{SympNeigh_lmm2} below, an arbitrary symplectic structure $(\wt\om_{t;i})_{i\in I}$
on a neighborhood~$\cN_{\prt}'$ of~$V$ in~$\cN_{\prt}$ 
can be deformed to  a standard one, $(\wh\om_{t;i})_{i\in I}$ as in~\eref{ombund_e2},
on a smaller neighborhood of~$V$.
As with Proposition~\ref{SympDef_prp}, the forms are kept fixed outside
of a neighborhood~$\cN_{\prt}''$ of~$V$.
By definition, the original symplectic forms $\wt\om_{t;i}$ on~$\cN_i'$
agree along their overlaps, i.e.~on $\cN_{i_1i_2}'$.

\begin{lmm}\label{SympNeigh_lmm2}
Let $U\!\subset\!V$, $I$, $\cN$, $N(\prt B)\!\subset\!B$, 
and $\cN'$ be as in Theorem~\ref{SympDefVB_thm}.
Suppose 
\begin{enumerate}[label=$\bullet$,leftmargin=*]

\item $(\wt\om_{t;i})_{t\in B,i\in I}$ and $(\wt\om_{t;i}')_{t\in B,i\in I}$ 
are smooth families of symplectic structures and 
of closed 2-forms,  respectively, on $\cN_{\prt}'$ 
such~that 
\BE{SympExtendLmm_e1}\begin{split}
&\hspace{1in}\big(\wt\om_{t;i}|_{T\cN_i|_V}\big)_{i\in I}=
\big(\wt\om_{t;i}'|_{T\cN_i|_V}\big)_{i\in I} ~~\forall\,t\!\in\!B,\\
&\big(\wt\om_{t;i}\big)_{i\in I}=\big(\wt\om_{t;i}'\big)_{i\in I}
~~\forall\,t\!\in\!N(\prt B), \qquad
\big(\wt\om_{t;i}|_{\cN_i'|_U}\big)_{i\in I}
=\big(\wt\om_{t;i}'|_{\cN_i'|_U}\big)_{i\in I}~~\forall\,t\!\in\!B;
\end{split}\EE

\item  $K\!\subset\!V$ is a compact subset and $\cK\!\subset\!V$ 
is an open neighborhood of~$K$.

\end{enumerate}
Then there exist neighborhoods $\wt\cN\!\subset\!\cN''$ of $V\!\subset\!\cN'$ 
such that $\ov{\cN''}\!\subset\!\cN'$
and a smooth family $(\mu_{t,\tau;i})_{t\in B,\tau\in\bI,i\in I}$ of
1-forms on~$\cN_{\prt}$ such~that
\BE{SympNeigh_e0c}\big(\wt\om_{t,\tau;i}\big)_{i\in I}\equiv
\big(\wt\om_{t;i}\!+\!\nd\mu_{t,\tau;i}|_{\cN_i'}\big)_{i\in I}\EE
is a symplectic structure on~$\cN_{\prt}'$ for all $(t,\tau)\!\in\!B\!\times\!\bI$ and
\BE{SympNeighLmm_e0}\begin{split}
\mu_{t,0;i}=0, &\quad  
\wt\om_{t,\tau;i}\big|_{T\cN_i|_V}=\wt\om_{t;i}\big|_{T\cN_i|_V},\quad
\wt\om_{t,1;i}|_{\wt\cN_i|_K}=\wt\om_{t;i}'|_{\wt\cN_i|_K},\\
&\supp\big(\mu_{\cdot,\tau;i}\big)\subset 
\big(B\!-\!N(\prt B)\big)\!\times\!\cN_i''|_{\cK-U}
\end{split}\EE
for all $t\!\in\!B$, $\tau\!\in\!\bI$, and $i\!\in\!I$.
\end{lmm}

\begin{proof}
For each $\tau\!\in\!\R$, let
$$m_{\tau}\!:\cN\lra\cN, \qquad v\lra\tau v,$$
be the scalar multiplication map;
it preserves the subbundles $\cN_i\!\subset\!\cN$.
For each $t\!\in\!B$ and $i\!\in\!I$, define
$$\vp_{t;i}=\wt\om_{t;i}'\!-\!\wt\om_{t;i}, \qquad
\mu_{t;i}\big|_v=\int_0^1\!\!
m_{\tau}^*\big\{\vp_{t;i}(\tau^{-1}\ze_{\cN},\cdot)\big\}\nd\tau\,.$$
By the second half of the proof of \cite[Lemma~3.14]{MS1},
\BE{sitdfn_e}\wt\om_{t;i}'=\wt\om_{t;i}+\nd\mu_{t;i}\,.\EE
By \eref{SympExtendLmm_e1},
\BE{zettau_e}
\mu_{t;i}|_V,d\mu_{t;i}\big|_{T\cN_i|_V}=0~~\forall\,t\!\in\!B,\quad
\mu_{t;i}=0~~\forall\,t\!\in\!N(\prt B), \quad
\mu_{t;i}\big|_{\cN_i|_U}=0~~\forall\,t\!\in\!B.\EE
Since $(\wt\om_{t;i})_{t\in B,i\in I}$ and $(\wt\om_{t;i}')_{t\in B,i\in I}$ are smooth families
of 2-forms on~$\cN_{\prt}'$,
$(\mu_{t;i})_{t\in B,i\in I}$ is a smooth family of 1-forms on~$\cN_{\prt}'$.\\

\noindent
Let $|\cdot|$ be a norm on $\cN$.
For $\de\!\in\!\R^+$, let 
$$\cN(\de)= \big\{v\!\in\!\cN\!:\,|v|\!<\!\de\big\}\,.$$
Since $B$ is compact, we can choose the norm on~$\cN$ so that 
$\cN(4)\!\subset\!\cN'$.
Choose smooth functions
\begin{alignat}{2}
\notag
\eta_{\R}&:\R\lra\bI, &\quad \eta_V&:V\lra\bI \qquad\hbox{s.t.}\\
\eta_{\R}(r)&=\begin{cases} 1,&\hbox{if}~r\!\le\!1;\\
\label{SympNeigh_e15}
0,&\hbox{if}~r\!\ge\!2; \end{cases} &\qquad
\eta_V(x)&=\begin{cases} 1,&\hbox{if}~x\!\in\!K;\\
0,&\hbox{if}~x\!\not\in\!\cK. \end{cases}
\end{alignat}   
For $\de\!\in\!(0,1)$, $t\!\in\!B$, $\tau\!\in\!\bI$, and $i\!\in\!I$,  let 
$$\mu_{t,\tau;i}^{(\de)}(v)
=\begin{cases}\tau\eta_{\R}\big(|v|/\de\big)\eta_V\big(\!\pi(v)\!\big)
\mu_{t;i}(v),&\hbox{if}\,v\!\in\!\cN_i';\\
0,&\hbox{if}\,v\!\in\!\cN_i\!-\!\ov{\cN(2)}.
\end{cases}$$
By \eref{sitdfn_e}-\eref{SympNeigh_e15}, 
\BE{SympNeigh_e17}\begin{split}
&\mu_{t,0;i}^{(\de)}=0, \quad
d\mu_{t,\tau;i}^{(\de)}\big|_{T\cN_i|_V}=0,\quad
\wt\om_{t;i}|_{\cN(\de)_i|_K}\!+\!\nd\mu_{t,1;i}^{(\de)}|_{\cN(\de)_i|_K}
=\wt\om_{t;i}'|_{\cN(\de)_i|_K}, \\
&\hspace{.8in}\supp\big(\mu_{\cdot,\tau;i}^{(\de)}\big)\subset 
\big(B\!-\!N(\prt B)\big)\!\times\!\cN(2\de)_i\big|_{\cK-U}\,.
\end{split}\EE
Thus, the smooth family $(\mu_{t,\tau;i}^{(\de)})_{t\in B,\tau\in\bI,i\in I}$ of
1-forms on~$\cN_{\prt}$ satisfies~\eref{SympNeighLmm_e0}
with $\wt\cN\!=\!\cN(\de)$, $\cN''\!=\!\cN(2\de)$, and~$\mu$ replaced by~$\mu^{(\de)}$.\\

\noindent
It remains to verify that \eref{SympNeigh_e0c} with $\mu$ replaced by~$\mu^{(\de)}$
is a symplectic structure on~$\cN_{\prt}'$ for all $(t,\tau)\!\in\!B\!\times\!\bI$
and $\de\!\in\!(0,1)$ sufficiently small.
We can assume that $\ov\cK\!\subset\!V$ is compact.
Since  $B$ is also compact, there exists $\ep\!\in\!\R^+$ with the property~that 
$\wt\om_v|_{T_v\cN_{I'}}$ is nondegenerate whenever 
\hbox{$v\!\in\!\cN(2)_{I'}|_{\cK}$}, $I'\!\in\!\cP_i(I)$,
and $\wt\om_v$ is a 2-tensor on $T_v\cN_i$ such~that 
$$\big|\wt\om_v-\wt\om_{t;i}|_v\big|<\ep$$
for some $t\!\in\!B$.
Since $\vp_{t;i}|_V\!=\!0$ and~$B\!\times\!\ov{\cN(2)}|_{\ov\cK}$ is compact, 
there exists $C\!\in\!\R^+$ such~that 
$$\big|\nd\mu_{t,\tau;i}^{(\de)}\big|_v\le C\big(\de^{-1}|v|^2\!+\!|v|\big)\le 6C\de
\quad\forall\,v\!\in\!\cN(2\de)|_{\cK},\,\de\!\in\!(0,1).$$
By the last two inequalities, 
\eref{SympNeigh_e0c} with $\mu$ replaced by~$\mu^{(\de)}$
is a symplectic structure on~$\cN_{\prt}'|_{\cK}$ for all $(t,\tau)\!\in\!B\!\times\!\bI$
and $\de\!\in\!(0,1)$ sufficiently small.
It is a symplectic structure on~$\cN_{\prt}'|_{V-\cK}$ for all $\de\!\in\!(0,1)$
because $\wt\om_{t,\tau;i}\!=\!\wt\om_{t;i}$ over~$V\!-\!\cK$.
\end{proof}

\begin{proof}[{\bf{\emph{Proof of Theorem \ref{SympDefVB_thm}}}}]
Let $\{K_{\ell}^{\circ}\}_{\ell\in\Z^+}$ be an open cover of~$V$ such that 
the closure~$K_{\ell}$ of~$K_{\ell}^{\circ}$ is compact and contained in~$K_{\ell+1}^{\circ}$
for every $\ell\!\in\!\Z^+$.
We inductively construct sequences $(\mu^{(\ell)}_{t,\tau,i})_{\ell\in\Z^+}$ of families of 
1-forms on~$\cN_{\prt}$ and $\wh\cN_{(\ell)}\!\subset\!\cN_{(\ell)}''$ of neighborhoods 
of $V\!\subset\!\cN$
so that each $\mu^{(\ell)}_{t,\tau,i}$ is supported in 
$\cN_{(\ell)}''|_{K_{\ell+1}-K_{\ell-1}}$ and for each $\ell^*\!\in\!\Z^+$
the sum of $\mu^{(\ell)}_{t,1,i}$ with $\ell\!\in\![\ell^*]$ satisfies the third condition
in~\eref{SympDefThm_e0} with $\wh\cN_i$ replaced by $(\wh\cN_{(\ell^*)})_i|_{K_{\ell^*}}$.
We then take $\mu_{t,\tau,i}$ and $\wh\cN$ to be the sum of all 1-forms $\mu^{(\ell)}_{t,\tau,i}$ 
and the union of the open sets $\wh\cN_{(\ell)}|_{K_{\ell}^{\circ}}$, respectively.
We use Lemma~\ref{SympNeigh_lmm2} followed by Proposition~\ref{SympDef_prp}
to construct  $\mu^{(\ell)}_{t,\tau,i}$ for each $\ell\!\in\!\Z^+$.\\

\noindent
Define
$$K_0=\eset,\quad \wh\cN_{(0)}=\cN', \quad 
U_{\ell}=K_{\ell}^{\circ}\!\cup\!U~~\forall\,\ell\!\in\!\Z^{\ge0},\quad
\wt\om_{t,1;i}^{(0)}=\wt\om_{t;i}~~\forall\,t\!\in\!B,\,i\!\in\!I.$$
For each $t\!\in\!B$, let $\om_t$ be the symplectic form on $V$ determined
by the symplectic structure $(\wt\om_{t;i})_{i\in I}$ on~$\cN_{\prt}'$ and
$$\big(\Om_{t;i}\big)_{i\in I},\big(\Om_{t;i}^{\bu})_{i\in I}\in  
\Symp_V^+\big(\cN_{\prt}\big)$$
be the fiberwise symplectic structures  on~$\cN_{\prt}$ 
determined by $(\wt\om_{t;i})_{i\in I}$ via \eref{cNXVsymp_e} and~\eref{diagpart_e}.\\

\noindent
Suppose $\ell^*\!\in\!\Z^+$ and for every $\ell\!\in\![\ell^*\!-\!1]$
we have constructed 
\begin{enumerate}[label=($\cN\mu\arabic*$),leftmargin=*]

\item  neighborhoods $\wh\cN_{(\ell)}\!\subset\!\cN_{(\ell)}''$
of~$V\!\subset\!\cN'$ such that $\ov{\cN_{(\ell)}''}\!\subset\!\wh\cN_{(\ell-1)}$, 

\item a smooth family $(\mu_{t,\tau;i}^{(\ell)})_{t\in B,\tau\in\bI,i\in I}$ 
of 1-forms on~$\cN_{\prt}$ such~that 
\BE{SympDef_e45a}
\big(\wt\om_{t,\tau;i}^{(\ell)}\big)_{i\in I}\equiv
\big(\wt\om_{t,1;i}^{(\ell-1)}\!+\!\nd\mu_{t,\tau;i}^{(\ell)}|_{\cN_i'}\big)_{i\in I}\EE
is a symplectic structure on~$\cN_{\prt}'$ for all $(t,\tau)\!\in\!B\!\times\!\bI$,
\BE{SympDef_e45}\begin{split}
&\mu_{t,0;i}^{(\ell)}=0, \quad
\wt\om_{t,\tau;i}^{(\ell)}|_V=\om_t\,,\quad
\wt\om_{t,1;i}^{(\ell)}\big|_{(\wh\cN_{(\ell)})_i|_{K_{\ell}}}
=\wh\om_{t;i}^{\bu}\big|_{(\wh\cN_{(\ell)})_i|_{K_{\ell}}},\\
&\hspace{.3in}\supp\big(\mu_{\cdot,\tau;i}^{(\ell)}\big)\subset 
\big(\!B\!-\!N(\prt B)\!\big)\!\times\!\big(\cN_{(\ell)}''\big)_i\big|_{K_{\ell+1}^{\circ}-U_{\ell-1}}
\end{split}\EE
for all $t\!\in\!B$, $\tau\!\in\!\bI$, and $i\!\in\!I$, and 
the family $(\Om_{t,1;i}^{(\ell)})_{t\in B,i\in I}$ of
the fiberwise symplectic structures 
on~$\cN_{\prt}$ determined by $(\wt\om_{t,1;i}^{(\ell)})_{i\in I}$ 
via \eref{cNXVsymp_e} satisfies
\BE{SympDef_e45b} 
\big(\Om_{t,1;i}^{(\ell)\,\bu}\big)_{t\in B,i\in I}=
\big(\Om_{t;i}^{\bu}\big)_{t\in B,i\in I}\,.\EE
\end{enumerate}

\vspace{.1in}

\noindent
By~\eref{SympDef_e45a} and induction,
\BE{SympDef_e45g} 
\wt\om_{t,1;i}^{(\ell^*-1)}=
\wt\om_{t;i}\!+\nd\!\sum_{\ell=1}^{\ell^*-1}\mu_{t,1;i}^{(\ell)}\big|_{\cN_i'}
\qquad\forall\,t\!\in\!B,\,i\!\in\!I\,.\EE
By \eref{SympDefVB_e0}, the last two properties in~\eref{SympDef_e45},
and induction,
\BE{SympDef_e45f}\begin{aligned} 
\big(\wt\om_{t,1;i}^{(\ell^*-1)}\big)_{i\in I}&=\big(\wh\om_{t;i}^{\bu}\big)_{i\in I}
&\quad&\forall\,t\!\in\!N(\prt B), \\
\wt\om_{t,1;i}^{(\ell^*-1)}|_{(\wh\cN_{(\ell-1)})_i|_{U_{\ell-1}}}
&=\wh\om_{t;i}^{\bu}|_{(\wh\cN_{(\ell-1)})_i|_{U_{\ell-1}}}
&\quad&\forall\,\ell\!\in\![\ell^*],\,t\!\in\!B.
\end{aligned}\EE
Along with the second property in~\eref{SympDef_e45}, this implies that 
\BE{SympDef_e45d} 
\big(\Om_{t,1;i}^{(\ell^*-1)}\big)_{i\in I}
=\big(\Om_{t;i}^{\bu}\big)_{i\in I}~~\forall\,t\!\in\!N(\prt B),\quad
\big(\Om_{t,1;i}^{(\ell^*-1)}|_{U_{\ell^*-1}}\big)_{i\in I}
=\big(\Om_{t;i}^{\bu}|_{U_{\ell^*-1}}\big)_{i\in I}~~\forall\,t\!\in\!B.\EE

\vspace{.2in}

\noindent
Let $K^{\circ}$ be an open neighborhood of $K_{\ell^*}\!\subset\!V$ so that
its closure~$K$ is contained in~$K_{\ell^*+1}^{\circ}$. 
For $t\!\in\!B$ and $i\!\in\!I$, let 
$$\wt\om_{t;i}'=\pi^*\om_t|_{\cN_i}\!+\!
\frac12\nd\io_{\ze_{\cN}}\{\Om_{t,1;i}^{(\ell^*-1)}\}_{\na^{(t)}}\,.$$
By the $\ell\!=\!\ell^*$ case of~\eref{SympDef_e45f},
the three conditions in~\eref{SympExtendLmm_e1} with~$U$, $\cN'$, and $\wt\om_{t;i}$
replaced by~$U_{\ell^*-1}$, $\wh\cN_{(\ell^*-1)}$, and
$\wt\om_{t,1;i}^{(\ell^*-1)}$, respectively, are satisfied.
By Lemma~\ref{SympNeigh_lmm2} applied with $\cK\!=\!K_{\ell^*+1}^{\circ}$, 
there thus exist neighborhoods $\wt\cN\!\subset\!\cN''_{(\ell^*)}$ of $V\!\subset\!\cN'$ 
such that $\ov{\cN''_{(\ell^*)}}\!\subset\!\wh\cN_{(\ell^*-1)}$ and a smooth family 
$(\mu_{t,\tau;i})_{t\in B,\tau\in\bI,i\in I}$ of 1-forms on~$\cN_{\prt}$ 
such~that
\BE{SympNeigh_e48a}
\big(\wt\om_{t,\tau;i}\big)_{i\in I}\equiv
\big(\wt\om_{t,1;i}^{(\ell^*-1)}|_{(\wh\cN_{(\ell^*-1)})_i}
\!+\!\nd\mu_{t,\tau;i}|_{(\wh\cN_{(\ell^*-1)})_i}\big)_{i\in I}\EE
is a symplectic structure on~$(\wh\cN_{(\ell^*-1)})_{\prt}$ for all 
$(t,\tau)\!\in\!B\!\times\!\bI$ and
\BE{SympNeighLmm_e48b}\begin{split}
\mu_{t,0;i}=0, &\quad
\wt\om_{t,\tau;i}\big|_{T\cN_i|_V}=\wt\om_{t,1;i}^{(\ell^*-1)}\big|_{T\cN_i|_V},\quad
\wt\om_{t,1;i}|_{\wt\cN_i|_K}=\wt\om_{t;i}'|_{\wt\cN_i|_K},\\
&\supp\big(\mu_{\cdot,\tau;i}\big)\subset 
\big(B\!-\!N(\prt B)\big)\!\times\!\big(\cN''_{(\ell^*)}\big)_i|_{K_{\ell^*+1}^{\circ}-U_{\ell^*-1}}
\end{split}\EE
for all $t\!\in\!B$, $\tau\!\in\!\bI$, and $i\!\in\!I$.\\

\noindent
Let $K'^{\circ}$ be an open neighborhood of $K_{\ell^*}\!\subset\!V$ so that
its closure~$K'$ is contained in~$K^{\circ}$. 
Choose a smooth function
$$\eta\!:V\lra\bI \qquad\hbox{s.t.}\quad 
\eta|_{K_{\ell^*}}\!=1,~~\eta|_{V-K'}\!=0.$$
For $t\!\in\!B$ and $i\!\in\!I$, define
$$\Om_{t;i}'=(1\!-\!\eta)\Om_{t,1;i}^{(\ell^*-1)}\!+\!\eta\Om_{t;i}^{\bu}\,,\quad
\wh\om_{t;i}'=\pi^*\om_t|_{\cN_i}\!+\!
\frac12\nd\io_{\ze_{\cN}}\{\Om_{t;i}'\}_{\na^{(t)}}\,.$$
In particular,
\BE{SympDef_e41} 
\wh\om_{t;i}'\big|_{\cN_i|_{K_{\ell^*}}}=
\wh\om_{t;i}^{\bu}\big|_{\cN_i|_{K_{\ell^*}}} \,.\EE

\vspace{.2in}

\noindent
By the $\ell\!=\!\ell^*\!-\!1$ case of~\eref{SympDef_e45b} 
and~\eref{SympDef_e45d}, 
\BE{SympDef_e50c}\begin{split}
&\big(\Om_{t;i}'^{\,\bu}\big)_{i\in I}
=\big(\Om_{t,1;i}^{(\ell^*-1)\bu}\big)_{i\in I}
~~\forall\,t\!\in\!B,\quad
\big(\Om_{t;i}'\big)_{i\in I}=\big(\Om_{t,1;i}^{(\ell^*-1)}\big)_{i\in I}
~~\forall\,t\!\in\!N(\prt B),\\
&\hspace{.5in}
\big(\Om_{t;i}'|_{U_{\ell^*-1}\cup(V-K')}\big)_{i\in I}
=\big(\Om_{t,1;i}^{(\ell^*-1)}|_{U_{\ell^*-1}\cup(V-K')}\big)_{i\in I}~~
\forall\,t\!\in\!B.  
\end{split}\EE
Thus, the three conditions in~\eref{Omcond_e} with~$U$ and $\Om_{t;i}$
replaced by $U_{\ell^*-1}\!\cup\!(V\!-\!K')$ and $\Om_{t,1;i}^{(\ell^*-1)}$, 
respectively, are satisfied.
Since $K_{\ell^*+1}$ is a compact subset of~$V$, so~is
$$V-U_{\ell^*-1}\!\cup\!(V\!-\!K')=K'-U_{\ell^*-1}\subset K_{\ell^*+1}\,.$$
By Proposition~\ref{SympDef_prp}, 
there thus exist neighborhoods $\wh\cN_{(\ell^*)}\!\subset\!\cN'''$ of~$V\!\subset\!\cN$
such that $\ov{\cN'''}\!\subset\!\wt\cN$ and a smooth family 
$(\mu_{t,\tau;i}')_{t\in B,\tau\in\bI,i\in I}$ of
1-forms on~$\cN_{\prt}$ such~that 
$$\big(\wh\om_{t,\tau;i}\big)_{i\in I}\equiv
\big((\wt\om_{t;i}'\!+\!\nd\mu_{t,\tau;i}')|_{\cN_i'}\big)_{i\in I}$$
is a symplectic structure on~$\cN_{\prt}'''$ for all $(t,\tau)\!\in\!B\!\times\!\bI$ and
\BE{SympDef_e50}\begin{split}
&\mu_{t,0;i}'=0, \quad \wh\om_{t,\tau;i}|_V=\om_t,\quad
\wh\om_{t,1;i}|_{(\wh\cN_{(\ell^*)})_i}=\wh\om_{t;i}'|_{(\wh\cN_{(\ell^*)})_i}, \\
&\hspace{.3in}\supp\big(\mu_{\cdot,\tau;i}'\big)\subset 
\big(B\!-\!N(\prt B)\big)\!\times\!\cN_i'''|_{K'-U_{\ell^*-1}}
\end{split}\EE
for all $t\!\in\!B$, $\tau\!\in\!\bI$, and $i\!\in\!I$.\\ 

\noindent
By reparametrizing $\mu_{t,\tau;i}$ and $\mu_{t,\tau;i}'$ as functions of~$\tau$,
we can assume~that 
\BE{SympDef_e53}
\supp\big(\mu_{t,\cdot;i}'\big)\subset\big(\frac12,1\big],
~~\supp\big(\mu_{t,\cdot;i}\!-\!\mu_{t,1;i}\big)
\subset\big[0,\frac12\big)
\quad\forall\,t\!\in\!B,\,i\!\in\!I.\EE
The tuple 
$$\big(\mu_{t,\tau;i}^{(\ell^*)}\big)_{t\in B,\tau\in\bI,i\in I}
\equiv \big(\mu_{t,\tau;i}\!+\!\mu_{t,\tau;i}'\big)_{t\in B,\tau\in\bI,i\in I}$$
is then a smooth family  of 1-forms on~$\cN_{\prt}$. 
By the first and last properties in~\eref{SympNeighLmm_e48b} and in~\eref{SympDef_e50}, 
it satisfies the first and last properties in~\eref{SympDef_e45} with $\ell\!=\!\ell^*$.
By~\eref{SympDef_e53}, 
the last properties in~\eref{SympNeighLmm_e48b} and in~\eref{SympDef_e50},
and the third property in~\eref{SympNeighLmm_e48b}, 
the closed 2-form on~$\cN_{\prt}'$ given by~\eref{SympDef_e45a}
with $\ell\!=\!\ell^*$ satisfies
\BE{SympDef_e55}\wt\om_{t,\tau;i}^{(\ell^*)}\big|_v =\begin{cases}
\wt\om_{t,1;i}^{(\ell^*-1)}|_v,&\hbox{if}~
(\tau,v)\!\in\!\bI\!\times\!(\cN_i'\!-\!\cN_{(\ell^*)}''|_{K_{\ell^*+1}^{\circ}-U_{\ell^*-1}});\\
\wt\om_{t,\tau;i}|_v,&\hbox{if}~
(\tau,v)\!\in\!\bI\!\times\!(\cN''_{(\ell^*)})_i\!-\!(\frac12,1]
\!\times\!\cN'''|_{K'-U_{\ell^*-1}};\\
\wh\om_{t,\tau;i}|_v,&\hbox{if}~(\tau,v)\!\in\![\frac12,1]\!\times\!\cN_i'''|_K.
\end{cases}\EE   
Thus, $\wt\om_{t,\tau;i}^{(\ell^*)}$ is a symplectic structure on~$\cN_{\prt}'$.
By~\eref{SympDef_e55} and the second properties in~\eref{SympDef_e45} with 
$\ell\!=\!\ell^*\!-\!1$, in~\eref{SympNeighLmm_e48b}, and in~\eref{SympDef_e50}, 
the second property in~\eref{SympDef_e45} with $\ell\!=\!\ell^*$ is satisfied as~well.
By~\eref{SympDef_e55}, \eref{SympDef_e45b} with $\ell\!=\!\ell^*\!-\!1$,
the second property in~\eref{SympNeighLmm_e48b}, 
 the third property in~\eref{SympDef_e50}, and the first property in~\eref{SympDef_e50c},
\eref{SympDef_e45b} with  $\ell\!=\!\ell^*$ holds.
By the third case in~\eref{SympDef_e55}, the third property in~\eref{SympDef_e50},
and~\eref{SympDef_e41},  
the second property in~\eref{SympDef_e45} with $\ell\!=\!\ell^*$ is satisfied.\\

\noindent
By the above, we can assume that $(\cN\mu1)$ and $(\cN\mu2)$
 hold for all $\ell\!\in\!\Z^+$.
We can also assume~that 
\BE{SympDef_e45c}
\supp\big(\mu_{t,\cdot;i}^{(\ell)}\big)\subset\big(1\!-\!2^{1-\ell},1\big],
~~
\supp\big(\mu_{t,\cdot;i}^{(\ell)}\!-\!\mu_{t,1;i}^{(\ell)}\big)
\subset\big[0,1\!-\!2^{-\ell}\big)
\quad\forall\,\ell\!\in\!\Z^+,\,t\!\in\!B,\,i\!\in\!I,\EE
i.e.~$\mu_{t,\tau;i}^{(\ell)}$ as a function of~$\tau$ changes only 
in the interval $(1\!-\!2^{1-\ell},1\!-\!2^{-\ell})$.
Let
$$\cN''=\cN_{(1)}'', \quad 
\wh\cN=\bigcup_{\ell=1}^{\i}\wh\cN_{(\ell)}\big|_{K_{\ell}^{\circ}}, \quad
\mu_{t,\tau;i}=\sum_{\ell=1}^{\i}\mu_{t,\tau;i}^{(\ell)}\big|_{\cN_i'} 
\quad\forall~t\!\in\!B,\,\tau\!\in\!\bI,\,i\!\in\!I.$$
The sets $\cN''$ and $\wh\cN$ are open neighborhoods of $V\!\subset\!\cN'$ such that
$\ov{\cN''}\!\subset\!\cN'$.
By the last property in~\eref{SympDef_e45}, 
$$\mu_{t,\tau;i}^{(\ell)}\big|_v=0
\qquad\forall\,v\!\in\!\cN_i|_{K_{\ell^*}-K_{\ell^*-1}^{\circ}},\,
\ell\!\in\!\Z^+\!-\!\{\ell^*\!-\!1,\ell^*\},\,\ell^*\!\in\!\Z^+\,.$$
Thus, the sum above is well-defined and determines a smooth family of
1-forms on~$\cN_{\prt}'$.\\

\noindent 
By the first, second, and last properties in~\eref{SympDef_e45}, 
the family $(\mu_{t,\tau;i})_{t\in B,\tau\in\bI,i\in I}$ satisfies
the first, second, and last requirements in~\eref{SympDefThm_e0}.
By the last two properties in~\eref{SympDef_e45} and~\eref{SympDef_e45g}, 
\begin{equation*}\begin{split}
\wt\om_{t,1;i}\big|_{(\wh\cN_{(\ell^*)})_i|_{K_{\ell^*}^{\circ}}}
\equiv\bigg(\wt\om_{t;i}\!+\nd\!\sum_{\ell=1}^{\i}\mu_{t,1;i}^{(\ell)}\bigg)
\Big|_{(\wh\cN_{(\ell^*)})_i|_{K_{\ell^*}^{\circ}}}
&=\bigg(\wt\om_{t;i}\!+\nd\!\sum_{\ell=1}^{\ell^*}\mu_{t,1;i}^{(\ell)}\bigg)
\Big|_{(\wh\cN_{(\ell^*)})_i|_{K_{\ell^*}^{\circ}}}\\
&=\wt\om_{t,1;i}^{(\ell^*)}\big|_{(\wh\cN_{(\ell^*)})_i|_{K_{\ell^*}^{\circ}}}
=\wh\om_{t;i}^{\bu}\big|_{(\wh\cN_{(\ell^*)})_i|_{K_{\ell^*}^{\circ}}}
\end{split}\end{equation*}
for all $\ell^*\!\in\!\Z^+$.
Thus, the third requirement in~\eref{SympDefThm_e0} is also satisfied.
If \hbox{$\tau\!\in\![1\!-\!2^{1-\ell^*},1\!-\!2^{-\ell^*}]$}  for some $\ell^*\!\in\!\Z^+$,
then
\begin{equation*}\begin{split}
\wt\om_{t,\tau;i}\equiv 
\wt\om_{t;i}\!+\nd\!\sum_{\ell=1}^{\i}\mu_{t,\tau;i}^{(\ell)}\big|_{\cN_i'}
&=\wt\om_{t;i}\!+\nd\!\sum_{\ell=1}^{\ell^*-1}\mu_{t,1;i}^{(\ell)}\big|_{\cN_i'}
+\nd\mu_{t,\tau;i}^{(\ell^*)}\big|_{\cN_i'}\\
&=\wt\om_{t,1;i}^{(\ell^*-1)}\!+\!\nd\mu_{t,\tau;i}^{(\ell^*)}\big|_{\cN_i'}
=\wt\om_{t,\tau;i}^{(\ell^*)}\,;
\end{split}\end{equation*}
see~\eref{SympDef_e45c} and~\eref{SympDef_e45g}.
Thus, $\wt\om_{t,\tau;i}$ is a symplectic structure on~$\cN_{\prt}'$
for all $(t,\tau)\!\in\!B\!\times\!\bI$.
\end{proof}

\begin{rmk}\label{SympDefVB_rmk}
By the proof of Theorem~\ref{SympDefVB_thm} above, 
the compactness requirements on~$V\!-\!U$ in Proposition~\ref{SympDef_prp}
and on~$K$ in Lemma~\ref{SympNeigh_lmm2} are not necessary.
\end{rmk}

\section{Tubular neighborhood theorems}
\label{TubulNeigh_sec}

\noindent
We next obtain a stratified version of the usual Tubular Neighborhood Theorem
which respects a symplectic form along a symplectic submanifold.
Proposition~\ref{TubulNeigh_prp} below is used in Section~\ref{SCCpf_sec} to apply 
the essentially local statement of Theorem~\ref{SympDefVB_thm}
in the setting of Theorem~\ref{SCC_thm}.
We continue with the notation of Sections~\ref{SCDregul_subs} and~\ref{SCCregul_subs}.

\begin{dfn}\label{ConfRegulLoc_dfn}
Let $N\!\in\!\Z^+$, $\X\!\equiv\!\{X_I\}_{I\in\cP^*(N)}$ be a transverse configuration,
$I^*\!\in\!\cP^*(N)$, and \hbox{$U\!\subset\!X_{I^*}$} be an open subset.
A \sf{regularization for~$U$ in~$\X$} is a tuple 
$(\Psi_i)_{i\in I^*}$, 
where $\Psi_i$ is a regularization for~$U$ in~$X_i$ 
in the sense of Definition~\ref{smreg_dfn}, such that
\begin{alignat}{2}
\label{ConfRegulLoc_e1}
\Psi_i\big(\cN_{I^*;I}\!\cap\!\Dom(\Psi_i)\big)&=X_I\!\cap\!\Im(\Psi_i)
&\quad &\forall\,i\!\in\!I\!\subset\!I^*,\\
\label{ConfRegulLoc_e2}
\Psi_{i_1}\big|_{\cN_{I^*;i_1i_2}\cap\Dom(\Psi_{i_1})}
&=\Psi_{i_2}\big|_{\cN_{I^*;i_1i_2}\cap\Dom(\Psi_{i_2})}
&\quad &\forall\,i_1,i_2\!\in\!I^*\,.
\end{alignat}
\end{dfn}

\vspace{.1in}

\noindent 
Smooth families of regularizations for~$U$ in~$\X$ 
are defined analogously to Definition~\ref{SCDregul_dfn}\ref{sympregul_it2}.

\begin{prp}\label{TubulNeigh_prp}
Let $N\!\in\!\Z^+$, $\X\!\equiv\!\{X_I\}_{I\in\cP^*(N)}$ be a transverse configuration
such that $X_{ij}$ is a closed submanifold of $X_i$ of codimension~2
for all $i,j\!\in\![N]$ distinct,
$I^*\!\in\!\cP^*(N)$, and $U,U'\!\subset\!X_{I^*}$ be open subsets, possibly empty, 
such that $\ov{U'}\!\subset\!U$.
Suppose
\begin{enumerate}[label=$\bullet$,leftmargin=*]

\item $B$ is a compact manifold, possibly with boundary,

\item  $N(\prt B),N'(\prt B)$ are neighborhoods of $\prt B\!\subset\!B$
such that $\ov{N'(\prt B)}\!\subset\!N(\prt B)$,

\item $(\om_{t;i})_{t\in B}$ is a smooth family of symplectic structures on 
$\X$ in the sense of Definition~\ref{TransConf_dfn2},

\item $(\Psi_{I^*;t;i})_{t\in N(\prt B),i\in I^*}$ and $(\Psi_{U;t;i})_{t\in B,i\in I^*}$
are smooth families of regularizations for~$X_{I^*}$ and~$U$, respectively,
in $\X$ such~that
\begin{gather}
\label{TubulNeigh_e0a}
\nd_x\Psi_{\star;t;i}\big(\cN_{I^*;i}|_x\big)=T_xX_{I^*}^{\,\om_{t;i}}
\quad\forall~\star\!=\!I^*,U,~
(t,x)\in\begin{cases}N(\prt B)\!\times\!X_{I^*}, &\hbox{if}~\star\!=\!I^*;\\
B\!\times\!U, &\hbox{if}~\star\!=\!U;
\end{cases}\\
\label{TubulNeigh_e0b}
\big(\Psi_{I^*;t;i}|_{\Dom(\Psi_{I^*;t;i})|_U}\big)_{t\in N(\prt B),i\in I^*}
 = \big(\Psi_{U;t;i}\big)_{t\in N(\prt B),i\in I^*}\,.
\end{gather}
\end{enumerate}
Then there exists a smooth family $(\Psi_{t;i})_{t\in B,i\in I^*}$ of regularizations for 
$X_{I^*}$ in $\X$ such~that 
\begin{gather}
\label{SympNeighCrl_e1b}
\nd_x\Psi_{t;i}\big(\cN_{I^*;i}|_x\big)=T_xX_{I^*}^{\,\om_{t;i}}
\quad\forall\,t\!\in\!B,\,x\!\in\!X_{I^*},\,i\!\in\!I^*,\\
\label{SympNeighCrl_e1a}
\begin{split}
\big(\Psi_{t;i}\big)_{t\in N'(\prt B),i\in I^*}
&=\big(\Psi_{I^*;t;i}\big)_{t\in N'(\prt B),i\in I^*},\\
\big(\Psi_{t;i}|_{\Dom(\Psi_{t;i})|_{U'}}\big)_{t\in B,i\in I^*}
&=\big(\Psi_{U;t;i}|_{\Dom(\Psi_{U;t;i})|_{U'}}\big)_{t\in B,i\in I^*}.
\end{split}
\end{gather}
\end{prp}

\subsection{Smooth regularizations for transverse collections}
\label{SCDsm_subs}

\noindent
Lemma~\ref{stratexp_lmm} below shows that regularizations 
in the sense of Definition~\ref{smreg_dfn} 
that satisfy the stratification condition~\eref{Psikk_e} always exist, 
if $V_I$ is a closed submanifold.
By Lemma~\ref{SympNeigh_lmm}, they can be chosen to extend given regularizations
over an open subspace, after slightly shrinking the latter, and 
to respect a symplectic form along~$V_I$.

\begin{lmm}\label{stratexp_lmm}
Let $X$ be a manifold and $\{V_i\}_{i\in I}$ be a  transverse collection  
of closed submanifolds of~$X$.
Then there exists a smooth map
\hbox{$\exp_I\!: TX|_{V_I}\!\lra\!X$} such~that 
\BE{stratexp_e12}\begin{split}
\exp_I\!\big|_{V_I}=\id, \qquad&
\nd_x\exp_I\!=\id: T_xX\lra T_xX~~\forall\,x\!\in\!V_I, \\
\exp_I\!\big(TV_{I'}|_{V_I}\big)\,&=V_{I'}\!\cap\!\Im(\exp_I)~\forall\,I'\!\subset\!I.
\end{split}\EE
\end{lmm}

\begin{proof}
Choose a metric~$g$ on~$X$ so that the orthogonal complements 
$L_i$ of $TV_I$ in $TV_{I-i}|_{V_I}$ are orthogonal for 
pairs of different values of $i\!\in\!I$.
For each $I'\!\subset\!I$, let
$$\cN_{I;I'}=\bigoplus_{i\in I-I'}\!\!\!L_i\approx \cN_{V_{I'}}V_I$$
and $\pi_{I'}\!:\cN_{I;\eset}\!\lra\!\cN_{I;I'}$ be the projection map.
There is then a canonical identification 
$$T\cN_{I;I'}|_{V_I} = TV_{I'}|_{V_I}\,.$$
Let $\exp\!:W\!\lra\!X$, where $W$ is a neighborhood of $X\!\subset\!TX$, be
the exponential map with respect to the Levi-Civita connection of the metric~$g$.\\

\noindent
Denote by $\Psi_0\!:\cN_{I;\eset}\!\lra\!X$
the composition of $\exp$ with a diffeomorphism
from $\cN_{I;\eset}$ to a neighborhood of $V_I\!\subset\!\cN_{I;\eset}\!\cap\!W$ which restricts
to the identity on a smaller neighborhood of $V_I\!\subset\!\cN_{I;\eset}$.
Suppose $\ell\!\in\!\{1,\ldots,|I|\}$ and we have constructed a smooth map
$\Psi_{\ell-1}\!:\cN_{I;\eset}\!\lra\!X$ such~that 
\BE{expcond_e12}\begin{split}
\Psi_{\ell-1}|_{V_I}=\id_{V_I}, &\quad
\nd\Psi_{\ell-1}|_{T\cN_{I;\eset}|_{V_I}}\!=\!\id_{T\cN_{I;\eset}|_{V_I}}\!:
T\cN_{I;\eset}|_{V_I}\lra TX|_{V_I}\,, \\
\Psi_{\ell-1}(\cN_{I;I''})&=V_{I''}\!\cap\!\Im(\Psi_{\ell-1})
\quad\forall\,I''\!\subset\!I
~~\hbox{s.t.}~~|I''|\!>\!|I|\!-\!\ell.
\end{split}\EE

\vspace{.1in}

\noindent
By the first two statements in~\eref{expcond_e12} and 
the Inverse Function Theorem \cite[Theorem~1.30]{Warner}, 
there exist a neighborhood~$W$ of $V_I\!\subset\!X$ and a smooth map
$\Phi\!:W\!\lra\!\cN_{I;\eset}$ such that 
\BE{PsiPhi_e}\Psi_{\ell-1}\!\circ\!\Phi=\id_W,\quad 
\Phi\!\circ\!\Psi_{\ell-1}|_{\Phi(W)}=\id_{\Phi(W)}\,.\EE
In particular, $\Phi(V_{I'}\!\cap\!W)\!\subset\!\cN_{I;\eset}$ is a smooth submanifold
for all $I'\!\subset\!I$.
By~\eref{expcond_e12} and the second equation in~\eref{PsiPhi_e}, 
\BE{PhiIprp_e}\begin{split}
\Phi|_{V_I}=\id_{V_I}, &\quad 
T\Phi(V_{I'}\!\cap\!W)\big|_{V_I}=T\cN_{I;I'}|_{V_I}~~\forall\,I'\!\subset\!I,\\
\Phi(V_{I''}\!\cap\!W)=&\,\cN_{I;I''}\!\cap\!\Im(\Phi)
\quad\forall\,I''\!\subset\!I
~~\hbox{s.t.}~~|I''|\!>\!|I|\!-\!\ell.
\end{split}\EE

\vspace{.2in}

\noindent
By the first two statements in~\eref{PhiIprp_e}, for every $I'\!\subset\!I$
we can apply 
the Inverse Function Theorem  to the projection
$$\pi_{I'}\!:\Phi(V_{I'}\!\cap\!W)\lra\cN_{I;I'}$$
forgetting the components in $L_i$ with $i\!\in\!I'$. 
There thus exist a neighborhood~$\cN'$ of~$V_I\!\subset\!\cN_{I;\eset}$, 
a neighborhood~$W'$ of $V_I\!\subset\!W$, and fiber-preserving smooth maps
\begin{gather}
\notag
h_{I';i}\!:  \cN_{I;I'}\!\cap\!\cN' \lra L_i, 
\quad i\!\in\!I'\!\subset\!I, \qquad\hbox{s.t.}\\ 
\label{PsiPhi_e2}
\pi_{I'}(\cN')\!=\!\cN_{I;I'}\!\cap\!\cN',~~~
\big(\id_{\cN_{I;I'}},(h_{I';i})_{i\in I'}\big)
\!\circ\!\pi_{I'}|_{\Phi(V_{I'}\cap W')}=\id_{\Phi(V_{I'}\cap W')}
\quad\forall\, I'\!\subset\!I.
\end{gather}
By~\eref{PhiIprp_e} and~\eref{PsiPhi_e2}, 
\BE{hprop_e}\nd_xh_{I';i}=0~~\forall\,x\!\in\!V_I,\qquad
h_{I';i}(v)=0~~\forall\,
v\!\in\!\cN_{I;I''}\!\cap\!\cN',~I''\!\supset\!I',~|I''|\!>\!|I|\!-\!\ell.\EE

\vspace{.2in}

\noindent
Let $\cP_{\ell}^c(I)$ denote the collection of subsets $I'\!\subset\!I$ with 
$|I'|\!=\!|I|\!-\!\ell$.
We define a smooth fiber-preserving map
$$\Th\!\equiv\!(\Th_i)_{i\in I}:\cN'\lra \cN_{I;\eset} \qquad\hbox{by}\quad
\Th_i(v)=v_i
+\sum_{\begin{subarray}{c}I'\in\cP_{\ell}^c(I)\\ i\in I'\end{subarray}}
\!\!\!\!\! h_{I';i}\big(\pi_{I'}(v)\big).$$
By~\eref{hprop_e} and $\pi_{I''}(\cN_{I;I'})\!=\!\cN_{I;I'\cup I''}$ for every $I''\!\subset\!I$, 
\BE{Thprop_e}
\nd\Th\big|_{T\cN_{I;\eset}|_{V_I}}=\id\big|_{T\cN_{I;\eset}|_{V_I}}, \qquad
\Th\big|_{\cN_{I;I'}\cap\cN'}=\big(\id_{\cN_{I;I'}},(h_{I';i})_{i\in I'}\big)
\big|_{\cN_{I;I'}\cap\cN'} ~~\forall~I'\!\in\!\cP_{\ell}^c(I).\EE
By the Inverse Function Theorem, $\Th$ thus restricts 
to a diffeomorphism on a neighborhood $\cN''$ of $V_I\!\subset\!\cN'$.
By the second statement in~\eref{Thprop_e} and~\eref{PsiPhi_e2}, the diffeomorphism
$$\Psi_{\ell}'\!\equiv\!\Psi_{\ell-1}\!\circ\!\Th\!: \cN''\lra \Psi_{\ell-1}\big(\Th(\cN'')\big)$$
satisfies the last condition in~\eref{expcond_e12} 
with $\ell$ replaced by~$\ell\!+\!1$ and 
$\cN_{I;I'}$ by $\cN_{I;I'}\!\cap\!\cN''$.
As it also satisfies the first two conditions in~\eref{expcond_e12}, we can obtain a smooth map
$\Psi_{\ell}\!:\cN_{I;\eset}\!\lra\!X$
satisfying~\eref{expcond_e12} with $\ell$ replaced by~$\ell\!+\!1$
by composing~$\Psi_{\ell}'$ with a diffeomorphism
from $\cN_{I;\eset}$ to a neighborhood of $V_I\!\subset\!\cN''$ which restricts
to the identity on a smaller neighborhood of $V_I\!\subset\!\cN''$ and preserves lines
inside of each fiber of~$\cN_{I;\eset}$.\\

\noindent
Thus, there exists a smooth map 
$\Psi_{\ell}\!:\cN_{I;\eset}\!\lra\!X$ satisfying~\eref{expcond_e12} 
with $\ell\!=\!|I|\!+\!1$.
By composing $\Psi_{\ell}$ with the orthogonal projection $TX|_{V_I}\!\lra\!\cN_{I;\eset}$,
we obtain a smooth map~$\exp_I$ with the desired properties.
\end{proof}
 
\begin{lmm}\label{SympNeigh_lmm}
Let $X$ be a  manifold, $\{V_i\}_{i\in S}$ be a finite transverse collection  
of closed submanifolds of~$X$ of codimension~2, $I\!\in\!\cP^*(S)$,
and $U,U'\!\subset\!V_I$ be open subsets, possibly empty, such that $\ov{U'}\!\subset\!U$.
Suppose
\begin{enumerate}[label=$\bullet$,leftmargin=*]

\item $N'(\prt B)\!\subset\!N(\prt B)\!\subset\!B$ are as in Proposition~\ref{TubulNeigh_prp},

\item $(\om_t)_{t\in B}$ is a smooth family of symplectic structures on 
$\{V_i\}_{i\in S}$ in $X$ in the sense of Definition~\ref{SCdivstr_dfn},

\item $(\Psi_{I;t})_{t\in N(\prt B)}$ and $(\Psi_{U;t})_{t\in B}$
are smooth families of regularizations for~$V_I$ and~$U$, respectively,
in~$X$ such~that
\BE{SympNeigh_e0b}
\Psi_{\star;t}\big(\cN_{I;I'}\!\cap\!\Dom(\Psi_{\star;t})\big)
=V_{I'}\!\cap\!\Im(\Psi_{\star;t})~~\forall\,I'\!\subset\!I,\quad
\nd_x\Psi_{\star;t}\big(\cN_XV_I|_x\big)=T_xV_I^{\om_t},\EE
for all $\star\!=\!I,U$, $(t,x)\!\in\!N(\prt B)\!\times\!V_I$ if $\star\!=\!I$, and 
$(t,x)\!\in\!B\!\times\!U$ if $\star\!=\!U$, and 
\BE{SympNeigh_e0a}
\big(\Psi_{I;t}|_{\Dom(\Psi_{I;t})|_U}\big)_{t\in N(\prt B)}
 = \big(\Psi_{U;t}\big)_{t\in N(\prt B)}\,.\EE
\end{enumerate}
Then there exists a smooth family $(\Psi_t)_{t\in B}$ of regularizations for 
$V_I$ in~$X$ such~that 
\begin{gather}
\label{SympNeigh_e1b}
\Psi_t\big(\cN_{I;I'}\!\cap\!\Dom(\Psi_t)\big)
=V_{I'}\!\cap\!\Im(\Psi_t)~~\forall\,I'\!\subset\!I,\quad
\nd_x\Psi_t\big(\cN_XV_I|_x\big)=T_xV_I^{\om_t}~~\forall\,x\!\in\!V_I,\\
\label{SympNeigh_e1a}
\big(\Psi_t\big)_{t\in N'(\prt B)}=\big(\Psi_{I;t}\big)_{t\in N'(\prt B)},
\quad
\big(\Psi_t|_{\Dom(\Psi_t)|_{U'}}\big)_{t\in B}
=\big(\Psi_{U;t}|_{\Dom(\Psi_{U;t})|_{U'}}\big)_{t\in B}.
\end{gather}
\end{lmm}

\begin{proof} Let $\pi\!:TV_I\!\lra\!V_I$ and
$\pi_I\!:\cN_XV_I\!\lra\!V_I$ be the projection maps and
\hbox{$\exp_I\!: TX|_{V_I}\!\lra\!X$} be as in Lemma~\ref{stratexp_lmm}.
Choose an isomorphism
$$\wt\exp_{V_I}\!:\pi^*\cN_XV_I\lra \big\{\!\exp_I|_{TV_I}\!\big\}^*\cN_XV_I
\subset TV_I\!\times\!\cN_XV_I$$
of split vector bundles over~$TV_I$ restricting to the identity 
over $V_I\!\subset\!TV_I$, i.e.
\BE{wtexpcond_e}\wt\exp_{V_I}(x,v)=(x,v)\qquad\forall~(x,v)\in
\big(\pi^*\cN_XV_I\big)|_{V_I}\subset  TV_I\!\times\!\cN_XV_I.\EE
Denote by $\pi_2\!:\{\exp_I|_{TV_I}\}^*\cN_XV_I\!\lra\!\cN_XV_I$ 
the projection onto the second component.\\

\noindent
Let $t\!\in\!B$.
We identify $\cN_XV_I\!=\!\cN_{I;\eset}$ with 
the $\om_t$-orthogonal complement $TV_I^{\om_t}\!\subset\!TX|_{V_I}$
of $TV_I$ via the quotient projection map;
it is the direct sum of the $\om_t$-orthogonal complements of $TV_I$ in 
$TV_{I-i}|_{TV_I}$ with $i\!\in\!I$.
Define
$$\wh\Psi_t'=\exp_I\big|_{TV_I^{\om_t}}:\cN_XV_I\!=\!TV_I^{\om_t}\lra X.$$
By the first two statements in~\eref{stratexp_e12} and 
the Tubular Neighborhood Theorem \cite[(12.11)]{BJ},
there exists a neighborhood~$\cW$ of $B\!\times\!V_I$ in $B\!\times\!\cN_XV_I$
such~that 
$$\wh\Psi_t\equiv\wh\Psi_t'|_{W_t}\!:W_t\lra X,
\qquad\hbox{where}\quad \{t\}\!\times\!W_t \equiv \big(\{t\}\!\times\!\cN_XV_I\big)\cap \cW,$$
is a regularization for $V_I$ in~$X$ for each $t\!\in\!B$.
By the last two statements in~\eref{stratexp_e12},
\BE{SympNeigh_e2}
\wh\Psi_t\big(\cN_{I;I'}\!\cap\!W_t\big)=V_{I'}\!\cap\!\Im(\wh\Psi_t)
~~\forall\,I'\!\subset\!I, \quad
\nd_x\wh\Psi_t\big(\cN_XV_I|_x\big)=T_xV_I^{\om_t}~~\forall\,x\!\in\!V_I.\EE 

\vspace{.2in}

\noindent
Let $\star\!=\!I,U$, $t\!\in\!N(\prt B)$ if  $\star\!=\!I$, and 
$t\!\in\!B$ if $\star\!=\!U$.
Since $\wh\Psi_t$ and $\Psi_{\star;t}$ are regularizations,  
\BE{SympNeigh_e5}
\nd_x\wh\Psi_t=\nd_x\Psi_{\star;t}\qquad\forall\,
x\in\begin{cases} V_I,&\hbox{if}~\star\!=\!I;\\
U,&\hbox{if}~\star\!=\!U.\end{cases}\EE
By~\eref{SympNeigh_e5} and the Inverse Function Theorem,
there exists a neighborhood~$\cW_{\star}$  
of $N(\prt B)\!\times\!V_I$ in~$\cW$ if $\star\!=\!I$ and 
of $B\!\times\!U$ in $\cW\!\cap\!(B\!\times\!\cN_XV_I|_U)$ if $\star\!=\!U$
independent of~$t$  such~that the~map
\BE{ThIUdfn_e}
\Th_{\star;t}\equiv \wh\Psi_t^{-1}\!\circ\!\Psi_{\star;t}: 
W_{\star;t}\lra\cN_XV_I,
~~\hbox{where}~~~ 
\{t\}\!\times\!W_{\star;t} \equiv \big(\{t\}\!\times\!\cN_XV_I\big)\cap \cW_{\star},\EE
is a well-defined diffeomorphism onto a neighborhood 
of $V_I\!\subset\!\cN_XV_I$ if $\star\!=\!I$ and 
of $U\!\subset\!\cN_XV_I|_U$ for $t\!\in\!B$.
By~\eref{SympNeigh_e5}, the first assumption in~\eref{SympNeigh_e0b}, 
and the second property in~\eref{SympNeigh_e2}, 
\BE{SympNeigh_e25b}
\Th_{\star;t}(x)=x,\quad \nd_x\Th_{\star;t}=\id,\quad
\Th_{\star;t}\big(\cN_{I;I'}\!\cap\!W_{\star;t}\big)
=\cN_{I;I'}\!\cap\!\Im\big(\Th_{\star;t}\big)~~\forall\,I'\!\subset\!I,\EE
for all $x\!\in\!\!V_I$ if $\star\!=\!I$ and
$x\!\in\!U$ if $\star\!=\!U$.
By~\eref{SympNeigh_e0a},
\BE{SympNeigh_e25a}\big(\Th_{I;t}|_{W_{I;t}\cap W_{U;t}}\big)_{t\in N(\prt B)}
=\big(\Th_{U;t}|_{W_{I;t}\cap W_{U;t}}\big)_{t\in N(\prt B)}\,.\EE

\vspace{.2in}

\noindent
With $\star$ and $t$ as above, define
\begin{gather*}
W_{\star;t}'=\big\{v\!\in\!W_{\star;t}\!:\,
\pi_I(\Th_{\star;t}(v))\!\in\!\exp_I(T_{\pi_I(v)}V_I)\big\},\\
\begin{aligned}
\Th_{\star;t}^{\hor}\!:W_{\star;t}'&\lra TV_I &~~&\hbox{by} &
\Th_{\star;t}^{\hor}(v)&\in\!T_{\pi_I(v)}V_I,&
&\exp_I\!\big(\Th_{\star;t}^{\hor}(v)\big)=\pi_I\big(\Th_{\star;t}(v)\big),\\
\Th_{\star;t}^{\ver}\!:W_{\star;t}'&\lra \cN_XV_I &~~&\hbox{by}&
\Th_{\star;t}^{\ver}(v)&\in\!\cN_XV_I|_{\pi_I(v)},&
&\pi_2\big(\wt\exp_{V_I}\!\big(\Th_{\star;t}^{\hor}(v),
  \Th_{\star;t}^{\ver}(v)\big)\big)=\Th_{\star;t}(v).
\end{aligned}\end{gather*}
For a smooth function $\eta\!:V_I\!\lra\!\R$, let
\begin{gather*}
\Th_{\star;t,\eta}\!:W_{\star;t}'\lra \cN_XV_I, \\
\Th_{\star;t,\eta}(v)=
\pi_2\Big(\wt\exp_{V_I}\big(\big(1\!-\!\eta(\pi_I(v))\!\big)\Th_{\star;t}^{\hor}(v),
v\!+\!\big(1\!-\!\eta(\pi_I(v))\big)(\Th_{\star;t}^{\ver}(v)\!-\!v)\big)\!\!\Big).
\end{gather*}
By~\eref{SympNeigh_e25b},
\BE{SympNeigh_e35}
\Th_{\star;t,\eta}(x)=x,\quad \nd_x\Th_{\star;t,\eta}=\id,\quad
\Th_{\star;t,\eta}\big(\cN_{I;I'}\!\cap\!W_{\star;t}'\big)
=\cN_{I;I'}\!\cap\!\Im\big(\Th_{\star;t,\eta}\big)~~\forall\,I'\!\subset\!I,\EE
for all $x\!\in\!V_I$ if $\star\!=\!I$ and $x\!\in\!U$ if $\star\!=\!U$.
By~\eref{SympNeigh_e25a},
\BE{SympNeigh_e34}
\big(\Th_{I;t,\eta}|_{W_{I;t}'\cap W_{U;t}'}\big)_{t\in N(\prt B)}
=\big(\Th_{U;t,\eta}|_{W_{I;t}'\cap W_{U;t}'}\big)_{t\in N(\prt B)}.\EE
The set 
$$\cW_{\star}''\equiv\big\{(t,v)\!\in\!B\!\times\!\cN_XV_I\!:~
v\!\in\!W_{\star;t}',\,
\Th_{\star;t,\tau}(v)\!\in\!W_t~\forall\,\tau\!\in\!\bI\big\}$$
is a neighborhood of 
\begin{alignat*}{2}
N(\prt B)\!\times\!V_I&\subset 
\bigcup\limits_{t\in N(\prt B)}\!\!\!\!\!\{t\}\!\times\!\Dom(\Psi_{I;t})
\subset N(\prt B)\!\times\!\cN_XV_I&\quad\hbox{if}~\star&=\!I,\\
B\!\times\!U &\subset\bigcup\limits_{t\in B}\{t\}\!\times\!\Dom(\Psi_{U;t})
\subset B\!\times\!\cN_XV_I|_U
&\quad\hbox{if}~\star&=\!U.
\end{alignat*}

\vspace{.2in}

\noindent
Let $N''(\prt B)\!\subset\!B$ and 
$U_0'\!\subset\!U''\!\subset\!V_I$ be open subsets such~that
$$\ov{N'(\prt B)}\subset N''(\prt B), ~~ \ov{N''(\prt B)}\subset N(\prt B),
\qquad
\ov{U'}\subset U_0', ~~ \ov{U_0'}\subset U'', ~~ \ov{U''}\subset U.$$
The set 
\begin{equation*}\begin{split}
\wt\cW\equiv& \Big(\cW\!-\!\ov{N''(\prt B)}\!\times\!\cN_XV_I
\!-\!B\!\times\!\cN_XV_I|_{\ov{U''}}\Big)\cup\cW_I'' \cup\cW_U''\\
&\qquad\cup
\bigg(\bigcup_{t\in N'(\prt B)}\!\!\!\!\!\!\{t\}\!\times\!\Dom(\Psi_{I;t})\!\!\bigg)
\cup \bigg(\bigcup_{t\in B}\{t\}\!\times\!\Dom(\Psi_{U;t})|_{U_0'}\!\!\bigg)
\end{split}\end{equation*}
is then a neighborhood of $B\!\times\!V_I$ in $B\!\times\!\cN_XV_I$.
Choose smooth $[0,1]$-valued functions $\eta_I$ on~$B$ and $\eta_U$ on~$V_I$  
such~that 
\BE{etaIetaUdfn_e}\eta_I(t)=\begin{cases}
0,&\hbox{if}~t\!\in\!N'(\prt B);\\
1,&\hbox{if}~t\!\not\in\!\ov{N''(\prt B)};
\end{cases} \qquad
\eta_U(x)=\begin{cases}
0,&\hbox{if}~x\!\in\!U_0';\\
1,&\hbox{if}~x\!\not\in\!\ov{U''}.
\end{cases}\EE
By~\eref{etaIetaUdfn_e}, \eref{wtexpcond_e}, \eref{ThIUdfn_e},
and~\eref{SympNeigh_e0a},  
\BE{ThIUdfn_e2}
\wh\Psi_t\big(\Th_{\star;t,\eta_I(t)\eta_U}(v)\big)=
\begin{cases}
\wh\Psi_t(v),&\hbox{if}~(t,v)\!\in\!\cW_{\star}'',\,
t\!\not\in\!\ov{N''(\prt B)},\,\pi_I(v)\!\not\in\!\ov{U''};\\
\Psi_{I;t}(v),&\hbox{if}~(t,v)\!\in\!\cW_{\star}'',\,t\!\in\!N'(\prt B),
\,v\!\in\!\Dom(\Psi_{I;t});\\
\Psi_{U;t}(v),&\hbox{if}~(t,v)\!\in\!\cW_{\star}'',\,v\!\in\!\Dom(\Psi_{U;t})|_{U_0'}.
\end{cases}\EE

\vspace{.2in}

\noindent
Define
\BE{wtPsidfn_e}
\wt\Psi\!:\wt\cW\!\lra\!X, \quad
\wt\Psi_t(v)=\begin{cases}
\wh\Psi_t(v),&\hbox{if}~t\!\not\in\!\ov{N''(\prt B)},\,\pi_I(v)\!\not\in\!\ov{U''};\\
\wh\Psi_t\big(\Th_{\star;t,\eta_I(t)\eta_U}(v)\big),
&\hbox{if}~(t,v)\!\in\!\cW_{\star}'',\,\star\!=\!I,U;\\
\Psi_{I;t}(v),&\hbox{if}~t\!\in\!N'(\prt B),\,v\!\in\!\Dom(\Psi_{I;t});\\
\Psi_{U;t}(v),&\hbox{if}~v\!\in\!\Dom(\Psi_{U;t})|_{U_0'}.
\end{cases}\EE
The domain in the first case above is disjoint from the domains in the last two cases.
By~\eref{SympNeigh_e34}, the definitions of~$\wt\Psi_t(v)$ agree on the overlap 
of the two domains in the second case.
By~\eref{ThIUdfn_e2}, its definition on either of these domains agrees
with the definitions in the first, second, and fourth cases on the overlaps.
By~\eref{SympNeigh_e0a}, the definitions of~$\wt\Psi_t(v)$ agree on the overlap 
of the last two cases.
Thus, $\wt\Psi$ is well-defined and smooth.
By the last two cases in~\eref{wtPsidfn_e}, 
$\wt\Psi_t$ satisfies \eref{SympNeigh_e1a} with $\Psi_t\!=\!\wt\Psi_t$.
By the first statements in~\eref{SympNeigh_e0b} and~\eref{SympNeigh_e2}
and the last statement in~\eref{SympNeigh_e35},
$\wt\Psi_t$ satisfies the first property in~\eref{SympNeigh_e1b}.
By the first two cases in~\eref{wtPsidfn_e}, \eref{stratexp_e12}, and
the first two statements in~\eref{SympNeigh_e35}, 
\BE{wtPstprop_e}
\wt\Psi_t(x)=x,\quad \nd_x\wt\Psi_t\!=\!\id\!:\,
T_x\cN_XV_I\!=\!T_xV_I\!\oplus\!T_xV_I^{\om_t}\lra T_xX 
\qquad\forall\,(t,x)\!\in\!B\!\times\!V_I.\EE
This implies that $\wt\Psi_t$ satisfies the second property in~\eref{SympNeigh_e1b}.\\

\noindent
By~\eref{wtPstprop_e} and the Tubular Neighborhood Theorem,
there exists a neighborhood 
$$\wt\cW'\equiv\bigcup_{t\in B}\{t\}\!\times\!W'_t$$ 
of $B\!\times\!V_I$ in $\wt\cW$ such that  
$$\wt\Psi_t|_{W_t'}\!: W_t'\lra\wt\Psi_t(W_t')$$
is a diffeomorphism onto an open neighborhood.
Let 
\begin{equation*}\begin{split}
\bigcup_{t\in B}\{t\}\!\times\!W''_t
\equiv& \bigcup_{t\in B}\{t\}\!\times\!\Big(W_t'\!-\!
\wt\Psi_t^{-1}\big(\ov{\wt\Psi_t(\Dom(\wt\Psi_t)|_{U'})}\big)\!\cap\!
\cN_XV_I|_{V_I-U_0'}\Big)\\
&\qquad\cup
\bigg(\bigcup_{t\in N'(\prt B)}\!\!\!\!\!\!\{t\}\!\times\!\Dom(\Psi_{I;t})\!\!\bigg)
\cup \bigg(\bigcup_{t\in B}\{t\}\!\times\!\Dom(\Psi_{U;t})|_{U'}\!\!\bigg).
\end{split}\end{equation*}  
This is a neighborhood of $B\!\times\!V_I$ in $B\!\times\!\cN_XV_I$
and $\Psi_t\!\equiv\!\wt\Psi_t|_{W_t''}$ is injective,
since $\Psi_{I;t}$ and $\Psi_{U;t}$ are.
Thus, $(\Psi_t)_{t\in B}$ is a smooth family of regularizations for~$V_I$ in~$X$
with the desired properties.   
\end{proof}

\begin{rmk}\label{SympNeigh_rmk}
The first requirement in~\eref{SympNeigh_e1b} is non-trivial only for $|I|\!\ge\!2$.
By \cite[Lemma~3.14]{MS1} and its proof, 
the second requirement in~\eref{SympNeigh_e1b} can be strengthened to 
the equality of $\Psi_t^*\om_t$ with a standard 2-form~$\wh\om_t$ 
on~$\cN_XV_I$ as in~\eref{ombund_e2} 
over a neighborhood~$\cN'$ of~$V_I$ in~$\cN_XV_I$ at the cost of 
dropping the first requirement in~\eref{SympNeigh_e1b}.
By Lemma~\ref{SympNeigh_lmm2} and Remark~\ref{SympDefVB_rmk},
this strengthening can also be achieved after deforming~$\om_t$ on
a neighborhood of~$V$ in~$X$ while keeping all submanifolds~$V_{I'}$ symplectic.
It appears that this strengthening can be achieved without weakening some other condition
if either $|I|\!=\!2$ or the submanifolds $V_i\!\subset\!X$ are $\om_t$-orthogonal.
Due to symplectic angle considerations, this strengthening cannot be achieved 
in the general case without weakening some other condition if $|I|\!\ge\!3$.
The approach of this paper, as summarized by the principle on page~\pageref{equiv_prn}
and the nexus on page~\pageref{equiv_nex}, is a way around this fundamental obstacle
in the setting of singular symplectic divisors and varieties.
\end{rmk}

\subsection{Proof of Proposition~\ref{TubulNeigh_prp}}
\label{SCCsm_subs}

\noindent
By Lemma~\ref{SympNeigh_lmm}, for each $i\!\in\!I^*$
there exists a smooth family of regularizations $(\Psi_{t;i})_{t\in B}$
for $X_{I^*}$ in~$X_i$ such~that
\begin{gather}
\label{SympNeighCrl_e3c}
\Psi_{t;i}\big(\cN_{I^*;I}\!\cap\!\Dom(\Psi_{t;i})\big)
=X_{I}\!\cap\!\Im(\Psi_{t;i}) ~~\forall\,i\!\in\!I\!\subset\!I^*, \\
\label{SympNeighCrl_e3b}
\nd_x\Psi_{t;i}\big(\cN_{I^*;i}|_x\big)=T_xX_{I^*}^{\,\om_{t;i}}
~~\forall\,x\!\in\!X_{I^*},\,i\!\in\!I^*,\\
\big(\Psi_{t;i}\big)_{t\in N'(\prt B)}=\big(\Psi_{I^*;t;i}\big)_{t\in N'(\prt B)},\quad
\label{SympNeighCrl_e3a}
\big(\Psi_{t;i}|_{\Dom(\Psi_{t;i})|_{U'}}\big)_{t\in B}
=\big(\Psi_{U;t;i}|_{\Dom(\Psi_{U;t;i})|_{U'}}\big)_{t\in B}.
\end{gather}
Below we modify the maps $\Psi_{t;i}$ on the intersections of their domains,
i.e.~neighborhoods of $X_{I^*}$ in $\cN_{I^*;i_1i_2}$, 
in order to make them agree there.\\

\noindent
We can assume that $I^*\!=\![\ell^*]$ for some $\ell^*\!\in\!\Z^+$.
Let $\pi\!:TX_{I^*}\!\lra\!X_{I^*}$
be the projection map and 
\hbox{$\exp\!:TX_{I^*}\!\lra\!X_{I^*}$} be a smooth map such~that
$$\exp(x)=x, ~~~ \nd_x\exp\!=\!\pi_1\!+\!\pi_2\!:
T_xTX_{I^*}=T_xX_{I^*}\!\oplus\!T_xX_{I^*}\lra T_xX_{I^*}
\quad\forall\,x\!\in\!X_{I^*}\,.$$
For each $\ell\!\in\![\ell^*]$, choose an isomorphism 
$$\wt\exp_{\ell}\!:\pi^*\cN_{X_{I^*-\ell}}X_{I^*}\lra \exp^*\cN_{X_{I^*-\ell}}X_{I^*}
\subset TX_{I^*}\!\times\!\cN_{X_{I^*-\ell}}X_{I^*}$$
of vector bundles over~$TX_{I^*}$ restricting to the identity 
over $X_{I^*}\!\subset\!TX_{I^*}$.
Let 
$$\pi_2\!:\exp^*\cN_{X_{I^*-\ell}}X_{I^*}\lra\cN_{X_{I^*-\ell}}X_{I^*}$$ 
be the projection to the second component.\\

\noindent
Suppose $\ell\!\in\![\ell^*\!-\!1]$, $\ell'\!\in\![\ell^*]\!-\![\ell]$, and
\BE{SympNeighCrl_e4}
\Psi_{t;i_1}\big|_{\cN_{I^*;i_1i_2}\cap\Dom(\Psi_{t;i_1})}
=\Psi_{t;i_2}\big|_{\cN_{I^*;i_1i_2}\cap\Dom(\Psi_{t;i_2})}
\quad\forall\,t\!\in\!B\EE
if either $i_1\!\in\![\ell\!-\!1]$ or 
$(i_1,i_2)\!\in\![\ell]\!\times\![\ell'\!-\!1]$.
For each $t\!\in\!B$, let 
\BE{SympNeighCrl_e6}W_t=\Psi_{t;\ell}^{\,-1}\big(\Im(\Psi_{t;\ell'})\big)
\subset\cN_{I^*;\ell\ell'}\!\cap\!\Dom(\Psi_{t;\ell})\,.\EE
By our assumptions,  
\begin{equation*}\begin{split}
\Psi_{t;\ell}\big|_{\cN_{I^*;\ell\ell'}\cap\Dom(\Psi_{t;\ell})}\!:
\cN_{I^*;\ell\ell'}\!\cap\!\Dom(\Psi_{t;\ell})&\lra X_{\ell\ell'} \qquad\hbox{and}\\
\Psi_{t;\ell'}\big|_{\cN_{I^*;\ell\ell'}\cap\Dom(\Psi_{t;\ell'})}\!:
\cN_{I^*;\ell\ell'}\!\cap\!\Dom(\Psi_{t;\ell'})&\lra X_{\ell\ell'}
\end{split}\end{equation*}
are regularizations for~$X_{I^*}$ in $X_{\ell\ell'}$
in the sense of Definition~\ref{smreg_dfn} satisfying 
the stratification condition~\eref{Psikk_e}
with $I\!=\!I^*$ and $I'\!\supset\!\{\ell,\ell'\}$.
Thus, 
\BE{SympNeighCrl_e5} \Th_t\equiv 
\Psi_{t;\ell'}^{\,-1}\!\circ\!\Psi_{t;\ell}|_{W_t}\!:
W_t\lra \cN_{I^*;\ell\ell'}=\cN_{X_{\ell\ell'}}X_{I^*},\EE
is a diffeomorphism onto a neighborhood of $X_{I^*}\!\subset\!\cN_{I^*;\ell\ell'}$
such~that 
\BE{SympNeighCrl_e7} 
\Th_t(x)=x, ~~\nd_x\Th_t=\id~~~\forall\,x\!\in\!X_{I^*},\quad
\Th_t\big(\cN_{I^*;I}\!\cap\!W_t)=\cN_{I^*;I}\!\cap\!\Im(\Th_t)
~~\hbox{if}~\ell,\ell'\in\!I\!\subset\!I^*\,.\EE
By~\eref{SympNeighCrl_e4},
\BE{SympNeighCrl_e9b}
\big(\Th_t|_{\cN_{I^*;i\ell\ell'}\cap W_t}\big)_{t\in B}
=\big(\id_{\cN_{I^*;i\ell\ell'}\cap\Dom(\Psi_{t;\ell'})}\big)_{t\in B}
\quad\forall~i\!\in\![\ell'\!-\!1]\!-\!\ell.\EE
Since $(\Psi_{I^*;t;i})_{i\in I^*}$ and $(\Psi_{U;t;i})_{i\in I^*}$ are regularizations
in the sense of Definition~\ref{ConfRegulLoc_dfn},  \eref{SympNeighCrl_e3a} implies~that 
\BE{SympNeighCrl_e9a} 
\big(\Th_t\big)_{t\in N'(\prt B)}
=\big(\id_{\cN_{I^*;\ell\ell'}\cap\Dom(\Psi_{t;\ell'})}\big)_{t\in N'(\prt B)},~~
\big(\Th_t|_{W_t|_{U'}}\big)_{t\in B}
=\big(\id_{\cN_{I^*;\ell\ell'}\cap\Dom(\Psi_{t;\ell'})|_{U'}}\big)_{t\in B}.\EE

\vspace{.2in}

\noindent
Let $\pi_{\ell\ell'}\!:\cN_{I^*;\ell\ell'}\!\lra\!X_{I^*}$ be the projection map.
For each $t\!\in\!B$, define
\begin{gather*}
W_t'=\big\{v\!\in\!W_t\!:\,
\pi_{\ell\ell'}(\Th_t(v))\!\in\!\exp(T_{\pi_{\ell\ell'}(v)}X_{I^*})\big\}
\subset\cN_{I^*;\ell\ell'},\\
\Th_t^{\hor}\!:W_t'\lra TX_{I^*} \qquad\hbox{by}\quad
\Th_t^{\hor}(v)\in\!T_{\pi_{\ell\ell'}(v)}X_{I^*},~~
\exp\big(\Th_t^{\hor}(v)\big)=\pi_{\ell\ell'}\big(\Th_t(v)\big),\\
\wt\Th_t\!:\pi_{\ell\ell'}^*\cN_{X_{I^*-\ell}}X_{I^*}\big|_{W_t'}
\lra\cN_{I^*;\ell'}\!=\!\cN_{X_{\ell'}}X_{I^*},\quad
\wt\Th_t(v,v_{\ell})=\big(\Th_t(v),\pi_2\big(\wt\exp_{\ell}(\Th_t^{\hor}(v),v_{\ell})\big)\big).
\end{gather*}
Let $\wt{W}_{\ell';t}\!=\!\wt\Th_t^{-1}(\Dom(\Psi_{t;\ell'}))$ and
$\wt\Th_t'\!=\!\wt\Th_t|_{\wt{W}_{\ell';t}}$.
With identifications as in~\eref{cNtot_e}, 
\hbox{$\wt{W}_{\ell';t}\!\subset\!\cN_{I^*;\ell'}$}.
By~\eref{SympNeighCrl_e7},
\BE{SympNeighCrl_e14} 
\wt\Th_t'(x)=x, ~~\nd_x\wt\Th_t'=\id~~~\forall\,x\!\in\!X_{I^*},\quad
\wt\Th_t'\big(\cN_{I^*;I}\!\cap\!\wt{W}_{\ell';t})=\cN_{I^*;I}\!\cap\!\Im(\wt\Th_t')
~~\hbox{if}~\ell'\!\in\!I\!\subset\!I^*.\EE
By~\eref{SympNeighCrl_e9b} and~\eref{SympNeighCrl_e9a},
\begin{gather}
\label{SympNeighCrl_e15b}
\big(\wt\Th_t'|_{\cN_{I^*;i\ell'}\cap\wt{W}_{\ell';t}}\big)_{t\in B}
=\big(\id_{\cN_{I^*;i\ell'}\cap\Dom(\Psi_{t;\ell'})}\big)_{t\in B}
\quad\forall~i\!\in\![\ell'\!-\!1]\!-\!\ell,\\
\label{SympNeighCrl_e15a} 
\big(\wt\Th_t'\big)_{t\in N'(\prt B)}
=\big(\id_{\Dom(\Psi_{t;\ell'})}\big)_{t\in N'(\prt B)},~~~
\big(\wt\Th_t'|_{\wt{W}_{\ell';t}|_{U'}}\big)_{t\in B}
=\big(\id_{\Dom(\Psi_{t;\ell'})|_{U'}}\big)_{t\in B}.
\end{gather}

\vspace{.2in}

\noindent
By~\eref{SympNeighCrl_e14}, the diffeomorphism
$$\Psi_{t;\ell'}'\equiv \Psi_{t;\ell'}\!\circ\!\wt\Th_t'\!:
\wt{W}_{\ell';t}\lra X_{\ell'}$$
is a regularization for~$X_{I^*}$ in~$X_{\ell'}$ for each $t\!\in\!B$
 satisfying 
the stratification condition~\eref{Psikk_e}
with $I\!=\!I^*$ and $I'\!\ni\!\ell'$.
By~\eref{SympNeighCrl_e15a} and~\eref{SympNeighCrl_e3a}, 
\BE{SympNeighCrl_e25}
\big(\Psi_{t;\ell'}'\big)_{t\in N'(\prt B)}=\big(\Psi_{I^*;t;\ell'}\big)_{t\in N'(\prt B)},
~~
\big(\Psi_{t;\ell'}'|_{\Dom(\Psi_{t;\ell'}')|_{U'}}\big)_{t\in B}
=\big(\Psi_{U;t;\ell'}|_{\Dom(\Psi_{U;t;\ell'})|_{U'}}\big)_{t\in B}.\EE
By~\eref{SympNeighCrl_e15b} and~\eref{SympNeighCrl_e4},
\BE{SympNeighCrl_e27}
\Psi_{t;i}\big|_{\cN_{I^*;i\ell'}\cap\Dom(\Psi_{t;i})}
=\Psi_{t;\ell'}'\big|_{\cN_{I^*;i\ell'}\cap\Dom(\Psi_{t;\ell'}')}
\qquad\forall~i\!\in\![\ell\!-\!1].\EE
By \eref{SympNeighCrl_e6} and~\eref{SympNeighCrl_e5},
\BE{SympNeighCrl_e28}\begin{split}
\cN_{I^*;\ell\ell'}\cap\Dom(\Psi_{t;\ell})
\supset  \cN_{I^*;\ell\ell'}&\cap\!\Dom(\Psi_{t;\ell'}')
=\cN_{I^*;\ell\ell'}\!\cap\!W_t', \\
\Psi_{t;\ell}\big|_{\cN_{I^*;\ell\ell'}\cap\Dom(\Psi_{t;\ell'}')}
=&\Psi_{t;\ell'}'\big|_{\cN_{I^*;\ell\ell'}\cap\Dom(\Psi_{t;\ell'}')}\,.
\end{split}\EE
For $i\!\in\![\ell^*]\!-\!\ell'$, let $\Psi_{t;i}'\!=\!\Psi_{t;i}$.\\

\noindent
Choose a neighborhood~$\wt\cW$ of $B\!\times\!X_{I^*}$ in $B\!\times\!\cN_{I^*}$ such~that 
$$\big(\{t\}\!\times\!\cN_{I^*;i}\big)\cap \wt\cW\subset \{t\}\times \Dom(\Psi'_{t;i})
\qquad\forall~i\!\in\!I^*,\,t\!\in\!B.$$
Define $\wt{W}_t''\subset\cN_{I^*}$ by  
$$\bigcup_{t\in B}\{t\}\!\times\!\wt{W}_t''=\wt\cW\cup
\bigcup_{i\in I^*}\!\!\bigg( \bigcup_{t\in N'(\prt B)}\!\!\!\!\!\!\!\Dom(\Psi_{t;i}')
\cup \bigcup_{t\in B}\!\!\Dom(\Psi_{t;i}')|_{U'}\!\!\bigg)\,.$$
For each $t\!\in\!B$ and each $i\!\in\![\ell^*]$,
$$\Psi_{t;i}''\!\equiv\!\Psi_{t;i}'|_{\cN_{I^*;i}\cap \wt{W}_t''}\!:
\cN_{I^*;i}\!\cap\!\wt{W}_t''\lra X_i$$
is a regularization for~$X_{I^*}$ in~$X_i$  satisfying 
the stratification condition~\eref{Psikk_e} with $I\!=\!I^*$ and $I'\!\ni\!i$.
By~\eref{SympNeighCrl_e3c} and the last statement in~\eref{SympNeighCrl_e14},
\eref{SympNeighCrl_e3c} with $\Psi_{t;i}$ replaced by~$\Psi_{t;i}''$ is satisfied.
By~\eref{SympNeighCrl_e3b} and the middle statement in~\eref{SympNeighCrl_e14},
\eref{SympNeighCrl_e3b} with $\Psi_{t;i}$ replaced by~$\Psi_{t;i}''$ is satisfied.
By~\eref{SympNeighCrl_e3a} and \eref{SympNeighCrl_e25},
\eref{SympNeighCrl_e3a} with $\Psi_{t;i}$ replaced by~$\Psi_{t;i}''$ is satisfied.
By~\eref{SympNeighCrl_e4}, \eref{SympNeighCrl_e27}, and \eref{SympNeighCrl_e28}, 
these new regularizations satisfy~\eref{SympNeighCrl_e4} 
with $\Psi_{t;i}$ replaced by~$\Psi_{t;i}''$
whenever
either $i_1\!\in\![\ell\!-\!1]$ or $(i_1,i_2)\!\in\![\ell]\!\times\![\ell']$.
This establishes the claim of the proposition by induction.

\section{Proof of Theorem~\ref{SCC_thm}}
\label{SCCpf_sec}

\noindent
We prove Theorem~\ref{SCC_thm} by induction on the strata of the transverse configuration~$\X$.
For each $I^*\!\in\!\cP^*(N)$, we view $\{X_I\}_{I\in\cP^*(I^*)}$ as 
an $|I^*|$-fold transverse configuration.
Definition~\ref{LocalRegul_dfn} introduces a notion of a weak symplectic regularization for
any transverse configuration~$\X$ over an open subset~$W$ of~$X_{\eset}$,
with $X_{\eset}$ given by~\eref{Xesetdfn_e}.
If  $W$ contains all $X_I$ with $I\!\supsetneq\!I^*$,
a family of such regularizations associated with a family of elements of $\Symp^+(\X)$
extends to a family of weak regularizations for $\{X_I\}_{I\in\cP^*(I^*)}$
over a neighborhood~$W_{I^*}$ of $X_{I^*}$ in~$\X$ after 
deforming the symplectic forms as in~\eref{SCCom_e}; see Corollary~\ref{extendreg_crl}.
Using the operations on regularizations described in Section~\ref{PastingRegul_subs},
we can combine the original family of weak regularizations for~$\X$ over~$W$ and 
the new family of weak regularizations for  $\{X_I\}_{I\in\cP^*(I^*)}$ over 
$W_{I^*}$ into a family of weak regularizations over an open subset~$\wt{W}$
containing all~$X_I$ with  $I\!\supset\!I^*$; see Lemma~\ref{regulcomb_lmm}.
This accomplishes the inductive step in the proof of Theorem~\ref{SCC_thm};
see Proposition~\ref{SCCweak_prp}.
By Lemma~\ref{weakregtoreg_lmm} and Corollary~\ref{weakregtoreg_crl}, 
the difference between a weak regularization for~$\X$ and 
a regularization is insignificant.\\

\noindent
We continue to use the notation introduced in Section~\ref{SCCregul_subs} and
combine it with the notation introduced in Section~\ref{SympDefVB_subs}.
In particular, for a configuration~$\X$ as in Theorem~\ref{SCC_thm},
$$\cN X_I=\bigoplus_{i\in I}\cN_{X_{I-i}}X_I, \quad
\cN_{I;I'}=\bigoplus_{i\in I-I'}\!\!\!\cN_{X_{I-i}}X_I\subset\cN X_I,
\quad \cN_{\prt}X_I=\bigcup_{i\in I}\cN_{I;i}\subset\cN X_I$$
for all $I'\!\subset\!I\!\subset\![N]$ with $|I|\!\ge\!2$.
If in addition $\cN'\!\subset\!\cN X_I$,
$$\cN'_{I;I'}=\cN_{I;I'}\!\cap\!\cN', \qquad
\cN'_{\prt}=\cN_{\prt}X_I\!\cap\!\cN'\,.$$

\subsection{Local weak regularizations}
\label{MainPf_subs}

\noindent
We begin with notions of a weak $\om$-regularization for~$\X$ over an open subset of~$X_{\eset}$
and of an equivalence of two such regularizations.
We then deduce Theorem~\ref{SCC_thm} from several technical statements proved 
in Sections~\ref{ExtendLoc_subs}-\ref{weakregtoreg_subs}.

\begin{dfn}\label{ConfRegulLoc2_dfn}
Let $\X\!\equiv\!\{X_I\}_{I\in\cP^*(N)}$ be a transverse configuration
such that $X_{ij}$ is a closed submanifold of~$X_i$ of codimension~2
for all $i,j\!\in\![N]$ distinct,
$I^*\!\in\!\cP^*(N)$, \hbox{$U\!\subset\!X_{I^*}$} be an open subset, 
and $(\om_i)_{i\in[N]}$ be a symplectic structure on $\X$ in
the sense of Definition~\ref{TransConf_dfn2}.
An \sf{$(\om_i)_{i\in[N]}$-regularization for~$U$ in~$\X$} is a~tuple
$(\rho_i,\na^{(i)},\Psi_i)_{i\in I^*}$
such that $(\Psi_i)_{i\in I^*}$ is a regularization for~$U$ in~$\X$ 
in the sense of Definition~\ref{ConfRegulLoc_dfn} and
$((\rho_j,\na^{(j)})_{j\in I^*-i},\Psi_i)$
is an $\om_i$-regularization for~$U$ in~$X_i$ 
in the sense of Definition~\ref{sympreg1_dfn}\ref{sympreg1_it}
for each $i\!\in\!I^*$. 
\end{dfn}

\begin{dfn}\label{LocalRegul_dfn}
Let $\X$ and $(\om_i)_{i\in[N]}$ be as in Definition~\ref{ConfRegulLoc2_dfn} and
$W\!\subset\!X_{\eset}$ be an open subset.
A \sf{weak $(\om_i)_{i\in[N]}$-regularization for $\X$ over~$W$} is a~tuple
\BE{LocalRegul_e1}
\fR\equiv (\cR_I)_{I\in\cP^*(N)} \equiv
\big(\rho_{I;i},\na^{(I;i)},\Psi_{I;i}\big)_{i\in I\subset[N]}\EE
such~that 
\begin{enumerate}[label=$\bu$,leftmargin=*]

\item $\cR_I$ is an $(\om_i)_{i\in[N]}$-regularization for~$X_I\!\cap\!W$ in~$\X$
for all $I\!\in\!\cP^*(N)$,

\item the associated isomorphism~\eref{fDPsiIIconf_e} of split vector bundles
is a product Hermitian isomorphism~and
\BE{LocalRegul_e2}
\Psi_{I;i}\big|_{\Dom(\Psi_{I;i})\cap\fD\Psi_{I;I'}^{\,-1}(\Dom(\Psi_{I';i}))}
=\Psi_{I';i}\circ\fD\Psi_{I;I'}|_{\Dom(\Psi_{I;i})\cap\fD\Psi_{I;I'}^{\,-1}(\Dom(\Psi_{I';i}))}\EE
for all $i\!\in\!I'\!\subset\!I\!\subset\![N]$ with $|I'|\!\ge\!2$.

\end{enumerate}
\end{dfn}

\vspace{.1in}

\noindent
An $(\om_i)_{i\in[N]}$-regularization for~$\X$ in the sense of 
Definition~\ref{SCCregul_dfn}\ref{SCCreg_it}
is a weak $(\om_i)_{i\in[N]}$-regularization for~$\X$ over $W\!=\!X_{\eset}$
such that 
$$\Dom(\Psi_{I;i})=\fD\Psi_{I;I'}^{\,-1}(\Dom(\Psi_{I';i}))
\qquad\forall\,i\!\in\!I'\!\subset\!I\!\subset\![N],~|I'|\!\ge\!2,$$
as required by the first condition in~\eref{overlap_e}.
By Lemma~\ref{weakregtoreg_lmm}, a weak $(\om_i)_{i\in[N]}$-regularization for~$\X$ 
over $W\!=\!X_{\eset}$ can be cut down to an $(\om_i)_{i\in[N]}$-regularization for~$\X$.
For a smooth family $(\om_{t;i})_{t\in B,i\in[N]}$ of symplectic structures on~$\X$,
we define $(\om_{t;i})_{t\in B,i\in[N]}$-families of regularization for~$U$ in~$\X$
and of weak regularizations for~$\X$ over~$W$ analogously to 
Definition~\ref{SCDregul_dfn}\ref{sympregul_it2}.\\

\noindent
Let $W,W^{(1)},W^{(2)}\!\subset\!X_{\eset}$ be open subsets and
$(\om_{t;i}^{(1)})_{t\in B,i\in[N]}$ and $(\om_{t;i}^{(2)})_{t\in B,i\in[N]}$
be two smooth families of symplectic structures on~$\X$ such~that 
$$W\subset W^{(1)}\cap W^{(2)} \qquad\hbox{and}\quad
\big(\om_{t;i}^{(1)}|_{X_i\cap W}\big)_{t\in B,i\in[N]} 
=\big(\om_{t;i}^{(2)}|_{X_i\cap W}\big)_{t\in B,i\in[N]}.$$
Suppose the tuples
\BE{R1R2_e}\begin{split}
\big(\fR_t^{(1)}\big)_{t\in B}\equiv
\big(\cR_{t;I}^{(1)}\big)_{t\in B,I\in\cP^*(N)}
&\equiv\big(\rho_{t;I;i}^{(1)},\na^{(1),(t;I;i)},\Psi_{t;I;i}^{(1)}
\big)_{t\in B,i\in I\subset[N]},\\
\big(\fR_t^{(2)}\big)_{t\in B}\equiv
\big(\cR_{t;I}^{(2)}\big)_{t\in B,I\in\cP^*(N)}
&\equiv\big(\rho_{t;I;i}^{(2)},\na^{(2),(t;I;i)},\Psi_{t;I;i}^{(2)}\big)_{t\in B,i\in I\subset[N]}
\end{split}\EE
are an $(\om_{t;i}^{(1)})_{t\in B,i\in[N]}$-family of weak regularizations for~$\X$ over~$W^{(1)}$
and an $(\om_{t;i}^{(2)})_{t\in B,i\in[N]}$-family of weak regularizations for~$\X$ over~$W^{(2)}$,
respectively.
We define
$$\big(\fR_t^{(1)}\big)_{t\in B} \cong_W\big(\fR_t^{(2)}\big)_{t\in B}$$
if there exists an  $(\om_{t;i}^{(1)})_{t\in B,i\in[N]}$-family 
\BE{fRtdfn_e}(\fR_t)_{t\in B} \equiv(\cR_{t;I})_{t\in B,I\in\cP^*(N)}
\equiv\big(\rho_{t;I;i},\na^{(t;I;i)},\Psi_{t;I;i}\big)_{t\in B,i\in I\subset [N]}\EE
of weak regularizations for $\X$ over~$W$ such~that 
\begin{gather*}
(\rho_{t;I;i},\na^{(t;I;i)})_{i\in I}=
(\rho_{t;I;i}^{(1)},\na^{(1),(t;I;i)})_{i\in I}\big|_{X_I\cap W},
(\rho_{t;I;i}^{(2)},\na^{(2),(t;I;i)})_{i\in I}\big|_{X_I\cap W},\\
\Dom(\Psi_{t;I;i})\subset\Dom(\Psi_{t;I;i}^{(1)}),\Dom(\Psi_{t;I;i}^{(2)}),
~~~
\Psi_{t;I;i}=\Psi_{t;I;i}^{(1)}\big|_{\Dom(\Psi_{t;I;i})},
\Psi_{t;I;i}^{(2)}\big|_{\Dom(\Psi_{t;I;i})}
\quad\forall\,i\!\in\!I
\end{gather*}
for all $I\!\in\!\cP^*(N)$ and $t\!\in\!B$.
The relation~$\cong_W$ is transitive.
By Corollary~\ref{weakregtoreg_crl}, two regularizations over $W\!=\!X_{\eset}$
that are equivalent as weak regularizations are also equivalent 
as regularizations.
 
\begin{prp}\label{SCCweak_prp}
Let $\X$,  $N'(\prt B)\!\subset\!N(\prt B)\!\subset\!B$,  and $(\om_{t;i})_{t\in B,i\in[N]}$ 
be as in Theorem~\ref{SCC_thm}.
Suppose 
\begin{enumerate}[label=$\bullet$,leftmargin=*]

\item  $I^*\!\in\!\cP^*(N)$ and
$X_{\eset}^*,W,W'\!\subset\!X_{\eset}$ are open subsets such~that 
\BE{SCCweak_e0}\ov{W'}\subset W,\quad
\ov{X_{\eset}^*}\cap X_{I^*}\!\subset W'~~\hbox{if}~|I^*|\!\ge\!3, \quad 
X_I\subset W'~~\forall\,I\!\in\!\cP^*(N),\,I\!\supsetneq\!I^*\,,\EE

\item $(\fR_t)_{t\in N(\prt B)}$ and $(\fR_t')_{t\in B}$
are an $(\om_{t;i})_{t\in N(\prt B),i\in[N]}$-family of weak regularizations for~$\X$ over~$X_{\eset}$
and an $(\om_{t;i})_{t\in B,i\in[N]}$-family of weak regularizations for~$\X$ over~$W$, 
respectively, such~that 
\BE{WeakExt_e0}\big(\fR_t\big)_{t\in N(\prt B)}\cong_W\big(\fR_t'\big)_{t\in N(\prt B)}\,.\EE
\end{enumerate}
Then there exist a neighborhood~$W_{I^*}$ of $X_{I^*}\!\subset\!X_{\eset}$,
a smooth family $(\mu_{t,\tau;i})_{t\in B,\tau\in\bI,i\in[N]}$ of
1-forms on~$X_{\eset}$ such~that 
\BE{WeakExt_e1}\begin{split} 
&\hspace{1in}
\big(\om_{t,\tau;i}\equiv\om_{t;i}\!+\!\nd\mu_{t,\tau;i}\big)_{i\in[N]}
\in \Symp^+(\X) ~\forall\,t\!\in\!B,\,\tau\!\in\!\bI,\\
&\mu_{t,0;i}=0 ~\forall\,t\!\in\!B,\,i\!\in\![N],\quad
\supp\big(\mu_{\cdot,\tau;i}\big)\subset 
\big(B\!-\!N'(\prt B)\big)\!\times\!\big(X_i\!-\!W'\!\cup\!X_{\eset}^*\big)
~\forall\,\tau\!\in\!\bI,\,i\!\in\![N],
\end{split}\EE
and  an $(\om_{t,1;i})_{t\in B,i\in[N]}$-family $(\wt\fR_t)_{t\in B}$
of weak regularizations for~$\X$ over $W'\!\cup\!W_{I^*}$ such~that
\BE{WeakExt_e}
\big(\wt\fR_t\big)_{t\in N'(\prt B)} \cong_{W'\cup W_{I^*}} \!\!\big(\fR_t\big)_{t\in N'(\prt B)}\,,
\quad 
\big(\wt\fR_t\big)_{t\in B} \cong_{W'} \!\big(\fR_t'\big)_{t\in B}\,.\EE
\end{prp}

\begin{proof} 
Let 
$$(\fR_t)_{t\in N(\prt B)}=(\cR_{t;I})_{t\in N(\prt B),I\in\cP^*(N)}
\quad\hbox{and}\quad 
(\fR_t')_{t\in B}=(\cR_{t;I}')_{t\in B,I\in\cP^*(N)}.$$
Choose  a neighborhood $W''$ of $\ov{W'}\!\subset\!X_{\eset}$   such~that 
$\ov{W''}\!\subset\!W$.
By Corollary~\ref{extendreg_crl} with~$W'$ replaced by~$W''$,
there exist 
\begin{enumerate}[label=$\bullet$,leftmargin=*]

\item a neighborhood~$W_{I^*}$ of $X_{I^*}\!\subset\!X_{\eset}$  such that 
$X_I\!\cap\!W_{I^*}\!\subset\!W''$ for all $I\!\in\!\cP(N)\!-\!\cP(I^*)$,

\item a smooth family $(\mu_{t,\tau;i})_{t\in B,\tau\in\bI,i\in[N]}$ of
1-forms on~$X_{\eset}$ satisfying~\eref{WeakExt_e1} with $W'$ replaced by~$W''$,

\item an $(\om_{t,1;i})_{t\in B,i\in[N]}$-family 
$(\wh\cR_{t;I})_{t\in B,I\in\cP^*(I^*)}$
of weak regularizations  for $\{X_I\}_{I\in\cP^*(I^*)}$ over~$W_{I^*}$ 
satisfying
\eref{extendreg_prp_e2a} and \eref{extendreg_prp_e2b} with $W'$ replaced by~$W''$.
\end{enumerate}
In particular, 
$$\om_{t;i}\big|_{X_i\cap W''}=\om_{t,1;i}\big|_{X_i\cap W''}
\qquad\forall\,t\!\in\!B,\,i\!\in\![N].$$

\vspace{.2in}

\noindent
Let $\!W_{I^*}'$ be a neighborhood of $X_{I^*}\!\subset\!X_{\eset}$ 
such~that $\ov{W_{I^*}'}\!\subset\!W_{I^*}$
and $W'''$ be a neighborhood of $\ov{W'}\!\subset\!X_{\eset}$ such~that 
$\ov{W'''}\!\subset\!W''$.
We next apply Lemma~\ref{regulcomb_lmm} with
\begin{gather*}
W=W'', \quad W'=W''', \quad  
(\om_{t;i})_{t\in B,i\in[N]}=\big(\om_{t,1;i}\big)_{t\in B,i\in[N]},\\
(\cR_{t;I})_{t\in B,I\in\cP^*(N)}=(\cR_{t;I}')_{t\in B,I\in\cP^*(N)};
\end{gather*}
the condition~\eref{cRoverlap_e} holds by \eref{extendreg_prp_e2b} with $W'$ replaced by~$W''$.
Thus, there exists an $(\om_{t,1;i})_{t\in B,i\in[N]}$-family 
$(\wt\cR_{t;I})_{t\in B,I\in\cP^*(N)}$
of weak regularizations  for $\X$ over $W'''\!\cup\!W_{I^*}'$
such~that  
\BE{WeakExt_e5}
\big(\wt\cR_{t;I}\big)_{t\in B,I\in\cP^*(N)}
\cong_{W'''}\!\big(\cR_{t;I}'\big)_{t\in B,I\in\cP^*(N)},\quad
\big(\wt\cR_{t;I}\big)_{t\in B,I\in\cP^*(I^*)} 
\cong_{W_{I^*}'} \!\big(\wh\cR_{t;I}\big)_{t\in B,I\in\cP^*(I^*)}\,.\EE
The first equivalence above implies the second equivalence in~\eref{WeakExt_e}.\\

\noindent
By~\eref{WeakExt_e5}, \eref{WeakExt_e0}, and \eref{extendreg_prp_e2a},   
\begin{gather*}
\big(\wt\cR_{t;I}\big)_{t\in N'(\prt B),I\in\cP^*(N)}
\cong_{W'''}
\!\big(\cR_{t;I}\big)_{t\in N'(\prt B),I\in\cP^*(N)}, \\
\big(\wt\cR_{t;I}\big)_{t\in N'(\prt B),I\in\cP^*(I^*)}
\cong_{W_{I^*}'}\!\!
\big(\cR_{t;I}\big)_{t\in N'(\prt B),I\in\cP^*(I^*)}\,.
\end{gather*}
Let $W_{I^*}''$ be a neighborhood of $X_{I^*}\!\subset\!X_{\eset}$ such~that 
$\ov{W_{I^*}''}\!\subset\!W_{I^*}'$.
Applying Corollary~\ref{regulcomb_crl} with
\begin{gather*}
W_{I^*}=W_{I^*}', \quad W_{I^*}'=W_{I^*}'', \quad
W=W''',\quad  B=N'(\prt B),\\
(\om_{t;i})_{t\in B,i\in[N]}=\big(\om_{t,1;i}\big)_{t\in B,i\in[N]},\quad
\cR_{t;I}^{(1)}=\cR_{t;I}, \quad \cR_{t;I}^{(2)}=\wt\cR_{t;I},
\end{gather*}
we obtain the first equivalence in~\eref{WeakExt_e} with $W_{I^*}$ replaced by~$W_{I^*}''$.
\end{proof}

\begin{proof}[{\bf{\emph{Proof of Theorem \ref{SCC_thm}}}}]
Choose a total order~$>$ on subsets $I\!\subset\![N]$ so that 
$I\!>\!I^*$ whenever $I\!\supsetneq\!I^*$.
Suppose $I^*\!\subset\![N]$ with
 $|I^*|\!\ge\!2$ and we have constructed 
\begin{enumerate}[label=$\bullet$,leftmargin=*]

\item a neighborhood~$W_{I^*}^>$ of 
$$X_{I^*}^>\equiv \bigcup_{I> I^*}\!\!X_I\subset X_{\eset},$$

\item a neighborhood $N_{I^*}^>(\prt B)$ of  $\ov{N'(\prt B)}\!\subset\!N(\prt B)$,

\item a smooth family $(\mu_{t,\tau;i})_{t\in B,\tau\in\bI,i\in[N]}$ of
1-forms on~$X_{\eset}$ such~that 
\BE{NC_e0}\begin{split} 
&\hspace{1in}
\big(\om_{t,\tau;i}\equiv\om_{t;i}\!+\!\nd\mu_{t,\tau;i}\big)_{i\in[N]}
\in \Symp^+(\X) ~~\forall\,t\!\in\!B,\,\tau\!\in\!\bI,\\
&\mu_{t,0;i}=0~~\forall\,t\!\in\!B,\,i\!\in\![N], ~~ 
\supp\big(\mu_{\cdot,\tau;i}\big)\subset 
\big(B\!-\!N_{I^*}^>(\prt B)\big)\!\times\!(X_i\!-\!X_i^*)
~~\forall\,\tau\!\in\!\bI,\,i\!\in\![N],\\
\end{split}\EE

\item an $(\om_{t,1;i})_{t\in B,i\in[N]}$-family $(\fR_t')_{t\in B}$
of weak regularizations for $\X$ over~$W_{I^*}^>$ such~that
\BE{NC_e3}\big(\fR'_t\big)_{t\in N_{I^*}^>(\prt B)}
\cong_{W_{I^*}^>} \!\!\big(\fR_t\big)_{t\in N_{I^*}^>(\prt B)}.\EE
\end{enumerate}

\vspace{.1in}

\noindent
Let $W'$ be a neighborhood of $X_{I^*}^>\!\subset\!X_{\eset}$
and $N_{I^*}^{\ge}(\prt B)$
be a neighborhood of $\ov{N'(\prt B)}\!\subset\!N(\prt B)$ such~that 
$$\ov{W'}\subset W_{I^*}^> \quad\hbox{and}\quad 
\ov{N_{I^*}^{\ge}(\prt B)}\subset N_{I^*}^>(\prt B).$$
We apply Proposition~\ref{SCCweak_prp} with 
\begin{gather*}
X_{\eset}^*=\bigcup_{i\in[N]}\!\!\!X_i^*, \quad
W=W_{I^*}^>, \quad N(\prt B)=N_{I^*}^>(\prt B), \quad
N'(\prt B)=N_{I^*}^{\ge}(\prt B), \\
(\om_{t;i})_{t\in B,i\in[N]}=(\om_{t,1;i})_{t\in B,i\in[N]}\,.
\end{gather*}
Thus, there exist 
\begin{enumerate}[label=$\bullet$,leftmargin=*]

\item a neighborhood~$W_{I^*}$ of~$X_{I^*}\!\subset\!X_{\eset}$,

\item a smooth family $(\mu_{t,\tau;i}')_{t\in B,\tau\in\bI,i\in[N]}$ of
1-forms on~$X_{\eset}$ such~that 
\begin{gather*} 
\big(\om_{t,\tau;i}'\equiv\om_{t,1;i}\!+\!\nd\mu_{t,\tau;i}'\big)_{i\in[N]}
\in \Symp^+(\X) ~~\forall\,t\!\in\!B,\,\tau\!\in\!\bI,\\
\mu_{t,0;i}'=0~~\forall\,t\!\in\!B,\,i\!\in\![N], ~~ 
\supp\big(\mu_{\cdot,\tau;i}'\big)\subset 
\big(B\!-\!N_{I^*}^{\ge}(\prt B)\big)\!\times\!\big(X_i\!-\!W'\!\cup\!X_i^*\big)
~~\forall\,\tau\!\in\!\bI,\,i\!\in\![N],
\end{gather*}

\item an $(\om'_{t,1;i})_{t\in B,i\in[N]}$-family $(\wt\fR_t)_{t\in B}$
of weak regularizations for~$\X$ over $W_{I^*}^{\ge}\!\equiv\!W'\!\cup\!W_{I^*}$
so that~\eref{WeakExt_e} holds with $N'(\prt B)$ replaced by~$N_{I^*}^{\ge}(\prt B)$.\\
\end{enumerate}

\noindent
We concatenate the families $(\mu_{t,\tau;i})_{t\in B,\tau\in\bI,i\in[N]}$ and
$(\mu_{t,1;i}\!+\!\mu_{t,\tau;i}')_{t\in B,\tau\in\bI,i\in[N]}$ of 1-forms on~$\X$ 
into a new smooth family 
$(\mu_{t,\tau;i})_{t\in B,\tau\in\bI,i\in[N]}$ 
such~that \eref{NC_e0} holds with $N_{I^*}^>(\prt B)$ 
replaced by~$N_{I^*}^{\ge}(\prt B)$.
By the first equivalence in~\eref{WeakExt_e} with $N'(\prt B)$ replaced by~$N_{I^*}^{\ge}(\prt B)$, 
\eref{NC_e3} holds with 
$N_{I^*}^>(\prt B)$ and $W_{I^*}^>$ replaced by 
$N_{I^*}^{\ge}(\prt B)$ and~$W_{I^*}^{\ge}$, respectively.\\

\noindent
By the downward induction on $\cP^*(N)$ with respect to~$<$, 
we thus obtain
a family $(\mu_{t,\tau;i})_{t\in B,\tau\in\bI,i\in[N]}$ of 
1-forms on~$\X$ satisfying~\eref{SCCom_e}
and an $(\om_{t,1;i})_{t\in B,i\in[N]}$-family 
$(\fR_t')_{t\in B}$ of weak regularizations for~$\X$ over~$X_{\eset}$ 
such~that
$$\big(\fR_t'\big)_{t\in N'(\prt B)} \cong_{X_{\eset}}\!\!\big(\fR_t\big)_{t\in N'(\prt B)}\,.$$
By Lemma~\ref{weakregtoreg_lmm}, these weak regularizations can be cut down  
to an $(\om_{t,1;i})_{t\in B,i\in[N]}$-family $(\wt\fR_t)_{t\in B}$ 
of regularizations for~$\X$. 
In particular,
$$(\wt\fR_t)_{t\in N'(\prt B)}\cong_{X_{\eset}}\!\!
(\fR_t')_{t\in N'(\prt B)}\cong_{X_{\eset}}\!\!
\big(\fR_t\big)_{t\in N'(\prt B)}\,.$$
By Corollary~\ref{weakregtoreg_crl}, this implies~\eref{SCCom_e2}.
\end{proof}

\subsection{Extending weak regularizations}
\label{ExtendLoc_subs}

\noindent
Lemma~\ref{extendreg0_lmm} below is the main step in the proof of 
Proposition~\ref{SCCweak_prp} which provides for extensions of weak regularizations.
This lemma implements the deformations for symplectic forms on split 
vector bundles obtained in Theorem~\ref{SympDefVB_thm} via Proposition~\ref{TubulNeigh_prp}.

\begin{lmm}\label{extendreg0_lmm}
Let $\X$,  $N'(\prt B)\!\subset\!N(\prt B)\!\subset\!B$, $(\om_{t;i})_{t\in B,i\in[N]}$,
$I^*$,  $X_{\eset}^*$, and $W'\!\subset\!W$ be as in 
Proposition~\ref{SCCweak_prp}. 
Suppose 
$$\big(\rho_{I^*;t;i},\na^{(I^*;t;i)},\Psi_{I^*;t;i}\big)_{t\in N(\prt B),i\in I^*}
\quad\hbox{and}\quad 
\big(\rho_{W;t;i},\na^{(W;t;i)},\Psi_{W;t;i}\big)_{t\in B,i\in I^*}$$
are $(\om_{t;i})_{t\in B}$-families of regularizations for~$X_{I^*}$ and $X_{I^*}\!\cap\!W$, 
respectively, in $\X$ such~that 
\BE{extendreg0_e0b}\begin{split}
&\big(\rho_{I^*;t;i}|_{X_{I^*}\cap W},\na^{(I^*;t;i)}|_{X_{I^*}\cap W},
\Psi_{I^*;t;i}|_{\Dom(\Psi_{I^*;t;i})|_{X_{I^*}\cap W}}\big)_{t\in N(\prt B),i\in I^*}\\
&\hspace{2in}
= \big(\rho_{W;t;i},\na^{(W;t;i)},\Psi_{W;t;i}\big)_{t\in N(\prt B),i\in I^*}\,. 
\end{split}\EE 
Then there exist  a smooth family $(\mu_{t,\tau;i})_{t\in B,\tau\in\bI,i\in[N]}$ of
1-forms on~$X_{\eset}$ satisfying~\eref{WeakExt_e1} and
an $(\om_{t,1;i})_{t\in B,i\in[N]}$-family 
$(\rho_{t;i},\na^{(t;i)},\Psi_{t;i})_{t\in B,i\in I^*}$ of regularizations 
for $X_{I^*}$ in~$\X$ such that 
\BE{extendreg0_e1a}\begin{split}
&\big(\rho_{t;i},\na^{(t;i)},\Psi_{t;i}\big)_{t\in N'(\prt B),i\in I^*}
=\big(\rho_{I^*;t;i},\na^{(I^*;t;i)},\Psi_{I^*;t;i}\big)_{t\in N'(\prt B),i\in I^*},\\
&\big((\rho_{t;i},\na^{(t;i)})|_{X_{I^*}\cap W'},
\Psi_{t;i}|_{\Dom(\Psi_{t;i})|_{X_{I^*}\cap W'}}\big)_{t\in B,i\in I^*}\\
&\hspace{1in}=\big((\rho_{W;t;i},\na^{(W;t;i)})|_{X_{I^*}\cap W'},
\Psi_{W;t;i}|_{\Dom(\Psi_{W;t;i})|_{X_{I^*}\cap W'}}\big)_{t\in B,i\in I^*}.
\end{split}\EE
\end{lmm}

\begin{proof} For each $i\!\in\!I^*$, let 
$$L_i=\cN_{X_{I^*-i}}X_{I^*}\lra X_{I^*}\,,\qquad
\cN_{I^*;i}=\bigoplus_{j\in I^*-i}\!\!\!\!L_j
\subset \bigoplus_{j\in I^*}\!L_j=\cN X_{I^*}\,.$$
If in addition $t\!\in\!B$, define
$$\om_{t;I^*}=\om_{t;i}|_{X_{I^*}}, \qquad
\Om_{t;i}^{\bu}=\bigoplus_{j\in I^*-i}\!\!\!\om_{t;i}|_{L_j}\,.$$
Choose a neighborhood $W''$ of $\ov{W'}\!\subset\!X_{\eset}$
such that $\ov{W''}\!\subset\!W$.\\

\noindent
Since $(\om_{t;i})_{i\in[N]}\!\in\!\Symp(\X)$,
$\om_{t;i}|_{L_i}$ is symplectic for every $i\!\in\!I^*$.
For each $i\!\in\!I^*$, choose 
a smooth family $(\rho_{t;i},\na^{(t;i)})_{t\in B}$ 
of $\om_{t;i}|_{L_i}$-compatible Hermitian structures on~$L_i$ such~that 
\BE{extendreg_e3}\begin{split}
\big(\rho_{t;i},\na^{(t;i)}\big)_{t\in N'(\prt B)}
&=\big(\rho_{I^*;t;i},\na^{(I^*;t;i)}\big)_{t\in N'(\prt B)},\\
\big( (\rho_{t;i},\na^{(t;i)})|_{X_{I^*}\cap W''}\big)_{t\in B} 
&= \big( (\rho_{W;t;i},\na^{(W;t;i)})|_{X_{I^*}\cap W''}\big)_{t\in B}\,;
\end{split}\EE
this is possible to do by~\eref{extendreg0_e0b}.
For each $t\!\in\!B$, denote by $(\wh\om_{t;i}^{\bu})_{i\in I^*}$ 
the closed 2-form on $\cN_{\prt}X_{I^*}$ induced by~$\om_{t;I^*}$,
the diagonal fiberwise 2-form  $(\Om_{t;i}^{\bu})_{i\in I^*}$
on $\cN_{\prt}X_{I^*}$, and $(\na^{(t;i)})_{i\in I^*}$ as in~\eref{ombund_e2}.
By~\eref{Psiomcond_e} and~\eref{extendreg_e3}, 
\BE{extendreg_e8}
\begin{split}
\big(\Psi_{I^*;t;i}^{\,*}\om_{t;i}\big)_{t\in N'(\prt B),i\in I^*}
&=\big(\wh\om_{t;i}^{\bu}|_{\Dom(\Psi_{I^*;t;i})}\big)_{t\in N'(\prt B),i\in I^*},\\
\big(\Psi_{W;t;i}^{\,*}\om_{t;i}|_{\Dom(\Psi_{W;t;i})|_{X_{I^*}\cap W''}}\big)_{t\in B,i\in I^*}
&=\big(\wh\om_{t;i}^{\bu}|_{\Dom(\Psi_{W;t;i})|_{X_{I^*}\cap W''}}\big)_{t\in B,i\in I^*}.
\end{split}\EE

\vspace{.1in}

\noindent
By~\eref{Psiomcond_e} and~\eref{ombund_e2}, 
\BE{extendreg_prp_e4}\begin{aligned}
\nd_x\Psi_{I^*;t;i}\big(\cN_{I^*;i}|_x\big)&=T_xX_{I^*}^{\,\om_{t;i}}
&\quad&\forall\,(t,x)\!\in\!N(\prt B)\!\times\!X_{I^*}, \\
\nd_x\Psi_{W;t;i}\big(\cN_{I^*;i}|_x\big)&=T_xX_{I^*}^{\,\om_{t;i}}
&\quad&\forall\,(t,x)\!\in\!B\!\times\!(X_{I^*}\!\cap\!W).
\end{aligned}\EE
We first apply Proposition~\ref{TubulNeigh_prp} with
$$U=X_{I^*}\!\cap\!W,\qquad U'=X_{I^*}\!\cap\!W'', \quad
\big(\Psi_{U;t;i}\big)_{t\in B,i\in I^*}=
\big(\Psi_{W;t;i}\big)_{t\in B,i\in I^*}\,;$$
the conditions~\eref{TubulNeigh_e0a} and~\eref{TubulNeigh_e0b} are satisfied by
\eref{extendreg_prp_e4} and~\eref{extendreg0_e0b}, respectively.
There thus exists a smooth family $(\Psi_{t;i})_{t\in B,i\in I^*}$ of 
regularizations for~$X_{I^*}$ in~$\X$ in the sense of 
Definition~\ref{ConfRegulLoc_dfn} such~that 
\begin{gather}\label{extendreg_e4}
\nd_x\Psi_{t;i}\big(\cN_{I^*;i}|_x\big)=T_xX_{I^*}^{\,\om_{t;i}}
\quad\forall\,t\!\in\!B,\,x\!\in\!\!X_{I^*},\,i\!\in\!I^*,\\
\label{extendreg_e4a}
\begin{split}
\big(\Psi_{t;i}\big)_{t\in N'(\prt B),i\in I^*}&
=\big(\Psi_{I^*;t;i}\big)_{t\in N'(\prt B),i\in I^*},\\
\big(\Psi_{t;i}|_{\Dom(\Psi_{t;i})|_{X_{I^*}\cap W''}}\big)_{t\in B,i\in I^*}
&=\big(\Psi_{W;t;i}|_{\Dom(\Psi_{W;t;i})|_{X_{I^*}\cap W''}}\big)_{t\in B,i\in I^*}.
\end{split}\end{gather}
The requirement~\eref{extendreg0_e1a} holds by~\eref{extendreg_e3} 
and~\eref{extendreg_e4a}.\\

\noindent
For $t\!\in\!B$ and $i\!\in\!I^*$, let $\wt\om_{t;i}\!=\!\Psi_{t;i}^{\,*}\om_{t;i}$.
By the last condition in Definition~\ref{smreg_dfn} and~\eref{extendreg_e4}, 
$(\wh\om_{t;i}^{\bu})_{t\in B,i\in I^*}$ is the smooth family of 
diagonalized 2-forms on $\cN_{\prt}X_{I^*}$ 
determined by $(\wt\om_{t;i})_{t\in B,i\in I^*}$ and $(\na^{(t;i)})_{t\in B,i\in I^*}$
in the terminology of Theorem~\ref{SympDefVB_thm}.
By~\eref{extendreg_e4a} and~\eref{extendreg_e8},
\BE{extendreg_e5}\begin{split}
\big(\wt\om_{t;i}\big)_{t\in N'(\prt B),i\in I^*}
&=\big(\wh\om_{t;i}^{\bu}|_{\Dom(\Psi_{t;i})}\big)_{t\in N'(\prt B),i\in I^*},\\
\big(\wt\om_{t;i}|_{\Dom(\Psi_{t;i})|_{X_{I^*}\cap W''}}\big)_{t\in B,i\in I^*}
&=\big(\wh\om_{t;i}^{\bu}|_{\Dom(\Psi_{t;i})|_{X_{I^*}\cap W''}}\big)_{t\in B,i\in I^*}.
\end{split}\EE
By the compactness of~$B$, 
there exists a neighborhood $\cN'$ of $X_{I^*}\!\subset\!\cN X_{I^*}$ such~that 
$\cN_i'\!\subset\!\Dom(\Psi_{t;i})$ for all $i\!\in\!I^*$ and $t\!\in\!B$.\\

\noindent
Suppose $|I^*|\!\ge\!3$.
Since
$$\ov{W}'\subset W'',\quad 
\ov{X_{\eset}^*}\cap X_{I^*}\subset W', \quad 
X_{I^*}\!\cap\!X_I=X_{I^*\cup I}\subset W'
~~\forall\,I\!\in\!\cP(N)\!-\!\cP(I^*),$$
we can shrink~$\cN'$ so that
\begin{gather}\label{nonoverlap_e}
\cN'\cap\Psi_{t;i}^{-1}\big(\ov{W}'\big)\subset\cN_i'\big|_{X_{I^*}\cap W''},\quad
\cN'\cap\Psi_{t;i}^{-1}\big(\ov{X_{\eset}^*}\big)\subset\cN_i'|_{X_{I^*}\cap W'},\\
\label{nonoverlap_e2}
\cN'\cap\Psi_{t;i}^{-1}\big(X_I\big)\subset\cN_i'\big|_{X_{I^*}\cap W'}
\quad\forall~I\!\in\!\cP(N)\!-\!\cP(I^*)
\end{gather}
for all $i\!\in\!I^*$ and $t\!\in\!B$.
We next apply Theorem~\ref{SympDefVB_thm} with
$$V=X_{I^*}, ~~ I=I^*,~~
U=X_{I^*}\!\cap\!W'', ~~ N(\prt B)=N'(\prt B),~~
\big(\wt\om_{t;i}\big)_{t\in B,i\in I}
=\big(\wt\om_{t;i}|_{\cN'_i}\big)_{t\in B,i\in I^*};$$
the requirements in~\eref{SympDefVB_e0} are satisfied by~\eref{extendreg_e5}.
Thus,  
there exist neighborhoods $\wh\cN,\cN''\!$ of $X_{I^*}\!\subset\!\cN'$ 
such that $\ov{\cN''}\!\subset\!\cN'$
and a smooth family  $(\mu'_{t,\tau;i})_{t\in B,\tau\in\bI,i\in I^*}$ of
1-forms on~$\cN_{\prt}'$ such~that 
\BE{extendreg_e12a}
\big(\wt\om_{t,\tau;i}\big)_{i\in I^*}\equiv
\big(\wt\om_{t;i}\!+\!\nd\mu'_{t,\tau;i}\big)_{i\in I^*}\EE
is a symplectic structure on~$\cN_{\prt}'$ for all $(t,\tau)\!\in\!B\!\times\!\bI$ and
\BE{extendreg_e12b}
\mu'_{t,0;i}=0,  \quad
\wt\om_{t,1;i}|_{\wh\cN_i}=\wh\om_{t;i}^{\bu}|_{\wh\cN_i},
\quad\supp\big(\mu'_{\cdot,\tau;i}\big)\subset 
\big(B\!-\!N'(\prt B)\big)\!\times\!\cN''|_{X_{I^*}-W''}\EE
for all $t\!\in\!B$, $\tau\!\in\!\bI$, and $i\!\in\!I^*$.\\

\noindent
For each $i\!\in\!I^*$,
define a smooth family $(\mu_{t,\tau;i})_{t\in B,\tau\in\bI}$ 
of 1-forms on~$X_i$ by
$$\mu_{t,\tau;i}\big|_x=\begin{cases}
0,&\hbox{if}~x\!\in\!X_i\!-\!\Psi_{t;i}(\ov\cN_i'');\\
\mu_{t,\tau;i}'|_{\Psi_{t;i}^{-1}(x)}
\!\circ\!\nd_x\Psi_{t;i}^{-1}, &\hbox{if}~x\!\in\!\Psi_{t;i}(\cN_i').
\end{cases}$$
By~\eref{extendreg_e12b}, 
\begin{gather}\label{extendreg_e23}
\mu_{t,0;i}=0, ~~~
\Psi_{t;i}^*\big\{\om_{t;i}\!+\!\nd\mu_{t,1;i}\big\}\big|_{\wh\cN_i}
=\wh\om_{t;i}^{\bu}\big|_{\wh\cN_i},\\
\notag
\supp\big(\mu_{\cdot,\tau;i}\big)\subset 
\big(B\!-\!N'(\prt B)\big)\!\times\!
\big(\Psi_{t;i}(\cN_i'')\!-\!\Psi_{t;i}\big(\cN_i'|_{X_{I^*}\cap W''}\big)\big)
\end{gather} 
for all $t\!\in\!B$, $\tau\!\in\!\bI$, and $i\!\in\!I^*$.
By the last statement, \eref{nonoverlap_e}, and~\eref{nonoverlap_e2}, 
\begin{gather}
\label{extendreg_e23b}
\supp\big(\mu_{\cdot,\tau;i}\big)\subset 
\big(B\!-\!N'(\prt B)\big)\!\times\!\big(X_i\!-\!W'\!\cup\!X_{\eset}^*\big)
\quad\forall~\tau\!\in\!\bI,\,i\!\in\!I^*,\\
\label{extendreg_e23c}
\mu_{t,\tau;i}|_{X_I}=0 \quad\forall~
I\!\in\!\cP(N)\!-\!\cP(I^*),\,t\!\in\!B,\,\tau\!\in\!\bI,\,i\!\in\!I^*.
\end{gather}
For $i\!\not\in\!I^*$, we set $\mu_{t,\tau;i}\!=\!0$ for all $(t,\tau)\!\in\!B\!\times\!\bI$.
Since $(\mu'_{t,\tau;i})_{i\in I^*}$ is a 1-form on~$\cN_{\prt}'$,
\eref{extendreg_e23c} implies~that 
$$\mu_{t,\tau;i_1}|_{X_{i_1i_2}}=\mu_{t,\tau;i_2}|_{X_{i_1i_2}}
\qquad\forall\,i_1,i_2\!\in\![N].$$
Since $(\wt\om_{t,\tau;i})_{i\in I^*}$ is a symplectic structure on~$\cN_{\prt}'$ 
for all $(t,\tau)\!\in\!B\!\times\!\bI$, we conclude that  the tuple
$(\mu_{t,\tau;i})_{t\in B,\tau\in\bI,i\in[N]}$ is a smooth family of
1-forms on~$X_{\eset}$ satisfying~\eref{WeakExt_e1}.
By the second statement in~\eref{extendreg_e23}, 
$((\rho_{t;j},\na^{(t;j)})_{j\in I^*-i},\Psi_{t;i})$
is an $\om_{t,1;i}$-regularization for~$X_{I^*}$ in~$X_i$ 
in the sense of Definition~\ref{sympreg1_dfn}\ref{sympreg1_it}
for all $t\!\in\!B$ and $i\!\in\!I^*$.\\ 

\noindent
Suppose $|I^*|\!=\!2$.
Denote by $\ze_{I^*}\!\equiv\!\ze_{\cN X_{I^*}}$  the radial vector field 
on the total space of $\cN X_{I^*}$  as defined above~\eref{ombund_e2}.
For each $\tau\!\in\!\R$, let
$$m_{\tau}\!:\cN X_{I^*}\lra\cN X_{I^*}, \qquad v\lra\tau v,$$
be the scalar multiplication map;
it preserves the subbundles $\cN_{I^*;i}\!\subset\!\cN X_{I^*}$.
For  $t\!\in\!B$ and $i\!\in\!I^*$, define a 2-form and a 1-form on~$\cN_i'$ by
$$\vp_{t;i}=\wh\om_{t;i}^{\bu}\!-\!\wt\om_{t;i}, \qquad
\mu_{t;i}\big|_v=\int_0^1\!\!
m_{\tau}^*\big\{\vp_{t;i}(\tau^{-1}\ze_{I^*},\cdot)\big\}\nd\tau\,.$$
Since $\vp_{t;i}$ vanishes on $T\cN_{I^*;i}|_{X_I^*}$, the integrand above 
extends smoothly over $\tau\!=\!0$.
We also note~that 
\BE{zettau2_e}
\mu_{t;i}\big|_{X_{I^*}}=0~~\forall\,t\!\in\!B,\quad
\mu_{t;i}=0~~\forall\,t\!\in\!N'(\prt B),\quad
\mu_{t;i}\big|_{\cN_i'|_{X_{I^*}\cap W''}}=0~~\forall\,t\!\in\!B.\EE
The first equality above is immediate from the definition of~$\mu_{t;i}$,
while the other two follow from~\eref{extendreg_e5}.\\

\noindent
Shrinking $\cN'$ if necessary, we can assume that the restrictions of 
the closed 2-forms $\wt\om_{t;i}\!+\!\tau\vp_{t;i}$ to~$\cN_i'$ are nondegenerate for all 
$(t,\tau)\!\in\!B\!\times\!\bI$ and $i\!\in\!I^*$.
Let $\xi_{t,\tau;i}$ be the vector field on~$\cN_i'$ given~by
$$\big\{\wt\om_{t;i}\!+\!\tau\vp_{t;i}\}\big(\xi_{t,\tau;i},\cdot\big)=\mu_{t;i}(\cdot);$$
it corresponds to the negative of the vector field $X_{\tau}$ below \cite[(3.7)]{MS1}.
By~\eref{zettau2_e},
\BE{zettau_e3}\xi_{t,\tau;i}\big|_{X_{I^*}}=0~~\forall\,t\!\in\!B,\quad
\xi_{t,\tau;i}=0~~\forall\,t\!\in\!N'(\prt B),\quad
\xi_{t,\tau;i}|_{\cN_i'|_{X_{I^*}\cap W''}}=0~~\forall\,t\!\in\!B.\EE
By the compactness of~$B$, there exists a neighborhood $\cN''$ of $X_{I^*}\!\subset\!\cN'$
such~that the time~1 flow~$\psi_{t;i}$ of~$\xi_{t,\tau;i}$ 
is defined on~$\cN''$ (and takes values in~$\cN'$).
By~\eref{zettau_e3},
\BE{zettau_e5}\begin{split}
&\hspace{1in}\psi_{t;i}\big|_{X_{I^*}}=\id_{X_{I^*}}~~\forall\,t\!\in\!B,\\
&\psi_{t;i}=\id_{\cN_i''}~~\forall\,t\!\in\!N'(\prt B),\quad
\psi_{t;i}|_{\cN_i''|_{X_{I^*}\cap W''}}=\id_{\cN_i''|_{X_{I^*}\cap W''}}~~\forall\,t\!\in\!B.
\end{split}\EE
By the proof of \cite[Lemma~3.14]{MS1}, 
\BE{pullbackom_e}
\psi_{t;i}^*\wt\om_{t;i}\big|_{\cN_i''}=\wh\om_{t;i}^{\bu}\big|_{\cN_i''}
\qquad\forall\,t\!\in\!B,\,i\!\in\!I\,.\EE

\vspace{.1in}

\noindent
For each $i\!\in\!I^*$, the set 
$$\cW\equiv B\!\times\!\cN_i''\cup
\bigg(\bigcup_{t\in N'(\prt B)}\!\!\!\!\!\!\!\{t\}\!\times\!\Dom(\Psi_{t;i})\!\!\bigg)
\cup
\bigg(\bigcup_{t\in B}\{t\}\!\times\!\Dom(\Psi_{t;i})\big|_{X_{I^*}\cap W''}\bigg)$$
is an open neighborhood of $B\!\times\!X_{I^*}$ in $B\!\times\!\cN_{I^*;i}$.
Let
$$\Th_{\cdot;i}\!:\cW\lra\cN_{I^*;i}, \quad 
\Th_{t;i}(v)=\begin{cases}\psi_{t;i}(v),&\hbox{if}~v\!\in\!\cN_i'',\\
v,&\hbox{if}~t\!\in\!N'(\prt B),\\
v,&\hbox{if}~v\!\in\!\Dom(\Psi_{t;i})\big|_{X_{I^*}\cap W''}.
\end{cases}$$
By~\eref{zettau_e5}, \eref{pullbackom_e}, and~\eref{extendreg_e5}, 
this map is well-defined and
\BE{zettau_e7}\begin{split}
&\hspace{.7in}
\Th_{t;i}\big|_{X_{I^*}}=\id_{X_{I^*}}~~\forall\,t\!\in\!B,\quad
\Th_{t;i}^*\wt\om_{t;i}\big|_{\Dom(\Th_{t;i})}=\wh\om_{t;i}^{\bu}\big|_{\Dom(\Th_{t;i})},\\
&\Th_{t;i}=\id_{\Dom(\Psi_{t;i})}~~\forall\,t\!\in\!N'(\prt B),\quad
\psi_{t;i}|_{\Dom(\Psi_{t;i})|_{X_{I^*}\cap W''}}
=\id_{\Dom(\Psi_{t;i})|_{X_{I^*}\cap W''}}~~\forall\,t\!\in\!B.
\end{split}\EE
For each $t\!\in\!B$, $\Th_{t;i}$ is a diffeomorphism onto an open subset 
of~$\Dom(\Psi_{t;i})$.
Since $|I^*|\!=\!2$, the tuple $(\Psi_{t;i}\!\circ\!\Th_{t;i})_{t\in B,i\in I^*}$ 
is thus a smooth family of regularizations for~$X_{I^*}$ in~$\X$ 
satisfying~\eref{extendreg_e4a}.
By the second statement in~\eref{zettau_e7}, 
$((\rho_{t;j},\na^{(t;j)})_{j\in I^*-i},\Psi_{t;i}\!\circ\!\Th_{t;i})$
is an $\om_{t;i}$-regularization for~$X_{I^*}$ in~$X_i$
for all $t\!\in\!B$ and $i\!\in\!I^*$.
\end{proof}

\begin{crl}\label{extendreg_crl}
Let $\X$,  $N'(\prt B)\!\subset\!N(\prt B)\!\subset\!B$, $(\om_{t;i})_{t\in B,i\in[N]}$,
$I^*$,  $X_{\eset}^*$, and $W'\!\subset\!W$ be as in 
Proposition~\ref{SCCweak_prp}.
Suppose $(\cR_{t;I})_{t\in N(\prt B),I\in\cP^*(N)}$ and $(\cR_{t;I}')_{t\in B,I\in\cP^*(I^*)}$
are an $(\om_{t;i})_{t\in N(\prt B),i\in[N]}$-family of weak regularizations for~$\X$ over~$X_{\eset}$
and an $(\om_{t;i})_{t\in B,i\in I^*}$-family of weak regularizations for 
$\{X_I\}_{I\in\cP^*(I^*)}$ over~$W$, respectively, such~that 
\BE{extendreg_prp_e1}\big(\cR_{t;I}\big)_{t\in N(\prt B),I\in\cP^*(I^*)}
\cong_W\big(\cR_{t;I}'\big)_{t\in N(\prt B),I\in\cP^*(I^*)}\,.\EE
Then there exist a neighborhood~$W_{I^*}$ of $X_{I^*}\!\subset\!X_{\eset}$
such~that  
\BE{extendreg_prp_e2} X_I\cap W_{I^*}\subset W' \qquad\forall~I\in\cP(N)\!-\!\cP(I^*),\EE
a smooth family $(\mu_{t,\tau;i})_{t\in B,\tau\in\bI,i\in[N]}$ of
1-forms on~$X_{\eset}$ satisfying~\eref{WeakExt_e1}, and
an $(\om_{t,1;i})_{t\in B,i\in[N]}$-family $(\wh\cR_{t;I})_{t\in B,I\in\cP^*(I^*)}$
of weak regularizations  for $\{X_I\}_{I\in\cP^*(I^*)}$ over~$W_{I^*}$
such~that 
\begin{alignat}{1}
\label{extendreg_prp_e2a}
\big(\wh\cR_{t;I}\big)_{t\in N'(\prt B),I\in\cP^*(I^*)}
&\cong_{W_{I^*}}\!\!\big(\cR_{t;I}\big)_{t\in N'(\prt B),I\in\cP^*(I^*)},\\
\label{extendreg_prp_e2b}
\big(\wh\cR_{t;I}\big)_{t\in B,I\in\cP^*(I^*)} 
&\cong_{W'\cap W_{I^*}} \!\!\big(\cR_{t;I}'\big)_{t\in B,I\in\cP^*(I^*)}\,.
\end{alignat}
\end{crl}

\begin{proof} 
Let 
\begin{gather*}
(\cR_{t;I})_{t\in N(\prt B),I\in\cP^*(N)} = 
\big(\rho_{t;I;i},\na^{(t;I;i)},\Psi_{t;I;i}\big)_{t\in N(\prt B),i\in I\subset[N]},\\
(\cR'_{t;I})_{t\in B,I\in\cP^*(I^*)} =  
\big(\rho'_{t;I;i},\na'^{(t;I;i)},\Psi'_{t;I;i}\big)_{t\in B,i\in I\subset I^*}.
\end{gather*}
Choose open subsets $W'',W'''\!\subset\!X_{\eset}$ and $N''(\prt B)\!\subset\!B$
such~that 
$$\ov{W'}\subset W'',\quad \ov{W''}\subset W''', \quad \ov{W'''}\subset W, \quad
\ov{N'(\prt B)}\!\subset N''(\prt B), \quad \ov{N''(\prt B)}\subset N(\prt B).$$
By \eref{extendreg_prp_e1} and the compactness of~$B$, 
there exist a neighborhood $\cN^{\circ}$ of~$X_{I^*}\!\subset\!\cN X_{I^*}$
such~that 
\begin{gather}
\notag
\cN_i^{\circ}\subset\Dom(\Psi_{t;I^*;i})~~\forall\,t\!\in\!N''(\prt B),\,i\!\in\!I^*, \quad
\cN_i^{\circ}\subset\Dom(\Psi_{t;I^*;i}')
~~\forall\,t\!\in\!B,\,i\!\in\!I^*, \\
\label{extendreg_prp_e6}
\big(\Psi_{t;I^*;i}|_{\cN_i^{\circ}|_{X_{I^*}\cap W'''}}\big)_{t\in N''(\prt B),i\in I^*}
=\big(\Psi_{t;I^*;i}'|_{\cN_i^{\circ}|_{X_{I^*}\cap W'''}}\big)_{t\in N''(\prt B),i\in I^*}\,.
\end{gather}

\vspace{.1in}

\noindent
We apply Lemma~\ref{extendreg0_lmm} with
\begin{gather*}
N(\prt B)=N''(\prt B), \quad W=W''', \quad W'=W'', \\
\begin{split}
\big(\rho_{I^*;t;i},\na^{(I^*;t;i)},\Psi_{I^*;t;i}\big)_{t\in N(\prt B),i\in I^*}
&=\big(\rho_{t;I^*;i},\na^{(t;I^*;i)},
\Psi_{t;I^*;i}|_{\cN_i^{\circ}}\big)_{t\in N''(\prt B),i\in I^*},\\
\big(\rho_{W;t;i},\na^{(W;t;i)},\Psi_{W;t;i}\big)_{t\in B,i\in I^*}
&=\big(\rho_{t;I^*;i}',\na'^{(t;I^*;i)},\Psi_{t;I^*;i}'|_{\cN_i^{\circ}}\big)_{t\in B,i\in I^*}\,;
\end{split}\end{gather*}
the requirement~\eref{extendreg0_e0b} is satisfied due to~\eref{extendreg_prp_e1}
and~\eref{extendreg_prp_e6}.
There thus exist a smooth family $(\mu_{t,\tau;i})_{t\in B,\tau\in\bI,i\in[N]}$ of
\hbox{1-forms} on~$X_{\eset}$ satisfying~\eref{WeakExt_e1} with $W'$ replaced
by $W''\!\supset\!W'$  and an $(\om_{t,1;i})_{t\in B,i\in[N]}$-family 
$(\rho_{t;i},\na^{(t;i)},\Psi_{t;i})_{t\in B,i\in I^*}$ of regularizations 
for $X_{I^*}$ in~$\X$ such that 
\BE{extendreg0_e21}\begin{split}
&\big(\rho_{t;i},\na^{(t;i)},\Psi_{t;i}\big)_{t\in N'(\prt B),i\in I^*}
=\big(\rho_{t;I^*;i},\na^{(t;I^*;i)},\Psi_{t;I^*;i}|_{\cN_i^{\circ}}\big)_{t\in N'(\prt B),i\in I^*},\\
&\big((\rho_{t;i},\na^{(t;i)})|_{X_{I^*}\cap W''},
\Psi_{t;i}|_{\Dom(\Psi_{t;i})|_{X_{I^*}\cap W''}}\big)_{t\in B,i\in I^*}\\
&\hspace{1in}=\big((\rho'_{t;i},\na'^{(t;i)})|_{X_{I^*}\cap W''},
\Psi_{t;I^*;i}'|_{\cN_i^{\circ}|_{X_{I^*}\cap W''}}\big)_{t\in B,i\in I^*}.
\end{split}\EE

\vspace{.2in}

\noindent
Since $B$ is compact, there exists a neighborhood $W_{I^*}$
of $X_{I^*}\!\subset\!X_{\eset}$ such~that
\BE{extendreg_e14}B\!\times\!W_{I^*}\subset \bigcup_{t\in B}\!\!
\bigg(\!\!\{t\}\!\times\!\bigcup_{i\in I^*}\!\Im\big(\Psi_{t;i}\big)\!\!\bigg) \,.\EE
Since
$$\ov{W}'\subset W''  \qquad\hbox{and}\qquad 
X_{I^*}\!\cap\!X_I=X_{I^*\cup I}\subset W'~~\forall\,I\!\in\!\cP(N)\!-\!\cP(I^*),$$
we can shrink~$W_{I^*}$ so~that
\BE{nonoverlap_e5}
W_{I^*}\!\cap\!W'\subset \bigcup_{i\in I^*}\!\Psi_{t;i}\big(\cN_i^{\circ}|_{X_{I^*}\cap W''}\big)
\quad\forall\,t\!\in\!B\EE
and that \eref{extendreg_prp_e2} holds.\\

\noindent
For $t\!\in\!B$ and $i\!\in\!I\!\subset\!I^*$ with $|I|\!\ge\!2$, let
$$\wh\cN_{t;I^*;I}=\cN_I^{\circ}\!\cap\!\Psi_{t;i}^{-1}\big(W_{I^*}\big)
\subset \cN_{I^*;I}\subset \cN X_{I^*},\quad
\Psi_{t;I}\!=\!\Psi_{t;i}\big|_{\wh\cN_{t;I^*;I}}\!:
\wh\cN_{t;I^*;I}\lra X_I\!\cap\!W_{I^*}\subset X_i\,;$$
the diffeomorphism~$\Psi_{t;I}$ is independent of the choice of $i\!\in\!I$ 
by~\eref{ConfRegulLoc_e2}. 
Let
\BE{extendreg_e16}\fD\Psi_{t;I}\!: \pi_{I^*;I}^*\cN_{I^*;I^*-I}\big|_{\wh\cN_{t;I^*;I}}
\lra \cN X_I\big|_{X_I\cap W_{I^*}}\EE
be the isomorphism of split vector bundles covering~$\Psi_{t;I}$ as in~\eref{fDPsiIIconf_e}
with $I'\!\subset\!I$ replaced by $I\!\subset\!I^*$.
Analogously to~\eref{cNtot_e4}, we identify $\pi_{I^*;I}^*\cN_{I^*;I^*-I}$
with~$\cN X_{I^*}$ so~that 
$$\wh\cN_{t;I^*;i}\!\equiv\!\Psi_{t;i}^{-1}\big(W_{I^*}\big)
\subset \pi_{I^*;I}^*\cN_{I^*;I^*-I}
\qquad\forall\,i\!\in\!I\,.$$
For $i\!\in\!I\!\subset\!I^*$, define 
\begin{gather}\label{extendreg_e27}
\big(\wh\rho_{t;I;i},\wh\na^{(t;I;i)}\big)
=\big\{\fD\Psi_{t;I}^{-1}\big\}^{\!*}\pi_{I^*;I}^*\big(\rho_{t;i},\na^{(t;i)}\big),\\
\label{extendreg_e27b}
\wh\cN_{t;I;i}=\fD\Psi_{t;I}\big(\wh\cN_{t;I^*;i}|_{\wh\cN_{t;I^*;I}}\big), \quad
\wh\Psi_{t;I;i}=\Psi_{t;i}\circ\fD\Psi_{t;I}^{\,-1}\big|_{\wh\cN_{t;I;i}}\!:
\wh\cN_{t;I;i}\lra W_{I^*}\subset X_i\,.
\end{gather}
Since $(\Psi_{t;i})_{i\in I^*}$ is a regularization for~$X_{I^*}$ in~$\X$
and~\eref{extendreg_e16} is an isomorphism of split vector bundles,
the tuple $(\wh\Psi_{t;I;i})_{i\in I}$ is a regularization for $X_I\!\cap\!W_{I^*}$
in~$\X$ in the sense of Definition~\ref{ConfRegulLoc_dfn} 
for all $I\!\subset\!I^*$ and $t\!\in\!B$.
Since $((\rho_{t;j},\na^{(t;j)})_{j\in I^*-i},\Psi_{t;i})$ is an
$(\om_{t,1;i})_{i\in[N]}$-regularization  for $X_{I^*}$ in~$\X$
for all $t\!\in\!B$ and $i\!\in\!I^*$,
$((\wh\rho_{t;I;j},\wh\na^{(t;I;j)})_{j\in I-i},\wh\Psi_{t;I;i})$ is
an $(\om_{t,1;i})_{i\in[N]}$-regularization for $X_I\!\cap\!W_{I^*}$ 
in~$\X$ for all $t\!\in\!B$ and $i\!\in\!I\!\subset\!I^*$.\\

\noindent
By the first equation in~\eref{extendreg_e27b},
$$\fD\Psi_{t;I}^{\,-1}\big(\wh\cN_{t;I;i}\!\cap\!\cN_{I;I'}\big)\subset\wh\cN_{t;I^*;I'},\quad
\fD\Psi_{t;I}^{\,-1}\big(\pi_{I;I'}^*\cN_{I;I-I'}\big|_{\wh\cN_{t;I;i}\cap\cN_{I;I'}}\big)  
\subset \pi_{I^*;I'}^*\cN_{I^*;I^*-I'}\big|_{\wh\cN_{t;I^*;I'}}\,,$$
whenever $i\!\in\!I'\!\subset\!I\!\subset\!I^*$.
By the second equation in~\eref{extendreg_e27b},
\begin{gather*} 
\fD\wh\Psi_{t;I;I'}=\fD\Psi_{t;I'}\!\circ\!\fD\Psi_{t;I}^{\,-1}\!:
\pi_{I;I'}^*\cN_{I;I-I'}\big|_{\wh\cN_{t;I;i}\cap\cN_{I;I'}}
\lra \cN X_{I'}\big|_{X_{I'}\cap W_{I^*}},\\
\fD\wh\Psi_{t;I;I'}^{\,-1}\big(\wh\cN_{t;I';i}\big)\subset\wh\cN_{t;I;i}\,,
\quad
\wh\Psi_{t;I;i}=\Psi_{t;I';i}\!\circ\!\fD\wh\Psi_{t;I;I'}\!:
\fD\wh\Psi_{t;I;I'}^{\,-1}\big(\wh\cN_{t;I';i}\big)\lra X_i\,,
\end{gather*}
whenever $i\!\in\!I'\!\subset\!I\!\subset\!I^*$ and $|I'|\!\ge\!2$.
Along with~\eref{extendreg_e27}, this implies that 
the tuple $(\wh\cR_{t;I})_{I\in\cP^*(I^*)}$ satisfies
the second bullet condition in Definition~\ref{LocalRegul_dfn} with~$[N]$ replaced by~$I^*$
for every $t\!\in\!B$.
Thus, 
$$\big(\wh\cR_{t;I}\big)_{t\in B,I\in\cP^*(I^*)}
\equiv \big(\wh\rho_{t;I;i},\wh\na^{(t;I;i)},\wh\Psi_{t;I;i}\big)_{i\in I\subset I^*,t\in B} $$
is an $(\om_{t,1;i})_{t\in B,i\in I^*}$-family of 
weak regularizations for $\{X_I\}_{I\in\cP^*(I^*)}$ over~$W_{I^*}$.\\

\noindent
By~\eref{extendreg_e14} and~\eref{nonoverlap_e5},  
\BE{extendreg_e24}
X_i\!\cap\!W_{I^*}\subset \Im(\Psi_{t;i})~~\forall\,t\!\in\!N'(\prt B),\quad
X_i\!\cap\!W'\!\cap\!W_{I^*}\subset \Psi_{t;i}\big(\Dom(\Psi_{t;i})|_{X_{I^*}\cap W''}\big)
~~\forall\,t\!\in\!B,\EE
whenever $i\!\in\!I^*$.
Along with the first part of the second bullet condition in Definition~\ref{LocalRegul_dfn}, 
this implies~that
\begin{equation*}\begin{aligned}
\big(\rho_{t;I;i},\na^{(t;I;i)}\big)\big|_{X_I\cap W_{I^*}}
&=\big\{\fD\Psi_{t;I^*;I}^{\,-1}\big\}^{\!*}\pi_{I^*;I}^*\big(\rho_{t;I^*;i},\na^{(t;I^*;i)}\big)
&&\forall\,t\!\in\!N'(\prt B),\\
\big(\rho_{t;I;i}',\na'^{(t;I;i)}\big)\big|_{X_I\cap W'\cap W_{I^*}}
&=\big\{\fD\Psi_{t;I^*;I}^{\,-1}\big\}^{\!*}\pi_{I^*;I}^*\big(\rho'_{t;I^*;i},\na'^{(t;I^*;i)}\big)
\big|_{X_I\cap W'\cap W_{I^*}}
&&\forall\,t\!\in\!B,
\end{aligned}\end{equation*}
whenever $i\!\in\!I\!\subset\!I^*$ and $|I|\!\ge\!2$.
Combining the last four equations with~\eref{extendreg_e27} and~\eref{extendreg0_e21},  we~obtain
\BE{extendreg_e32}\begin{split}
\big(\wh\rho_{t;I;i},\wh\na^{(t;I;i)}\big)_{t\in N'(\prt B),i\in I\subset I^*}
&=\big((\rho_{t;I;i},\na^{(t;I;i)})|_{X_I\cap W_{I^*}}\big)_{t\in N'(\prt B),i\in I\subset I^*},\\
\big((\wh\rho_{t;I;i},\wh\na^{(t;I;i)})|_{X_I\cap W'\cap W_{I^*}}
\big)_{t\in B,i\in I\subset I^*} 
&= \big((\rho'_{t;I;i},\na'^{(t;I;i)})|_{X_I\cap W'\cap W_{I^*}}
\big)_{t\in B,i\in I\subset I^*}\,.
\end{split}\EE

\vspace{.1in}

\noindent
By~\eref{extendreg_e24} and~\eref{extendreg0_e21}, 
\begin{equation*}\begin{aligned}
&\fD\Psi_{t;I^*;I}^{\,-1}\big(\wh\cN_{t;I;i}\!\cap\!\Dom(\Psi_{t;I;i})\big)
=\fD\Psi_{t;I}^{\,-1}\big(\wh\cN_{t;I;i}\!\cap\!\Dom(\Psi_{t;I;i})\big)\subset
\wh\cN_{t;I^*;i}\subset\Dom(\Psi_{t;i}) &&\forall\,t\!\in\!N'(\prt B),\\
&\fD\Psi_{t;I^*;I}'^{\,-1}\big(\wh\cN_{t;I;i}|_{X_I\cap W'\cap W_{I^*}}\!\cap\!\Dom(\Psi_{t;I;i}')\big)
=\fD\Psi_{t;I}^{\,-1}
\big(\wh\cN_{t;I;i}|_{X_I\cap W'\cap W_{I^*}}\!\cap\!\Dom(\Psi_{t;I;i}')\big)\\
&\hspace{2.7in}
\subset \wh\cN_{t;I^*;i}\big|_{X_{I^*}\cap W''}
\subset\Dom(\Psi_{t;i})\big|_{X_{I^*}\cap W''} &&\forall\,t\!\in\!B,
\end{aligned}\end{equation*}
whenever $i\!\in\!I\!\subset\!I^*$ and $|I|\!\ge\!2$.
Along with the second part of the second bullet condition in Definition~\ref{LocalRegul_dfn}, 
this implies~that 
\begin{equation*}\begin{aligned}
\Psi_{t;I;i}\big|_{\wh\cN_{t;I;i}\cap\Dom(\Psi_{t;I;i})}
&=\Psi_{t;I^*;i}\circ\fD\Psi_{t;I^*;I}^{\,-1}\big|_{\wh\cN_{t;I;i}\cap\Dom(\Psi_{t;I;i})} 
&&\forall\,t\!\in\!N'(\prt B),\\
\Psi_{t;I;i}'\big|_{\wh\cN_{t;I;i}|_{X_I\cap W'\cap W_{I^*}}\cap\Dom(\Psi_{t;I;i}')}
&=\Psi_{t;I^*;i}'\circ\fD\Psi_{t;I^*;I}'^{\,-1}
\big|_{\wh\cN_{t;I;i}|_{X_I\cap W'\cap W_{I^*}}\cap\Dom(\Psi_{t;I;i}')} 
&&\forall\,t\!\in\!B.
\end{aligned}\end{equation*}
Combining the last four equations with~\eref{extendreg_e27b}  and~\eref{extendreg0_e21}, 
we~obtain
\BE{extendreg_e42}\begin{split}
\big(\wh\Psi_{t;I;i}|_{\wh\cN_{t;I;i}\cap\Dom(\Psi_{t;I;i})}
\big)_{t\in N'(\prt B),i\in I\subset I^*}
&=\big(\Psi_{t;I;i}|_{\wh\cN_{t;I;i}\cap\Dom(\Psi_{t;I;i})}
\big)_{t\in N'(\prt B),i\in I\subset I^*}\,,\\
\big(\wh\Psi_{t;I;i}|_{\wh\cN_{t;I;i}|_{X_I\cap W'\cap W_{I^*}}\cap\Dom(\Psi_{t;I;i}')}
\big)_{t\in B,i\in I\subset I^*}
&=\big(\Psi_{t;I;i}'|_{\wh\cN_{t;I;i}|_{X_I\cap W'\cap W_{I^*}}\cap\Dom(\Psi_{t;I;i}')}
\big)_{t\in B,i\in I\subset I^*}.
\end{split}\EE
By~\eref{extendreg_e32} and~\eref{extendreg_e42},
$(\wh\cR_{t;I})_{t\in B,I\in\cP^*(I^*)}$ 
satisfies~\eref{extendreg_prp_e2a} and~\eref{extendreg_prp_e2b}.
\end{proof}

\subsection{Merging weak regularizations and equivalences}
\label{PastingRegul_subs}

\noindent
By Lemma~\ref{regulcomb_lmm} below, two weak regularizations for~$\X$
over open subsets of~$X_{\eset}$
that are equivalent over their intersection can be pasted together over 
the union of slightly smaller open subsets.
By Corollary~\ref{regulcomb_crl}, two weak regularizations that are equivalent
over each of two open subsets are also equivalent over 
the union of slightly smaller open subsets.

\begin{lmm}\label{regulcomb_lmm}
Let $\X$, $B$, $(\om_{t;i})_{t\in B,i\in[N]}$, $I^*$, and $W'\!\subset\!W$ be as in 
Proposition~\ref{SCCweak_prp}.
Suppose \hbox{$W_{I^*},W_{I^*}'\!\subset\!X_{\eset}$} are open subsets such~that 
\BE{regulcomb_lmm_e1}
\ov{W_{I^*}'}\subset W_{I^*}, \qquad
X_I\!\cap\!\ov{W_{I^*}'}\subset W~~~\forall~I\in\cP(N)\!-\!\cP(I^*),\EE
and $(\cR_{t;I})_{t\in B,I\in\cP^*(N)}$ and $(\wh\cR_{t;I})_{t\in B,I\in\cP^*(I^*)}$
are an $(\om_{t;i})_{t\in B,i\in[N]}$-family of weak regularizations for~$\X$ over~$W$
and an $(\om_{t;i})_{t\in B,i\in I^*}$-family of weak regularizations for 
$\{X_I\}_{I\in\cP^*(I^*)}$ over~$W_{I^*}$, respectively, such~that 
\BE{cRoverlap_e}\big(\cR_{t;I}\big)_{t\in B,I\in\cP^*(I^*)}
\cong_{W\cap W_{I^*}}\!\!\big(\wh\cR_{t;I}\big)_{t\in B,I\in\cP^*(I^*)}\,.\EE
Then there exists an $(\om_{t;i})_{t\in B,i\in[N]}$-family $(\wt\cR_{t;I})_{t\in B,I\in\cP^*(N)}$
of weak regularizations for $\X$ over $W'\!\cup\!W_{I^*}'$
such~that  
\BE{regulcomb_lmm_e2}\begin{split}
\big(\wt\cR_{t;I}\big)_{t\in B,I\in\cP^*(N)}
&\cong_{W'}\!\big(\cR_{t;I}\big)_{t\in B,I\in\cP^*(N)},~~
\big(\wt\cR_{t;I}\big)_{t\in B,I\in\cP^*(I^*)} 
\cong_{W_{I^*}'}\!\! \big(\wh\cR_{t;I}\big)_{t\in B,I\in\cP^*(I^*)}\,.
\end{split}\EE
\end{lmm}

\begin{proof} 
Let
\begin{alignat}{2}
\label{regulcomb_e2}
(\cR_{t;I})_{t\in B,I\in\cP^*(N)} &\equiv  
\big(\rho_{t;I;i},\na^{(t;I;i)},\Psi_{t;I;i}\big)_{t\in B,i\in I\subset[N]},
&\quad \cN_{t;I;i}&=\Dom(\Psi_{t;I;i})\subset\cN_{I;i},\\
\notag
(\wh\cR_{t;I})_{t\in B,I\in\cP^*(I^*)} &\equiv  
\big(\wh\rho_{t;I;i},\wh\na^{(t;I;i)},\wh\Psi_{t;I;i}\big)_{t\in B,i\in I\subset I^*},
&\quad \wh\cN_{t;I;i}&=\Dom(\wh\Psi_{t;I;i})\subset\cN_{I;i}.
\end{alignat}
By~\eref{cRoverlap_e}, there exists an $(\om_{t;i})_{t\in B,i\in I^*}$-family 
$$(\cR'_{t;I})_{t\in B,I\in\cP^*(I^*)} \equiv  
\big(\rho'_{t;I;i},\na'^{(t;I;i)},\Psi'_{t;I;i}\big)_{t\in B,i\in I\subset I^*}$$
of weak regularizations 
for $\{X_I\}_{I\in\cP^*(I^*)}$ over~$W\!\cap\!W_{I^*}$ such~that 
\begin{gather}
\label{regulcomb_e1a}
(\rho'_{t;I;i},\na'^{(t;I;i)})=
(\rho_{t;I;i},\na^{(t;I;i)})\big|_{X_I\cap W\cap W_{I^*}},
(\wh\rho_{t;I;i},\wh\na^{(t;I;i)})\big|_{X_I\cap W\cap W_{I^*}}
\quad\forall\,i\!\in\!I\!\subset\!I^*,\\
\label{regulcomb_e1b}
\cN_{t;I;i}'\!\equiv\!\Dom(\Psi_{t;I;i}')\subset\cN_{t;I;i},\wh\cN_{t;I;i}, \quad
\Psi_{t;I;i}'=\Psi_{t;I;i}|_{\cN_{t;I;i}'},\wh\Psi_{t;I;i}|_{\cN_{t;I;i}'}
\qquad\forall~i\!\in\!I\!\subset\!I^*.
\end{gather}

\vspace{.1in}

\noindent
Define
\begin{alignat*}{3}
W^{\circ}&=W'\!-\!\ov{W'\!\cap\!W_{I^*}'}, &\quad
W_{I^*}^{\circ}&=W_{I^*}'\!-\!\ov{W'\!\cap\!W_{I^*}'}, &\quad
W_{\cap}&=W\!\cap\!W_{I^*}\cap(W'\!\cup\!W_{I^*}'),\\
\cN_{t;I;i}^{\circ}&=\Psi_{t;I;i}^{\,-1}\big(W^{\circ}\big)\big|_{X_I\cap W^{\circ}}, 
&\quad
\wh\cN_{t;I;i}^{\circ}&=\wh\Psi_{t;I;i}^{\,-1}
\big(W_{I^*}^{\circ}\big)\big|_{X_I\cap W_{I^*}^{\circ}}, 
&\quad
\cN_{t;I;i}^{\,\cap}&=\Psi_{t;I;i}'^{\,-1}\big(W_{\cap}\big)\big|_{X_I\cap W_{\cap}}.
\end{alignat*}
By the first assumption in~\eref{SCCweak_e0} and~\eref{regulcomb_lmm_e1},
\begin{gather}
\label{SCCweak_e2}
W^{\circ}\cap W_{I^*}^{\circ}=\eset, \qquad
W'\!\cup\!W_{I^*}'=W^{\circ}\!\cup\!W_{I^*}^{\circ}\!\cup\!W_{\cap},\\
\label{SCCweak_e2b}
X_I\!\cap\!(W'\!\cup\!W_{I^*}')= 
X_I\!\cap\!\big(W^{\circ}\!\cup\!W_{\cap}\big)
\subset X_I\!\cap\!W
~~\forall\,I\!\in\!\cP(N)\!-\!\cP(I^*)\,.
\end{gather}
For $t\!\in\!B$ and $i\!\in\!I\!\subset\![N]$ with $|I|\!\ge\!2$, let
$$\wt\cN_{t;I;i}=\begin{cases}
\cN_{t;I;i}^{\circ}\!\cup\!\cN_{t;I;i}^{\,\cap}
,&\hbox{if}~I\!\in\!\cP(N)\!-\!\cP(I^*);\\
\cN_{t;I;i}^{\circ}\!\cup\!\wh\cN_{t;I;i}^{\circ}\!\cup\!
\cN_{t;I;i}^{\,\cap},&\hbox{if}~I\!\in\!\cP^*(I^*).
\end{cases}$$
By the second statement in~\eref{SCCweak_e2} and~\eref{SCCweak_e2b}, 
$\wt{\cN}_{t;I;i}$ is a neighborhood of $X_I\!\cap\!(W'\!\cup\!W_{I^*}')$
in $\cN_{I;i}|_{X_I\cap(W'\cup W_{I^*}')}$.\\

\noindent 
With $t\!\in\!B$ and $i\!\in\!I$ as above, define
\BE{regulcomb_e4}\wt\Psi_{t;I;i}\!:\wt\cN_{t;I;i}\lra X_i, \qquad
\wt\Psi_{t;I;i}(v)=\begin{cases}
\Psi_{t;I;i}(v),&\hbox{if}~v\!\in\!\cN_{t;I;i}^{\circ}\!\cup\!\cN_{t;I;i}^{\cap};\\
\wh\Psi_{t;I;i}(v),&\hbox{if}~v\!\in\!\wh\cN_{t;I;i}^{\circ},\,I\!\in\!\cP^*(I^*).
\end{cases}\EE
By the first statement in~\eref{SCCweak_e2} and~\eref{regulcomb_e1b}, 
these definitions agree on the overlap  
$\cN_{t;I;i}^{\cap}\!\cap\!\wh\cN_{t;I;i}^{\circ}$.
Since 
\BE{regulcomb_e6}\wt\Psi_{t;I;i}\big(\cN_{t;I;i}^{\circ}\big)\cap
\wt\Psi_{t;I;i}\big(\wh\cN_{t;I;i}^{\circ}\big)
\subset W^{\circ}\!\cap\!W_{I^*}^{\circ}=\eset\EE
by the first statement in~\eref{SCCweak_e2}
and the maps $\Psi_{t;I;i}$ and~$\wh\Psi_{t;I;i}$ are injective,
\eref{regulcomb_e1b} implies that the map~$\wt\Psi_{t;I;i}$ is injective as well.
Since the tuple $(\Psi_{t;I;i})_{i\in I}$ is a regularization for $X_I\!\cap\!W$ in~$\X$
for every $I\!\in\!\cP^*(N)$ and 
$(\wh\Psi_{t;I;i})_{i\in I}$ is a regularization for $X_I\!\cap\!W_{I^*}$ in~$\X$
for every $I\!\in\!\cP^*(I^*)$, we conclude that 
 $(\wt\Psi_{t;I;i})_{i\in I}$ is a regularization for $X_I\!\cap\!(W'\!\cup\!W_{I^*}')$ 
in~$\X$ for every $I\!\in\!\cP^*(N)$.
By~\eref{regulcomb_e6} and~\eref{regulcomb_e1b}, 
these regularizations satisfy~\eref{LocalRegul_e2}.\\

\noindent
We define a metric~$\wt\rho_{t;I;i}$ and a connection~$\wt\na^{(t;I;i)}$ 
on the vector bundle 
$\cN_{X_{I-i}}X_I|_{X_I\cap(W'\cup W_{I^*}')}$ by 
\BE{regulcomb_e8}\big(\wt\rho_{t;I;i},\wt\na^{(t;I;i)}\big)_x=
\begin{cases}
(\rho_{t;I;i},\na^{(t;I;i)})_x,&\hbox{if}~x\!\in\!X_I\!\cap\!(W^{\circ}\!\cup\!W_{\cap});\\
(\wh\rho_{t;I;i},\wh\na^{(t;I;i)})_x,&\hbox{if}~x\!\in\!X_I\!\cap\!W_{I^*}^{\circ},
\,I\!\in\!\cP^*(I^*).
\end{cases}\EE
By the first statement in~\eref{SCCweak_e2} and~\eref{regulcomb_e1a}, 
these definitions agree on the overlap $X_I\!\cap\!W_{\cap}\!\cap\!W_{I^*}^{\circ}$.
Since the~tuples
$$\big((\rho_{t;I;j},\na^{(t;I;j)})_{j\in I-i},\Psi_{t;I;i}\big)
\qquad\hbox{and}\qquad
\big((\wh\rho_{t;I;j},\wh\na^{(t;I;j)})_{j\in I-i},\wh\Psi_{t;I;i}\big)$$
are an $\om_{t;i}$-regularization for $X_I\!\cap\!W$ in~$X_i$ whenever $i\!\in\!I\!\subset\![N]$
and
an $\om_{t;i}$-regularization for $X_I\!\cap\!W_{I^*}$ in~$X_i$ whenever $i\!\in\!I\!\subset\!I^*$,
respectively, we conclude that the~tuple
$((\wt\rho_{t;I;j},\wt\na^{(t;I;j)})_{j\in I-i},\wt\Psi_{t;I;i})$
is an $\om_{t;i}$-regularization for $X_I\!\cap\!(W'\!\cup\!W_{I^*}')$ 
in~$X_i$ whenever $i\!\in\!I\!\subset\![N]$.
By~\eref{regulcomb_e6}, the maps~\eref{regulcomb_e4} and  
the pairs~\eref{regulcomb_e8} satisfy the first part of the second bullet condition
in Definition~\ref{LocalRegul_dfn}.\\

\noindent
By the last two paragraphs, the tuple
\BE{regulcomb_e10}\big(\wt\cR_{t;I}\big)_{I\in\cP^*(N)}\equiv 
\big(\wt\rho_{t;I;i},\wt\na^{(t;I;i)},\wt\Psi_{t;I;i}\big)_{i\in I\subset[N]}\EE
is an $(\om_{t;i})_{i\in[N]}$-family of weak regularizations for $\X$ over $W'\!\cup\!W_{I^*}'$.
By~\eref{regulcomb_e4} and~\eref{regulcomb_e8}, it satisfies~\eref{regulcomb_lmm_e2}.
\end{proof}

\begin{crl}\label{regulcomb_crl}
Let $\X$, $B$, $(\om_{t;i})_{t\in B,i\in[N]}$, $I^*$, $W'\!\subset\!W$, and
$W_{I^*}'\!\subset\!W_{I^*}$ be  as in Lemma~\ref{regulcomb_lmm}.
If
$(\cR_{t;I}^{(1)})_{t\in B,I\in\cP^*(N)}$ and 
$(\cR_{t;I}^{(2)})_{t\in B,I\in\cP^*(N)}$ are
$(\om_{t;i})_{t\in B,i\in[N]}$-families of weak regularizations for $\X$
over $W\!\cup\!W_{I^*}$ such~that 
\BE{regulcombcrl3_e}
\big(\cR_{t;I}^{(1)}\big)_{t\in B,I\in\cP^*(N)}\cong_W \!
\big(\cR_{t;I}^{(2)}\big)_{t\in B,I\in\cP^*(N)},  ~~
\big(\cR_{t;I}^{(1)}\big)_{t\in B,I\in\cP^*(I^*)}\cong_{W_{I^*}} \!\!
\big(\cR_{t;I}^{(2)}\big)_{t\in B,I\in\cP^*(I^*)}, \EE
then
$$\big(\cR_{t;I}^{(1)}\big)_{t\in B,I\in\cP^*(N)}\cong_{W'\cup W_{I^*}'} \!\!
\big(\cR_{t;I}^{(2)}\big)_{t\in B,I\in\cP^*(N)}\,.$$
\end{crl}

\begin{proof}
Let $\cR_{t;I}^{(1)}$ and $\cR_{t;I}^{(2)}$ be as in~\eref{R1R2_e}.
By the first assumption in~\eref{regulcombcrl3_e}, there exists 
an  $(\om_{t;i})_{t\in B,i\in[N]}$-family
\BE{regulcomb3_e2}
\big(\cR_{t;I}\big)_{t\in B,I\in\cP^*(N)}
\equiv\big(\rho_{t;I;i},\na^{(t;I;i)},\Psi_{t;I;i}\big)_{t\in B,i\in I\subset[N]}\EE
of weak regularizations for $\X$ over~$W$ such~that 
\BE{regulcomb3_e3}\begin{split}
&\qquad\big(\rho_{t;I;i},\na^{(t;I;i)}\big)=
\big(\rho_{t;I;i}^{(1)},\na^{(1),(t;I;i)}\big)\big|_{X_I\cap W},
\big(\rho_{t;I;i}^{(2)},\na^{(2),(t;I;i)}\big)\big|_{X_I\cap W},\\
&\Dom\big(\Psi_{t;I;i}\big)\subset
\Dom\big(\Psi_{t;I;i}^{(1)}\big),\Dom\big(\Psi_{t;I;i}^{(2)}\big),
\quad
\Psi_{t;I;i}=\Psi_{t;I;i}^{(1)}\big|_{\Dom(\Psi_{t;I;i})},\Psi_{t;I;i}^{(2)}\big|_{\Dom(\Psi_{t;I;i})}
\end{split}\EE
for all $t\!\in\!B$ and $i\!\in\!I\!\subset\![N]$ with $|I|\!\ge\!2$.
By the second assumption in~\eref{regulcombcrl3_e}, there exists 
an $(\om_{t;i})_{t\in B,i\in I^*}$-family
\BE{regulcomb3_e4}
\big(\wh\cR_{t;I}\big)_{t\in B,I\in\cP^*(I^*)}\equiv
\big(\wh\rho_{t;I;i},\wh\na^{(t;I;i)},
\wh\Psi_{t;I;i}\big)_{t\in B,i\in I\subset I^*}\EE
of weak regularizations for $\{X_I\}_{I\in\cP(I^*)}$ over~$W_{I^*}$ such that 
\BE{regulcomb3_e5}\begin{split}
&\qquad\big(\wh\rho_{t;I;i},\wh\na^{(t;I;i)}\big)=
\big(\rho_{t;I;i}^{(1)},\na^{(1),(t;I;i)}\big)\big|_{X_I\cap W_{I^*}},
\big(\rho_{t;I;i}^{(2)},\na^{(2),(t;I;i)}\big)\big|_{X_I\cap W_{I^*}},\\
&\Dom\big(\wh\Psi_{t;I;i}\big)\subset
\Dom\big(\Psi_{t;I;i}^{(1)}\big),\Dom\big(\Psi_{t;I;i}^{(2)}\big),
\quad
\wh\Psi_{t;I;i}=\Psi_{t;I;i}^{(1)}\big|_{\Dom(\wh\Psi_{t;I;i})},
\Psi_{t;I;i}^{(2)}\big|_{\Dom(\wh\Psi_{t;I;i})}
\end{split}\EE
for all $t\!\in\!B$ and $i\!\in\!I\!\subset\!I^*$ with $|I|\!\ge\!2$.\\

\noindent
By~\eref{regulcomb3_e3} and~\eref{regulcomb3_e5},
\begin{gather*}
\big(\rho_{t;I;i},\na^{(t;I;i)}\big)\big|_{X_I\cap W\cap W_{I^*}}=
\big(\wh\rho_{t;I;i},\wh\na^{(t;I;i)}\big)\big|_{X_I\cap W\cap W_{I^*}},\\
\Psi_{t;I;i}|_{\Dom(\Psi_{t;I;i})\cap\Dom(\wh\Psi_{t;I;i}) |_{X_I\cap W\cap W_{I^*}}}
=\wh\Psi_{t;I;i}|_{\Dom(\Psi_{t;I;i})\cap\Dom(\wh\Psi_{t;I;i})|_{X_I\cap W\cap W_{I^*}}}
\end{gather*}
for all $t\!\in\!B$ and $i\!\in\!I\!\subset\!I^*$ with $|I|\!\ge\!2$.
Thus, the families \eref{regulcomb3_e2} and~\eref{regulcomb3_e4}
of weak regularizations  satisfy~\eref{cRoverlap_e}.
The proof of Lemma~\ref{regulcomb_lmm} provides 
an $(\om_{t;i})_{t\in B,i\in[N]}$-family~\eref{regulcomb_e10} 
of weak regularizations  for $\X$ over $W'\!\cup\!W_{I^*}'$ 
such~that 
\begin{equation*}\begin{split}
\big(\wt\rho_{t;I;i},\wt\na^{(t;I;i)}\big)_x
&=\begin{cases}(\rho_{t;I;i},\na^{(t;I;i)})_x,&\hbox{if}~x\!\in\!X_I\!\cap\!W';\\
(\rho_{t;I;i},\na^{(t;I;i)})_x,&\hbox{if}~x\!\in\!X_I\!\cap\!W_{I^*}',
\,I\!\in\!\cP(N)\!-\!\cP(I^*);\\
(\wh\rho_{t;I;i},\wh\na^{(t;I;i)})_x,&\hbox{if}~x\!\in\!X_I\!\cap\!W_{I^*}',
\,I\!\in\!\cP^*(I^*);
\end{cases}\\
\Dom\big(\wt\Psi_{t;I;i}\big)&\subset
\begin{cases}\Dom(\Psi_{t;I;i}),&\hbox{if}~I\!\in\!\cP(N)\!-\!\cP(I^*);\\
\Dom(\Psi_{t;I;i})\big|_{X_I\cap W'}\!\cup\!\Dom(\wh\Psi_{t;I;i})|_{X_I\cap W_{I^*}'},
&\hbox{if}~I\!\in\!\cP^*(I^*);
\end{cases}\\
\wt\Psi_{t;I;i}(v)&=\begin{cases}
\Psi_{t;I;i}(v),&\hbox{if}~v\!\in\!\Dom(\wt\Psi_{t;I;i})\big|_{X_I\cap W'};\\
\Psi_{t;I;i}(v),&\hbox{if}~v\!\in\!\Dom(\wt\Psi_{t;I;i})\big|_{X_I\cap W_{I^*}'},
\,I\!\in\!\cP(N)\!-\!\cP(I^*);\\
\wh\Psi_{t;I;i}(v),&\hbox{if}~v\!\in\!\Dom(\wt\Psi_{t;I;i})|_{X_I\cap W_{I^*}'},
\,I\!\in\!\cP^*(I^*).
\end{cases}
\end{split}\end{equation*}
Along with~\eref{regulcomb3_e3} and~\eref{regulcomb3_e5}, these identities imply that 
\begin{gather*}
\big(\wt\rho_{t;I;i},\wt\na^{(t;I;i)}\big)=
\big(\rho_{t;I;i}^{(1)},\na^{(1),(t;I;i)}\big)\big|_{X_I\cap(W'\cup W_{I^*}')},
\big(\rho_{t;I;i}^{(2)},\na^{(2),(t;I;i)}\big)\big|_{X_I\cap(W'\cup W_{I^*}')},\\
\Dom\big(\wt\Psi_{t;I;i}\big)\subset
\Dom\big(\Psi_{t;I;i}^{(1)}\big),\Dom\big(\Psi_{t;I;i}^{(2)}\big),
\quad
\wt\Psi_{t;I;i}=\Psi_{t;I;i}^{(1)}\big|_{\Dom(\wt\Psi_{t;I;i})},
\Psi_{t;I;i}^{(2)}\big|_{\Dom(\wt\Psi_{t;I;i})}
\end{gather*}
for all $t\!\in\!B$ and $i\!\in\!I\!\subset\![N]$.
This establishes the claim.
\end{proof}

\subsection{From weak regularizations to regularizations}
\label{weakregtoreg_subs}

\noindent
We show below that the first requirement in~\eref{overlap_e} is not material, provided 
the second requirement in~\eref{overlap_e} is appropriately modified.
By Lemma~\ref{weakregtoreg_lmm} below, a weak regularization for~$\X$ over~$X_{\eset}$
can be cut down to a regularization for~$\X$.
By Corollary~\ref{weakregtoreg_crl}, two regularizations for~$\X$ 
that are equivalent as weak regularizations over~$X_{\eset}$
are also equivalent as regularizations.

\begin{lmm}\label{weakregtoreg_lmm}
Let $\X$, $B$, and $(\om_{t;i})_{t\in B,i\in[N]}$ be as in Theorem~\ref{SCC_thm}
and $(\cR_{t;I})_{t\in B,I\in\cP^*(N)}$ be an $(\om_{t;i})_{t\in B,i\in[N]}$-family
of weak regularizations for $\X$ over~$X_{\eset}$ as in~\eref{regulcomb_e2}.
Then there exists a collection of neighborhoods 
$$\bigcup_{t\in B}\{t\}\!\times\!\cN_{t;I}'\subset B\!\times\!\cN X_I$$
of $B\!\times\!X_I$ with $I\!\in\!\cP^*(N)$ and $|I|\!\ge\!2$ 
such that $\cN_{t;I}'\!\cap\!\cN_{I;i}\!\subset\!\cN_{t;I;i}$
for all $i\!\in\!I\!\subset\![N]$ with $|I|\!\ge\!2$ and the~tuple 
\BE{weakregtoreg_e0c}
\big(\cR_{t;I}'\big)_{t\in B,I\in\cP^*(N)} \equiv
\big(\rho_{t;I;i},\na^{(t;I;i)},
\Psi_{t;I;i}|_{\cN_{t;I}'\cap\cN_{I;i}} \big)_{t\in B,i\in I\subset[N]}\EE
is an $(\om_{t;i})_{t\in B,i\in[N]}$-family of  regularizations for $\X$
in the sense of Definition~\ref{SCCregul_dfn}\ref{SCCreg2_it}.
\end{lmm}

\begin{proof} 
For each $I\!\in\!\cP^*(N)$ with $|I|\!\ge\!2$, let
\hbox{$\pi_I\!:\cN X_I\!\lra\!X_I$} be the bundle projection~map.
If in addition $t\!\in\!B$, define
$$\Psi_{t;I}\!:\bigcup_{i\in I}\cN_{t;I;i}\lra X_{\eset}, \qquad
\Psi_{t;I}(v)=\Psi_{t;I;i}(v)~~\forall\,v\!\in\!\cN_{t;I;i},\,i\!\in\!I;$$
by~\eref{ConfRegulLoc_e2}, 
$\Psi_{t;I}(v)$ is well-defined for $v\!\in\!\cN_{t;I;i_1}\!\cap\!\cN_{t;I;i_2}$.
Let
$$\rho_{t;I}\!:\cN X_I\lra\R, \qquad
\rho_{t;I}\big((v_i)_{i\in I}\big)
=\max\big\{\rho_{t;I;i}(v_i)\!:\,i\!\in\!I\big\},$$
to be the square metric on~$\cN X_I$.\\

\noindent
For $I\!\in\!\cP^*(N)$ with $|I|\!\ge\!2$, let
$$\bigcup_{t\in B}\{t\}\!\times\!\cN_{t;I}\subset B\!\times\!\cN X_I$$
be  a neighborhood of $B\!\times\!X_I$ such that
$\cN_{t;I}\!\cap\!\cN_{I;i}\!\subset\!\cN_{t;I;i}$
for all $t\!\in\!B$ and $i\!\in\!I$.
Define
\BE{weakregtoreg_e4}
\cN_{t;I}^{\circ}=
\bigcap_{\begin{subarray}{c}I'\subset I\\ |I'|\ge2\end{subarray}} \!\!\!
\fD\Psi_{t;I;I'}^{\,-1}(\cN_{t;I'}), \quad
\cN_{t;I;\prt}^{\circ}= \cN_{t;I}^{\circ}\cap\cN_{\prt}X_I
\qquad\forall\,t\!\in\!B.\EE
By~\eref{LocalRegul_e2},
\BE{weakregtoreg_e4b}
\Psi_{t;I}\big|_{\cN_{t;I;\prt}^{\circ}}=
\Psi_{t;I'}\circ\fD\Psi_{t;I;I'}\big|_{\cN_{t;I;\prt}^{\circ}}
\quad\forall~I'\!\subset\!I\!\subset\![N],\,|I'|\!\ge\!2.\EE
By~\eref{weakregtoreg_e4} and~\eref{weakregtoreg_e4b}, 
\BE{weakregtoreg_e4c}
\fD\Psi_{t;I;I'}\big(\cN_{t;I}^{\circ}\big)\subset\cN_{t;I'}^{\circ} 
 \quad\forall~I'\!\subset\!I\!\subset\![N],\,|I'|\!\ge\!2.\EE
If in addition $\ve\!\in\!C^{\i}(B\!\times\!X_I;\R^+)$, define
$$\cN_{t;I}(\ve)=\big\{v\!\in\!\cN X_I\!:\,
\rho_{t;I}(v)\!<\!\ve\big(t,\pi_I(v)\big)\big\},\quad
\cN_{t;I;\prt}(\ve)= \cN_{t;I}(\ve)\cap\cN_{\prt}X_I.$$
In particular,
$$\cN^{\circ} X_I\equiv\bigcup_{t\in B}\{t\}\!\times\!\cN_{t;I}^{\circ},\quad
\cN X_I(\ve)\equiv\bigcup_{t\in B}\{t\}\!\times\!\cN_{t;I}(\ve)
\subset B\!\times\!\cN X_I$$
are neighborhoods of $B\!\times\!X_I$ in $B\!\times\!\cN X_I$.\\

\noindent
We show below that there exist  functions $\ve_I\!\in\!C^{\i}(B\!\times\!X_I;\R^+)$
with $I\!\subset\![N]$, $|I|\!\ge\!2$, such~that 
\begin{gather}
\label{weakregtoreg_e5a}
\ov{\cN_{t;I}(2^{|I|}\ve_I)}\subset\cN_{t;I}^{\circ},\\
\label{weakregtoreg_e5c}
\ve_{I'}\big(t,\Psi_{t;I}(v)\big)=\ve_I\big(t,\pi_I(v)\big)
~~\forall\,v\!\in\!\cN_{t;I}\big(2^{|I'|}\ve_I\big)\!\cap\!\cN_{I;I'}
\end{gather}
for all $t\!\in\!B$ and $I'\!\subset\!I\!\subset\![N]$ with $|I'|\!\ge\!2$.
We take $\cN_{t;I}'\!=\!\cN_{t;I}(\ve_I)$.
By~\eref{weakregtoreg_e4b} and~\eref{weakregtoreg_e5a}, 
$$\Psi_{t;I}\big|_{\cN_{t;I}'\cap\cN_{\prt}X_I}=
\Psi_{t;I'}\circ\fD\Psi_{t;I;I'}\big|_{\cN_{t;I}'\cap\cN_{\prt}X_I}\,.$$
Since $\fD\Psi_{t;I;I'}$ is a product Hermitian isomorphism,
\BE{weakregtoreg_e5}\begin{split}
\fD\Psi_{t;I;I'}\big(\cN_{t;I}(\ve_I)\big)
&=\bigcup_{v\in\cN_{t;I}(\ve_I)\cap\cN_{I;I'}} \hspace{-.4in}
\big\{w\!\in\!\cN X_{I'}|_{\Psi_{t;I}(v)}\!:\,\rho_{t;I'}(w)\!<\!\ve_I(\pi_I(v))\big\}\\
&=\cN_{t;I'}(\ve_{I'})\big|_{\Psi_{t;I}(\cN_{t;I}(\ve_I)\cap\cN_{I;I'})};
\end{split}\EE
the last equality holds by~\eref{weakregtoreg_e5c}.
Combining \eref{weakregtoreg_e5} and \eref{ConfRegulLoc_e1}, we conclude that 
$$\fD\Psi_{t;I;I'}(\cN_{t;I}')=\cN_{t;I'}'
\big|_{X_{I'}\cap\Psi_{t;I}(\cN_{t;I}'\cap\cN_{I;I'})}
\quad\forall~I'\!\subset\!I\!\subset\![N],\,|I'|\!\ge\!2,~t\!\in\!B\,.$$
Along with the assumption that  $(\cR_{t;I})_{t\in B,I\in\cP^*(N)}$ is 
an $(\om_{t;i})_{t\in B,i\in[N]}$-family of weak regularizations for~$\X$ over~$X_{\eset}$,
this implies that~\eref{weakregtoreg_e0c} is an
$(\om_{t;i})_{t\in B,i\in[N]}$-family of regularizations for~$\X$.\\

\noindent
In the remainder of this proof, we inductively construct functions 
$\ve_I\!\in\!C^{\i}(B\!\times\!X_I;\R^+)$ satisfying~\eref{weakregtoreg_e5a}
and~\eref{weakregtoreg_e5c}.
By~\eref{weakregtoreg_e4b} and~\eref{weakregtoreg_e5a}, 
\eref{weakregtoreg_e5c} for all $I'\!\subset\!I\!\subset\![N]$ with $|I'|\!\ge\!2$
is equivalent to~\eref{weakregtoreg_e5c} for such  $I',I$ with $|I\!-\!I'|\!=\!1$.
For each $\ell\!\in\!\Z$, let  
$$\cP^{=\ell}(N),\cP^{>\ell}(N)\subset\cP(N)$$
denote the collections of subsets of cardinality~$\ell$ and
of cardinality greater than~$\ell$, respectively.\\

\noindent
Suppose $\ell\!\in\!\{2,\ldots,N\}$ and we have chosen~$\ve_I$ for all
$I\!\in\!\cP^{>\ell}(N)$ so that \eref{weakregtoreg_e5a} and~\eref{weakregtoreg_e5c} 
are satisfied by all elements of~$\cP^{>\ell}(N)$,
\BE{weakregtoreg_e5b}
\Psi_{t;I_1}\big(\cN_{t;I_1;\prt}(2^{\ell+1}\ve_{I_1})\big)\cap
\Psi_{t;I_2}\big(\cN_{t;I_2;\prt}(2^{\ell+1}\ve_{I_2})\big)\\
\subset 
\Psi_{t;I_1\cup I_2}\big(\cN_{t;I_1\cup I_2;\prt}(2^{\ell+1}\ve_{I_1\cup I_2})\big)
\EE
for all $I_1,I_2\!\in\!\cP^{>\ell}(N)$, and 
\BE{weakregtoreg_e6}
\ov{\Psi_{t;I_1}\big(\cN_{t;I_1;\prt}(2^{\ell+1}\ve_{I_1})\big)}\cap X_{I_2}
\subset 
\ov{\Psi_{t;I_1\cup I_2}\big(\cN_{t;I_1\cup I_2;\prt}(2^{\ell+1}\ve_{I_1\cup I_2})\big)}\EE
whenever $I_1\!\in\!\cP^{>\ell}(N)$ and $I_2\!\in\!\cP(N)$.
Furthermore,
\eref{weakregtoreg_e5b} and~\eref{weakregtoreg_e6} hold
 with $2^{\ell+1}$ and the inclusions 
replaced by $C\!\in\![0,2^{\ell+1}]$ and the equalities.\\

\noindent
For $t\!\in\!B$ and $I^*\!\subsetneq\!I\!\subset\![N]$ with $I^*\!\in\!\cP^{=\ell}(N)$, let
\begin{equation*}\begin{split}
W_{t;I^*;I}=
\big\{\Psi_{t;I}(u,v)\!:
(u,v)\!\in\!(\cN_{I;I^*}\!\oplus\!\cN_{I;I-I^*})\!\cap\!\cN_{\prt}X_I,
\,\rho_{t;I}(u)\!<\!2^{\ell+1}\ve_I(\pi_I(u)),&\\
\rho_{t;I}(v)\!<\!2^{\ell}\ve_I(\pi_I(v))&\big\}\subset X_{\eset}.
\end{split}\end{equation*}
By~\eref{weakregtoreg_e5a} and~\eref{Psikk_e},
$$\ov{\Psi_{t;I}\big(\cN_{t;I;\prt}(2^{\ell+1}\ve_I)\big)\!-\!W_{t;I^*;I}}
\cap X_{I^*}=\eset$$
for all $I^*\!\subsetneq\!I\!\subset\![N]$ with $I^*\!\in\!\cP^{=\ell}(N)$.
Along with \eref{weakregtoreg_e6}, this implies that 
\BE{weakregtoreg_e23}
\ov{\Psi_{t;I}\big(\cN_{t;I;\prt}(2^{\ell+1}\ve_I)\big)\!-\!W_{t;I^*;I^*\cup I}}
\cap X_{I^*}=\eset\EE
for all $I^*\!\in\!\cP^{=\ell}(N)$ and $I\!\in\!\cP^{>\ell}(N)$.
By~\eref{weakregtoreg_e5b}, \eref{weakregtoreg_e5a}  with $I$ replaced 
by $I_1\!\cup\!I_2$, and \eref{weakregtoreg_e4b}  with $(I,I')$ 
replaced by $(I_1\!\cup\!I_2,I_1^*)$, $(I_1\!\cup\!I_2,I_2)$, and
$(I_1\!\cup\!I_2,I_1^*\!\cup\!I_2)$, 
\BE{weakregtoreg_e25a1}
W_{t;I_1^*;I_1} \cap 
\Psi_{t;I_2}\big(\cN_{t;I_2;\prt}(2^{\ell}\ve_{I_2})\big) \subset
\Psi_{t;I_1^*\cup I_2}\big(\cN_{t;I_1^*\cup I_2;\prt}(2^{\ell}\ve_{I_1^*\cup I_2})\big)\EE
for all  $I_1^*\!\subsetneq\!I_1\!\subset\![N]$ with $I_1^*\!\in\!\cP^{=\ell}(N)$ and
$I_2\!\in\!\cP^{>\ell}(N)$.
By~\eref{weakregtoreg_e5b}, \eref{weakregtoreg_e5a}  with $I$ replaced 
by $I_1\!\cup\!I_2$, and \eref{weakregtoreg_e4b}  with $(I,I')$ 
replaced by $(I_1\!\cup\!I_2,I_1^*)$, $(I_1\!\cup\!I_2,I_2^*)$, and
$(I_1\!\cup\!I_2,I_1^*\!\cup\!I_2^*)$, 
\BE{weakregtoreg_e25a2}
W_{t;I_1^*;I_1} \cap W_{t;I_2^*;I_2} \subset
\Psi_{t;I_1^*\cup I_2^*}\big(\cN_{t;I_1^*\cup I_2^*;\prt}(2^{\ell}\ve_{I_1^*\cup I_2^*})\big)\EE
for all  $I_1^*\!\subsetneq\!I_1\!\subset\![N]$ and
$I_2^*\!\subsetneq\!I_2\!\subset\![N]$ with 
$I_1^*,I_2^*\!\in\!\cP^{=\ell}(N)$  and $I_1^*\!\neq\!I_2^*$.
By~\eref{weakregtoreg_e6},  \eref{weakregtoreg_e5a}  with $I$ replaced 
by $I_1\!\cup\!I_2$,  
\eref{weakregtoreg_e4b}  with $(I,I')$  replaced by $(I_1\!\cup\!I_2,I_1^*)$
and \hbox{$(I_1\!\cup\!I_2,I_1^*\!\cup\!I_2)$}, and~\eref{ConfRegulLoc_e1}
 with $(I^*,I)$ replaced by $(I_1\!\cup\!I_2,I_2)$, 
\BE{weakregtoreg_e25b}
\ov{W_{t;I_1^*;I_1}} \cap X_{I_2} \subset
\ov{\Psi_{t;I_1^*\cup I_2}\big(\cN_{t;I_1^*\cup I_2;\prt}(2^{\ell}\ve_{I_1^*\cup I_2})\big)}\EE
for all $I_1^*\!\subsetneq\!I_1\!\subset\![N]$ with $I_1^*\!\in\!\cP^{=\ell}(N)$
and $I_2\!\in\!\cP^*(N)\!-\!\cP^*(I_1^*)$.\\

\noindent
For each $I^*\!\in\!\cP^{=\ell}(N)$, let
\begin{gather*}
W_{t;I^*}=X_{\eset}
-\bigcup_{I^*\subsetneq I\subset [N]}\!\!\!\!\!\!\!
\ov{\Psi_{t;I}\big(\cN_{t;I;\prt}(2^{\ell}\ve_I)\big)}
-\bigcup_{I\in\cP^{>\ell}(N)}\!\!\!\!\!\!\!
\ov{\Psi_{t;I}\big(\cN_{t;I;\prt}(2^{\ell+1}\ve_I)\big)\!-\!W_{t;I^*;I^*\cup I}}
-\bigcup_{\begin{subarray}{c}I\subset[N]\\ I\not\subset I^*\end{subarray}}\!\!\!X_I,\\
X_{t;I^*}'=X_{I^*}-\bigcup_{I^*\subsetneq I\subset[N]}\!\!\!\!\!\!\!W_{t;I^*;I},\quad
\cX_{I^*}'=\bigcup_{t\in B}\{t\}\!\times\!X_{t;I^*}', \quad
\cW_{I^*}=\bigcup_{t\in B}\{t\}\!\times\!W_{t;I^*}.
\end{gather*}
Since $\Psi_{t;I_1}$ depends continuously on~$t$, 
$\cX_{I^*}'$ is a closed subset  of~$B\!\times\!X_{\eset}$ 
and $\cW_{I^*}$ is an open subset.
By~\eref{weakregtoreg_e25a1}, \eref{weakregtoreg_e25a2},  and the definition of~$W_{t;I^*}$,
\BE{weakregtoreg_e29a}
W_{t;I_1^*} \cap 
\Psi_{t;I_2}\big(\cN_{t;I_2;\prt}(2^{\ell}\ve_{I_2})\big)=\eset,
\quad W_{t;I_1^*} \cap W_{t;I_2^*;I_2} =\eset,
\quad W_{t;I_1^*}\cap X_{I_2}=\eset\EE
for all $I_1^*\!\in\!\cP^{=\ell}(N)$, 
$I_2\!\in\!\cP^{>\ell}(N)$ in the first case,  
\hbox{$I_2^*\!\subsetneq\!I_2\!\subset\![N]$} with
\hbox{$I_2^*\!\in\!\cP^{=\ell}(N)$} and $I_2^*\!\neq\!I_1^*$ in the second case,
and $I_2\!\in\!\cP(N)\!-\!\cP(I_1^*)$ in the third case.
By~\eref{weakregtoreg_e23}, $\cX_{I^*}'\!\subset\!\cW_{I^*}$.
Since the closed sets $\cX_{I^*}'$ are disjoint,
there exist open subsets 
\begin{gather}
\notag
\cW_{I^*}'\equiv \bigcup_{t\in B}\{t\}\!\times\!W_{t;I^*}' \subset B\!\times\!X_{\eset}
~~~\forall\,I^*\!\in\!\cP^{=\ell}(N)\qquad\hbox{s.t.}\\
\label{weakregtoreg_e8} 
\cW_{I_1^*}'\!\cap\! \cW_{I_2^*}'=\eset~~\forall\,
I_1^*,I_2^*\!\in\!\cP^{=\ell}(N),\,I_1^*\!\neq\!I_2^*, \quad
\cX_{I^*}'\!\subset\cW_{I^*}',~\ov{\cW_{I^*}'}\subset\cW_{I^*}
~~\forall\,I^*\!\in\!\cP^{=\ell}(N)\,.
\end{gather}

\vspace{.2in}

\noindent
For each $I^*\!\in\!\cP^{=\ell}(N)$, define
\begin{gather*}
\wt\cW_{I^*}\equiv \bigcup_{t\in B}\{t\}\!\times\!\wt{W}_{t;I^*}
=\cW_{I^*}'\cup 
\bigcup_{t\in B}\bigcup_{I^*\subsetneq I\subset[N]}\!\!\!\!\!\!\{t\}\!\times\!W_{t;I^*;I}, \\
\cN_{t;I^*;\prt}'=\Psi_{t;I^*}^{-1}\big(\wt{W}_{t;I^*}\big), \qquad
\cN_{\prt}'X_{I^*}\equiv \bigcup_{t\in B}\{t\}\!\times\!\cN_{t;I^*;\prt}'.
\end{gather*}
By~\eref{weakregtoreg_e5a} and~\eref{weakregtoreg_e4b},
\BE{weakregtoreg_e12}
\big\{v\!\in\!\cN_{\prt}X_{I^*}|_{\Psi_{t;I}(u)}\!:
u\!\in\!\cN_{t;I;\prt}(2^{\ell+1}\ve_I)\!\cap\!\cN_{I;I^*},\,
\rho_{t;I^*}(v)\!<\!2^{\ell}\ve_I(t,\pi_I(u))\big\}
\subset \cN_{t;I^*;\prt}'\EE
for all $t\!\in\!B$ and $I^*\!\subsetneq\!I\!\subset\![N]$.
By the last assumption in~\eref{weakregtoreg_e8},
the first statement in \eref{weakregtoreg_e29a}, and~\eref{weakregtoreg_e25a1},
\BE{weakregtoreg_e9b}
\wt{W}_{t;I_1^*}\cap
\Psi_{t;I_2}\big(\cN_{t;I_2;\prt}(2^{\ell}\ve_{I_2})\big)
\subset
\Psi_{t;I_1^*\cup I_2}\big(\cN_{t;I_1^*\cup I_2;\prt}(2^{\ell}\ve_{I_1^*\cup I_2})\big)
\EE
for all $I_1^*\!\in\!\cP^{=\ell}(N)$ and $I_2\!\in\!\cP^{>\ell}(N)$.
By the first and last assumptions in~\eref{weakregtoreg_e8},
the second statements in \eref{weakregtoreg_e29a}, and~\eref{weakregtoreg_e25a2},
\BE{weakregtoreg_e9a}
\wt{W}_{t;I_1^*}\cap \wt{W}_{t;I_2^*}\subset 
\Psi_{t;I_1^*\cup I_2^*}\big(\cN_{t;I_1^*\cup I_2^*;\prt}(2^{\ell}\ve_{I_1^*\cup I_2^*})\big)
\EE
for all $I_1^*,I_2^*\!\in\!\cP^{=\ell}(N)$ with $I_1^*\!\neq\!I_2^*$.
By the last assumption in~\eref{weakregtoreg_e8},
the last statement in~\eref{weakregtoreg_e29a}, and~\eref{weakregtoreg_e25b},
\BE{weakregtoreg_e9c}
\ov{\wt{W}_{t;I_1^*}} \cap X_{I_2} \subset
\ov{\Psi_{t;I_1^*\cup I_2}\big(\cN_{t;I_1^*\cup I_2;\prt}(2^{\ell}\ve_{I_1^*\cup I_2})\big)}\EE
for all $I_1^*\!\in\!\cP^{=\ell}(N)$ and $I_2\!\in\!\cP^*(N)\!-\!\cP^*(I_1^*)$.\\

\noindent
Since $\wt\cW_{I^*}$ is a neighborhood of $B\!\times\!X_{I^*}$ in $B\!\times\!X_{\eset}$, 
$\cN_{\prt}'X_{I^*}$
is a neighborhood of $B\!\times\!X_{I^*}$ in $B\!\times\!\cN_{\prt}X_{I^*}$.
Thus, there exists an open subset 
$$\cN' X_{I^*}\equiv \bigcup_{t\in B}\{t\}\!\times\!\cN_{t;I^*}'\subset B\!\times\!\cN X_{I^*}
\qquad\hbox{s.t.}\quad \cN_{t;I^*;\prt}'=\cN_{\prt}X_{I^*}\cap \cN_{t;I^*}'
~~\forall\,t\!\in\!B.$$
Choose $\ve_{I^*}'\!\in\!C^{\i}(B\!\times\!X_{I^*};\R^+)$ so that 
\BE{weakregtoreg_e15}
\ov{\cN X_{I^*}(2^{\ell}\ve_{I^*}')}\subset\cN^{\circ} X_{I^*}\cap \cN' X_{I^*}\,.\EE
Let
\begin{alignat*}{2}
X_{t;I^*;\prt}&=X_{I^*}\cap
\bigcup_{\begin{subarray}{c}I\in\cP^{=(\ell+1)}(N)\\ I^*\subset I\end{subarray}}
\hspace{-.3in}\Psi_{t;I}\big(\cN_{t;I;\prt}(2^{\ell+1}\ve_I)\big), &\qquad
\cX_{I^*;\prt}&= \bigcup_{t\in B}\{t\}\!\times\!X_{t;I^*;\prt},\\
X_{t;I^*}'&=X_{I^*}-
\bigcup_{\begin{subarray}{c}I\in\cP^{=(\ell+1)}(N)\\ I^*\subset I\end{subarray}}
\hspace{-.3in}\ov{\Psi_{t;I}\big(\cN_{t;I;\prt}(2^{\ell}\ve_I)\big)}, &\qquad
\cX_{I^*}'&= \bigcup_{t\in B}\{t\}\!\times\!X_{t;I^*}'.
\end{alignat*}
Since $\Psi_{t;I}$ depends continuously on~$t$, $\cX_{I^*}'$ is an open subset 
of~$B\!\times\!X_{I^*}$.
Let $\{\eta_{I^*;\prt},\eta_{I^*}'\}$ be a partition of unity on~$B\!\times\!X_{I^*}$
subordinate to the open cover $\{\cX_{I^*;\prt},\cX_{I^*}'\}$ of $B\!\times\!X_{I^*}$.\\

\noindent
Define
\begin{gather*}
\ve_{I^*;\prt}\!:\cX_{I^*;\prt}\lra\R^+ \qquad\hbox{by}\\
\ve_{I^*;\prt}\big(t,\Psi_{t;I}(u)\big)=\ve_I\big(t,\pi_I(u)\big)~~
\forall\,u\!\in\!\cN_{t;I;\prt}(2^{\ell+1}\ve_I)\!\cap\!\cN_{I;I^*},~
I\!\in\!\cP^{=(\ell+1)}(N),\,I^*\!\subset\!I.
\end{gather*}
By~\eref{weakregtoreg_e5b}, \eref{weakregtoreg_e5a}, \eref{weakregtoreg_e4b},
and~\eref{weakregtoreg_e5c}, these definitions agree on the overlaps.
Let 
$$\ve_{I^*}=\eta_{I^*;\prt}\ve_{I^*;\prt}+\eta_{I^*}'\ve_{I^*}'\!:
B\!\times\!X_{I^*}\lra\R^+.$$

\vspace{.1in}

\noindent
We next observe that 
\BE{weakregtoreg_e17}
\ov{\cN_{t;I^*}(2^{|\ell|}\ve_{I^*})}\subset\cN_{t;I^*}^{\circ},~~
\cN_{t;I^*;\prt}(2^{\ell}\ve_{I^*})\subset\cN'_{t;I^*}
\quad\forall\,t\!\in\!B,\,I^*\!\in\!\cP^{=\ell}(N)\,.\EE
By~\eref{weakregtoreg_e15}, this is the case for the fibers over $\cX_{I^*}'\!-\!\cX_{I^*;\prt}$.
The second inclusion in~\eref{weakregtoreg_e17} for the fibers over 
$\cX_{I^*;\prt}\!-\!\cX_{I^*}'$ is a special case of~\eref{weakregtoreg_e12}.
The first inclusion in~\eref{weakregtoreg_e17} for these fibers follows 
from~\eref{weakregtoreg_e5a} and~\eref{weakregtoreg_e4c} with $I'\!=\!I^*$.
If  $(t,x)\!\in\!\cX_{I^*}'\!\cap\!\cX_{I^*;\prt}$, then 
$$\ve_{I^*}(t,x)\le \ve_{I^*}'(t,x) \qquad\hbox{or}\qquad
\ve_{I^*}(t,x)\le \ve_{I^*;\prt}(t,x).$$
Either of these cases implies~\eref{weakregtoreg_e17}.\\

\noindent
By the first inclusion in~\eref{weakregtoreg_e17}, 
$\ve_{I^*}$ satisfies~\eref{weakregtoreg_e5a} with $I\!=\!I^*$.
Since $\ve_{I^*}\!=\!\ve_{I^*;\prt}$ on $\cX_{I^*;\prt}\!-\!\cX_{I^*}'$,
$\ve_{I^*}$ satisfies~\eref{weakregtoreg_e5c} with $I'\!=\!I^*$
and $|I|\!=\!\ell\!+\!1$ and thus for all $I\!\supset\!I^*$.
By the second inclusion in~\eref{weakregtoreg_e17}, 
\BE{weakregtoreg_e30}\Psi_{t;I^*}\big(\cN_{t;I^*;\prt}(2^{\ell}\ve_{I^*})\big)
\subset \wt{W}_{t;I^*}\,.\EE
By~\eref{weakregtoreg_e9c}, $\ve_{I^*}$ thus satisfies \eref{weakregtoreg_e6}
with $I_1\!=\!I^*$ and $2^{\ell+1}$ replaced by~$2^{\ell}$.
By~\eref{weakregtoreg_e30} and~\eref{weakregtoreg_e9b}, 
$\ve_{I^*}$ satisfies \eref{weakregtoreg_e5b} with $I_1\!=\!I^*$, $|I_2|\!>\!\ell$,
and  $2^{\ell+1}$ replaced by~$2^{\ell}$.
By~\eref{weakregtoreg_e30} and~\eref{weakregtoreg_e9a}, 
\eref{weakregtoreg_e5b} with $2^{\ell+1}$ replaced by~$2^{\ell}$
 is satisfied whenever $|I_1|,|I_2|\!=\!\ell$.
By the downward induction on~$|I|$, 
we thus obtain functions 
$\ve_I\!\in\!C^{\i}(B\!\times\!X_I;\R^+)$ satisfying 
\eref{weakregtoreg_e5a} and \eref{weakregtoreg_e5c},
as well as~\eref{weakregtoreg_e5b} and~\eref{weakregtoreg_e6}.
\end{proof}

\begin{crl}\label{weakregtoreg_crl}
Let $\X$, $B$, and $(\om_{t;i})_{t\in B,i\in[N]}$ be as in Theorem~\ref{SCC_thm}
and $(\fR_t^{(1)})_{t\in B}$ and  $(\fR_t^{(2)})_{t\in B}$ be
$(\om_{t;i})_{t\in B,i\in[N]}$-families of regularizations for~$\X$ 
that are equivalent as families of weak regularizations for~$\X$ over~$X_{\eset}$,
i.e.
\BE{weakregtoregcrl_e1}
(\fR_t^{(1)})_{t\in B}\cong_{X_{\eset}} (\fR_t^{(2)})_{t\in B}.\EE
Then they are equivalent as families of regularizations for~$\X$ as in~\eref{fRequidfn_e}.
\end{crl}

\begin{proof}
Let $(\fR_t^{(1)})_{t\in B}$ and $(\fR_t^{(2)})_{t\in B}$ be as in~\eref{R1R2_e}.
By~\eref{weakregtoregcrl_e1}, there exists 
an $(\om_{t;i})_{t\in B,i\in[N]}$-family $(\fR_t)_{t\in B}$ 
of weak regularizations for~$\X$ over~$X_{\eset}$ which satisfies 
the conditions below~\eref{fRtdfn_e}. 
By Lemma~\ref{weakregtoreg_lmm}, it can be cut down to 
an $(\om_{t;i})_{t\in B,i\in[N]}$-family $(\fR_t')_{t\in B}$  
of regularizations for~$\X$.
Since the latter still satisfies the conditions below~\eref{fRtdfn_e},
we obtain~\eref{fRequidfn_e}.
\end{proof}

\begin{rmk}\label{weakregtoreg_rmk}
Lemmas~\ref{regulcomb_lmm} and~\ref{weakregtoreg_lmm},
Corollaries~\ref{regulcomb_crl} and~\ref{weakregtoreg_crl},
and their proofs apply in the smooth category as well 
(as opposed to the symplectic category).
For a \sf{smooth regularization}, we need only Riemannian metrics~$\rho_{t;I;i}$
on the real rank~2 vector bundles $\cN_{X_{I-i}}X_I$ which are preserved
by the differentials $\fD\Psi_{t;I;I'}$.
\end{rmk}

\vspace{.2in}

\noindent
{\it Simons Center for Geometry and Physics, Stony Brook University, Stony Brook, NY 11794\\
mtehrani@scgp.stonybrook.edu}\\

\noindent
{\it Department of Mathematics, Stony Brook University, Stony Brook, NY 11794\\
markmclean@math.stonybrook.edu, azinger@math.stonybrook.edu}\\

\end{document}